\documentclass[reqno,10pt]{amsart}
\usepackage{amsmath}
\usepackage{amssymb}
\usepackage{amsfonts}
\usepackage{graphicx}
\usepackage[abs]{overpic}
\usepackage{caption}
\usepackage{subcaption}
\usepackage{wrapfig}
\usepackage{float}
\usepackage{graphicx}
\usepackage{physics}
\usepackage{esint} 
\usepackage{braket} 

\newcommand{\HH}{\mathcal{H}} 

\usepackage{tikz}
\usetikzlibrary{decorations.pathreplacing,calc}
\def\centerarc[#1](#2)(#3:#4:#5)
{ \draw[#1] ($(#2)+({#5*cos(#3)},{#5*sin(#3)})$) arc (#3:#4:#5); }

\definecolor{blue_links}{RGB}{13,0,180} 

\usepackage{hyperref}
\hypersetup{
    colorlinks=true, 
    linktoc=all,     
    linkcolor=blue_links,  
    citecolor=blue_links,
    urlcolor=blue_links,
}

\usepackage{mathtools}

\usepackage{xcolor}

\newtheorem{theorem}{Theorem}[section]
\newtheorem{lemma}[theorem]{Lemma}
\newtheorem{proposition}[theorem]{Proposition}
\newtheorem{corollary}[theorem]{Corollary}
\newtheorem{definition}[theorem]{Definition}

\newtheorem{remark}[theorem]{Remark}

\newtheorem*{theorem*}{Theorem}

\newcommand{\N}{\mathbb{N}}

\newcommand{\R}{\mathbb{R}}

\newcommand{\EEE}{\color{black}}

\newcommand{\RRR}{\color{red}}

\def\eps{\varepsilon}

\def\dist{\operatorname{dist}}
\def\Xint#1{\mathchoice
    {\XXint\displaystyle\textstyle{#1}}%
    {\XXint\textstyle\scriptstyle{#1}}%
    {\XXint\scriptstyle\scriptscriptstyle{#1}}%
    {\XXint\scriptscriptstyle\scriptscriptstyle{#1}}%
\!\int}
\def\XXint#1#2#3{{\setbox0=\hbox{$#1{#2#3}{\int}$}
\vcenter{\hbox{$#2#3$}}\kern-.5\wd0}}

\def\dashint{\Xint-}

\usepackage{color}

\numberwithin{equation}{section}

\begin{document} 

\title[Epsilon-Regularity for Griffith almost-minimizers in the plane]
{An Epsilon-Regularity result for Griffith almost-minimizers \\ in the plane}
\author[M. Friedrich]{Manuel Friedrich} 
\address[Manuel Friedrich]{Department of Mathematics, Friedrich-Alexander Universit\"at Erlangen-N\"urnberg. Cauerstr.~11,
    D-91058 Erlangen, Germany, \& Mathematics M\"{u}nster,  
University of M\"{u}nster, Einsteinstr.~62, D-48149 M\"{u}nster, Germany}
\email{manuel.friedrich@fau.de}
\author[C. Labourie]{Camille Labourie}
\address[Camille Labourie]{Department of Mathematics, Friedrich-Alexander Universit\"at Erlangen-N\"urnberg. Cauerstr.~11,
D-91058 Erlangen, Germany}
\email{camille.labourie@fau.de}
\author[K. Stinson] {Kerrek Stinson} 
\address[Kerrek Stinson]{Hausdorff Center for Mathematics, University of Bonn, Endenicher Allee 62, 53115 Bonn,
Germany}
\email{kstinson@math.utah.edu}

\begin{abstract} 
    We present regularity results for the crack set of a minimizer for the Griffith fracture energy, arising in the variational modeling of brittle materials. In the planar setting, we prove an epsilon-regularity theorem showing that the crack is locally a $C^{1,1/2}$ curve outside of a singular set of zero Hausdorff measure. The main novelty is that, in contrast to previous results, no topological constraints on the crack are required. The results also apply to almost-minimizers.
\end{abstract}
\maketitle
{\small
\tableofcontents
}

\section{Introduction}
The theory of brittle solids was pioneered by {\sc Griffith} \cite{griffith} in the 1920's and is based on the idea that   the energy required to form a crack is proportional to the size of the crack. Casting this principle into a variational framework, {\sc Francfort and Marigo}   \cite{Francfort-Marigo:1998} (see also with {\sc Bourdin} in \cite{Bourdin-Francfort-Marigo:2008}) initiated the study of  brittle fracture via minimization of the so-called  \emph{Griffith energy}. For this, in a planar  setting,  we let $u \colon \Omega \setminus K \to \R^2$ be the \emph{displacement} for a \emph{reference configuration}   $\Omega \subset \R^2$ and $K \subset \Omega$ be the closed \emph{crack set}.  The  energy is given by
\begin{equation}\label{eq:energy}
    \int_{\Omega\setminus K} \mathbb{C}e(u)\colon e(u) \, \dd{x} + \mathcal{H}^1(K),
\end{equation}
where $e(u) : = (\nabla u + \nabla u^T)/2$ is the symmetrized gradient and $\mathbb{C}$ represents a fourth-order   tensor of {elasticity} constants. In this paper, we are   interested in the local regularity of (almost-)minimizers for the Griffith energy under Dirichlet boundary conditions.
This model represents the basic building block for many models including those in cohesive fracture, the nonlinear setting, and in cases of material inhomogeneity. Despite its incorporation into many complex models, the fine properties of minimizers to (\ref{eq:energy}) are only recently coming to light.

Understanding brittle fracture via the Griffith fracture energy has been a topic of intense research in the past two decades. Initially many results were confined to the {simplified} setting of {antiplanar shear} where the displacement $u$ is   merely out-of-plane, and can thus be identified with a scalar function. In this setting, the energy (\ref{eq:energy}) reduces to the classical \emph{Mumford--Shah functional} arising in image segmentation \cite{Ambrosio-Fusco-Pallara:2000}. At the static level, the mathematical behavior of minimizers was relatively well understood with the results of {\sc Ambrosio} \cite{Ambrosio:90}, {\sc De Giorgi et al}.\  \cite{deGiorgiCarrieroLeaci}, and {\sc Ambrosio et al}.\  \cite{Ambrosio-Fusco-Pallara:2000}, which say that minimizers in $SBV$ (special functions of bounded variation) exist, the crack of minimizers is closed and  coincides with a smooth {submanifold} at most points. However, the initiation of cracks also gives rise to interesting evolutionary  problems. {\sc Francfort and Larsen}  \cite{Francfort-Larsen:2003} proved existence of (weak) solutions to  quasistatic brittle fracture. Subsequently, {\sc Babadjian and Giacomini}  \cite{BJ} showed that closedness of the crack set {also holds} in the evolutionary setting, giving rise to so-called \emph{strong solutions} of quasistatic fracture.

In the context of linearized elasticity, as in (\ref{eq:energy}), new techniques were necessary. Building on the  function space  $BD$ (bounded deformation), see e.g.\  \cite{ACD},   {\sc Dal Maso} \cite{DalMaso:13} introduced the space $GSBD$, which represents the right class of   displacements  for which the Griffith fracture energy  is  finite. However, existence of minimizers with Dirichlet boundary conditions was only proven under further restrictions, such as $L^\infty$-bounds on the functions or fidelity terms.  The first author, in a series of papers \cite{Friedrich:15-5,Friedrich:15-4,Solombrino}, introduced  a \emph{piecewise  Korn inequality} in dimension two to prove unconditional existence of minimizers of (\ref{eq:energy}) in the space $GSBD$. In particular, these techniques were used by the first author and {\sc Solombrino} \cite{Solombrino} to prove existence of solutions for the quastistatic evolution problem in linearized elasticity. Though this approach via the piecewise  Korn inequality is powerful,  the drawback is that for now it is  restricted to dimension two. Using an alternative approach based only on a \emph{Korn inequality for functions with small jump set}  \cite{CCS, Chambolle-Conti-Francfort:2014, Conti-Iurlano:15.2, Friedrich:15-3},  {\sc Chambolle and Crismale} \cite{Crismale} proved lower semicontinuity and coercivity of the functional in (\ref{eq:energy}), thereby providing an existence result for the Dirichlet problem in arbitrary dimension. Turning towards regularity of minimizers, {\sc Conti et al}.\  \cite{Conti-Focardi-Iurlano:19} showed that the crack is  (locally) closed   in dimension two. 
Like the approach of {\sc De Giorgi et al}.\  \cite{deGiorgiCarrieroLeaci}, this relies on showing \emph{Ahlfors-regularity} of the crack, a uniform density estimate of the form
\begin{equation}\label{eqn:AhlfTemp}
    \frac{1}{C_{\rm Ahlf}} r \leq \mathcal{H}^1(K \cap B(x,r)) \leq C_{\rm Ahlf} r 
\end{equation}  
for some constant $C_{\rm Ahlf} \ge 1  $ and arbitrary open balls $B(x,r)\subset \Omega$ with center $x \in K$ and radius $r>0$.  This result  was then extended to arbitrary dimension by {\sc Chambolle et al}.\  \cite{Iu3},  and   later {\sc Chambolle and Crismale} \cite{CrismaleCalcVar} addressed the case of Dirichlet boundary conditions  by  uniting the aforementioned local results and techniques showing that the Ahlfors density estimate is stable up to the boundary.

In this paper, we prove an \emph{epsilon-regularity} statement for minimizers in the plane, which provides a local criterion for  the crack $K$ to be  a $C^{1, 1/2  }$ curve. This is the first regularity result for the Griffith energy which does not rely on a topological constraint for the crack (discussed below). With essentially no extra work, the epsilon-condition holds for $\HH^1$-a.e.\ point on the crack $K$.
Below we discuss this result and related literature at a formal level, leaving rigor to the following sections. In particular, within the introduction, we restrict the discussion to minimizers. Let us mention, however, that our results also apply to almost-minimizers which allows us to cover a variety of models similar to \eqref{eq:energy}, see Subsection \ref{sec: basic not} for examples and details. 

\noindent \textbf{An epsilon-regularity result.} Going beyond measure-theoretic notions of regularity such as \eqref{eqn:AhlfTemp}, to understand higher-order regularity of rectifiable sets, or even varifolds, a typical hope is to prove an epsilon-regularity result. Such results  state that, if certain local quantities are below a fixed threshold $\varepsilon$, then the set is actually a smooth surface in this small region.  A couple of the most celebrated results in this spirit are {\sc Allard}'s regularity for varifolds with integrable curvature \cite{Allard1972} and the approach introduced  by {\sc Ambrosio et al}.\   \cite{AFP2, AFP3,  Ambrosio-Fusco-Pallara:2000}  for the study of the Mumford--Shah functional. Common to these approaches is the use of the Euler--Lagrange equation which is {instrumental} for recovering decay estimates of the flatness/excess  of the  surface. \EEE However, if one computes the Euler--Lagrange equation for the Griffith fracture energy, {non-symmetric gradient terms appear thereby} preventing a clear path to a tilt-estimate that would be able to control variations of the tangent spaces on $K$ in terms of the  {elastic energy} and  the \EEE  {flatness}.

In contrast to the above PDE-based approaches, {\sc Bonnet}  \cite{Bonnet},   {\sc David} \cite{DavidC1,David}, and {\sc Lemenant} \cite{Lemenant} developed a  powerful variational strategy   for proving regularity of minimizers of the Mumford--Shah functional. {They also  obtained a local description of the set at certain singularites such as  triple junctions in $\R^2$ and the $\mathbb{Y}$, $\mathbb{T}$ cones in $\R^3$. This approach avoids the use of the tilt, which in fact  has no analogue to study such  singularities.} {\sc Babadjian et al}.\ \cite{BIL} showed in the planar setting that  this strategy can even be transferred to  the vectorial Griffith energy, so long as a topological constraint on the crack is included. Relying on a technical stopping time argument, the second author and {\sc Lemenant} showed in   \cite{CL}  that epsilon-regularity holds in higher dimensions under a related topological separating condition. Both arguments crucially rely  on the separation assumption to introduce variational competitors which in turn lead to  decay estimates.

Our primary contribution is an epsilon-regularity statement that does not need any assumptions on  the topology of the crack, and hence our result applies to generic $GSBD$  minimizers of the Griffith energy (\ref{eq:energy}). We remark that, in contrast to the approach in \cite{BIL}  which  depends on the Airy function, our strategy  allows for a non-isotropic elastic energy, as determined by the tensor $\mathbb{C}$.  Moreover, the H\"older-regularity exponent is improved from some value in $(0,1/2)$ to $1/2$.   We formally summarize the result below, but refer to Theorem \ref{th: eps_reg} and Remark \ref{rmk:alpha12} for  details.

\begin{theorem*}
    Let $\Omega\subset \R^2$ be a bounded Lipschitz domain and  $(u,K)$  be a minimizer of the Griffith energy \eqref{eq:energy}.  Then,  there are constants $\eps>0$ and {$\gamma>0$} and two locally defined quantities called the \emph{flatness} $\beta_K(x,r)$  and the  \emph{jump}  $J_u(x,r)$   such that, if on a ball  $B(x,r)\subset \Omega$ we have   
    \begin{equation}\label{eqn:epsTemp}
        \beta_K(x,r) + J_u^{-1}(x,r) \leq \eps,
    \end{equation}
    then 
    $$\text{$K$ is given by the graph of a $C^{1,1/2}$ function in {$B(x,\gamma r)$}.}$$
\end{theorem*}

We refer to \eqref{eq_beta} and  \eqref{eq: optimal values}--\eqref{def:normalizedJump} for the definition of $\beta_K$ and $J_u$, respectively.  
Using that $K$ is a rectifiable set, which has a measure-theoretic tangent space at most points, and, similarly, properties of $BD$-functions, it is relatively straightforward to show for $\HH^1$-a.e.\ $x\in K$ that there is $r_x>0$ such that the $\eps$-condition of \eqref{eqn:epsTemp} is satisfied. 
{We refer  to \EEE  the set of points in $K$ such that there is no radius $r_x$ with $K\cap B(x,r_x)$ given by a $C^{1,1/2}$-graph as the \emph{singular set}, and denote this by $K^*$. The above result shows that $\mathcal{H}^1(K^*) = 0.$}
{In light of the Mumford--Shah conjecture, we expect that $K^*$ consists of isolated points} for which there is branching, creating triple junctions, or crack tips, where an arc terminates.
As a first step towards understanding the behavior of the singular set, in \cite{FLS3}, we use epsilon-regularity as a tool to estimate the dimension of the set $K^*$.

\noindent\textbf{Strategy for epsilon-regularity.}
 While there is much at play in the proof of our regularity theorem, a key point is understanding how \emph{separation} fails.
A novelty of our paper is that we \emph{quantify} the lack of topological separation for cracks in the setting of linearized elasticity.   To gain  this fine information we will use the piecewise  Korn inequality \cite{Friedrich:15-4}, see Subsection \ref{korn-sec}.  As  this  result has only be shown in the plane, our statements are restricted to 2D.  On the plus side, though,  this is also the only reason our results are dimension restricted. Let us emphasize that estimates provided  by  Korn inequalities for functions  with small jump set \cite{CCS, Chambolle-Conti-Francfort:2014, Conti-Iurlano:15.2, Friedrich:15-3}  are not adequate as  they are designed to  `clean' up small pieces of the crack, whereas in the present setting   the fine behavior of a big crack close to an interface needs to be controlled.   

\begin{figure}
    \begin{tikzpicture}
        \draw[very thick](0,0) circle (3.0);
        \draw [blue,very thick,densely dashed] plot [smooth] coordinates {(0.8,-0.15) (1.12,0.05)  (1.5,0.35)};
        \draw [red,very thick] plot [smooth] coordinates {(-3.0,0.1)  (-1.5,-0.5)  (0,0) (0.8,-0.15) (1,-0.15)};
        \draw [red,very thick] plot [smooth] coordinates {(1.3,0.35)  (1.5,0.35) (2.5,-0.25)  (3,0.05)};
        \node[above] at (-1.5, -0.5) {$K$};
        \draw[dotted] (-3,0.5) -- (3,0.5);
        \draw[dotted] (-3,-0.5) -- (3,-0.5);
        \node[above] at (-1.5, -0.5) {$K$};
        \draw[->] (3.4,-0.8) -- (1.2,0.045);
        \node[right] at (3.4,-0.8) {$\partial^*\{u>t\}\setminus K$ with $t\approx 1/2$};
        \node[] at (0.0, 1.8) {$u \approx 1$};
        \node[] at (0.0, -1.8) {$u \approx 0$};
    \end{tikzpicture}
    \caption{The above figure shows how a crack can be made separating by taking the super-level set of the function that is discontinuous across the crack $K$.}
    \label{fig:basiccoarea}
\end{figure}
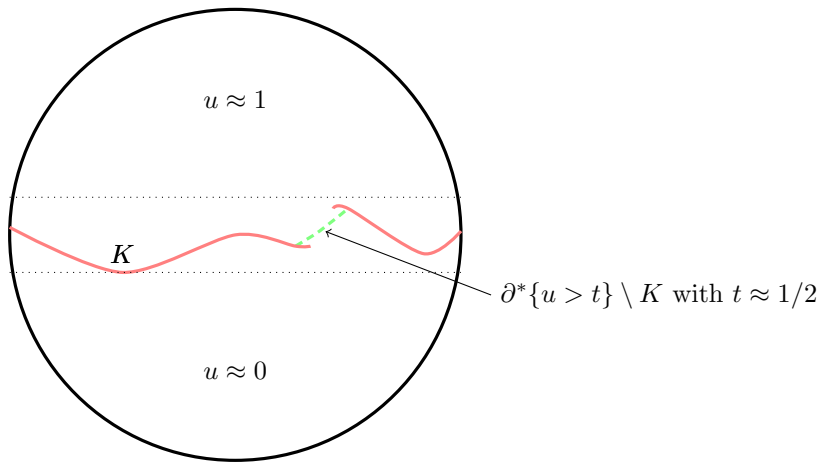

To understand how a Korn-type inequality is used to control separation, we briefly explain what is done in the scalar setting of the Mumford--Shah functional. In this setting, for a crack (or segment) which is nearly flat in some ball and has a length close to the diameter of that ball, one would like to understand how close the crack is to separating the upper and lower regions. Supposing $u \approx 1$ above the crack and $u\approx 0$ below the crack in the ball, one can apply the coarea formula (in $\Omega\setminus K$), to select a level set $\{x: u(x)>t\}$ with $t$ close to $1/2$ such that the reduced boundary $\partial^*\{x: u(x)>t\}$ effectively fills in the gaps in the crack and has small length in the sense that $\mathcal{H}^1(\partial^*\{x: u(x)>t\} \setminus K)$ is controlled nicely in terms of the elastic energy and the jump in trace (here $0$ jumps to $1$). We refer to  Figure \ref{fig:basiccoarea} for the set-up. However, there is no known analogue of the coarea formula for $BD$ functions with only their symmetrized gradient controlled.

This is where the piecewise  Korn inequality comes into play:  though it does not provide as fine of information as the coarea formula for $BV$ functions, it introduces a Caccioppoli partition with controllable perimeter such that on each element of the partition one has access to the Korn inequality bounding the full gradient (up to a skew-affine {map})   in terms of the  symmetrized gradient. This is analogous  to a piecewise inequality in the $BV$-setting {that is} a direct consequence of the  coarea formula, namely that for {a $BV$-function} there {is a  Caccioppoli partition} consisting of level sets such that on each element of the partition one has a Poincar\'e-type estimate, see \cite[Theorem 2.3]{Friedrich:15-4}.  Understanding the scale on which topological separation fails then boils down to analyzing  how the boundary of elements in the partition  interacts \EEE with the crack.

Once we have control on the size of the holes in the crack, we use an approach {motivated by} those for epsilon-regularity in the case of topological separation. In particular, we will make heavy use of the technical extension theorem introduced in \cite{CL} for separating cracks (see Subsection  \ref{subset: ext} for details).
With this, we are able to obtain appropriate decay estimates for the flatness and the  elastic energy, and show that the crack on small scales is actually given by an almost-minimal surface. We refer to Subsection \ref{sec: outline} where we provide an extended outline of the strategy including the main auxiliary steps. 


It will turn out that, while the elastic energy plays an essential role, we will be able to recast epsilon-regularity in terms of a quantity which is easier to handle:  the jump $J_u$.  In the case of the Mumford--Shah functional, the jump is the difference in average values    of the function above and below the crack, with the geometry as in Figure \ref{fig:basiccoarea}. As the Griffith energy is invariant under skew-affine maps, our notion of jump must capture the difference in average affine maps   away from the crack.  We refer to  \eqref{eq: optimal values}--\eqref{def:normalizedJump} below for the precise definition which is   tailor-made to account for the additional `flexibility' that allows infinitesimal rotations  without changing the energy. \EEE

\noindent \textbf{Organization of the paper.}
We briefly outline the organization of the paper. In Section \ref{sec:prelim}, we introduce the necessary definitions relevant to minimizers of the Griffith energy and state our main results. In Section \ref{sec: prelim}, we recall mathematical tools that will be used throughout the paper, such as an extension lemma,  sets of finite perimeter, and Korn inequalities. The proof of the epsilon-regularity theorem is contained in Section \ref{sec: eps_reg}. In particular, in Subsection \ref{sec: outline} we present an extended outline of the strategy.  

\section{Basic notions and main results}\label{sec:prelim}
\subsection{Basic  notions\EEE}\label{sec: basic not}
{We begin by carefully introducing a variety of definitions that will be required to state our main results and are used repeatedly within the paper.}

{Throughout}, we let $\Omega \subset \R^2$ be {an open set}.  We denote the open ball centered at $x\in \Omega$ with radius $r>0$ by $B(x,r)$. {The} two-dimensional Lebesgue and the one-dimensional Hausdorff measure are {denoted} by $\mathcal{L}^2$ and $\mathcal{H}^1$, respectively.  
The sets of symmetric and skew symmetric matrices are  denoted by $\R^{2 \times 2}_{\rm sym}$ and $ \R^{2 \times 2}_{\rm skew}\EEE$.  {The unit circle in $\R^2$ is given by $\mathbb{S}^1$.} We write $U \subset \subset \Omega$ if $\overline{U} \subset \Omega$. By $\chi_U$ we denote the characterisic function of $U$. The {Euclidean} distance of a point $y \in \R^2$ from a set $U$ is denoted by ${\rm dist}(y,U)$. We often use the notation $\pm$ as a placeholder for $+$ and $-$. {The average} integral over a set $U$ {is} denoted by $\dashint_U$. An affine map is called  an \textit{(infinitesimal) rigid motion} if its gradient is a skew-symmetric matrix, and is usually written as $a(x) = A\,x + b$ for $A \in \R^{2 \times 2}_{\rm skew}$ and $b \in \R^2$.  Given an open set ${U} \subset \R^2$, we also define $LD({U})$ as
\begin{equation*}
    LD({U}) := \Big\{ \text{$u \in W^{1,2}_{\mathrm{loc}}(U;  \R^2 \EEE )$ such that $\int_{{U}} \abs{e(u)}^2 \dd{x} < +\infty$} \Big\}.
\end{equation*}
Here, $e(u) := (\nabla u + \nabla u^T)/2$ denotes the \emph{symmetrized gradient} of $u$. In a similar fashion, we define $LD_{\rm loc}({U})$.

We  proceed with some basic notions in order to study almost-minimizers of the Griffith energy.

\noindent{\bf Admissible  pairs.}
We define an \emph{admissible pair} as a pair $(u,K)$ such that $K$ is a relatively closed subset of $\Omega$ and $u \in W^{1,2}_{\mathrm{loc}}(\Omega \setminus K;\R^2)$.

\noindent{\bf {Elasticity tensor and} Griffith energy.}
For {the entire paper,} we fix a linear  map $\mathbb{C} : \R^{2 \times 2} \to \R^{2 \times 2}_{\rm sym}$ such that
\begin{equation*}
    \mathbb{C}(\xi - \xi^T) = 0 \quad \text{and} \quad \mathbb{C} \xi : \xi \geq  c_0  \abs{\xi + \xi^T}^2 \quad \text{ for all } \xi \in \R^{2 \times 2}
\end{equation*}
for some $c_0 > 0$.
Given an admissible pair $(u,K)$ in $\Omega$ and a ball $B \subset \Omega$, we define the \emph{Griffith energy} of $(u,K)$ in $B$ as
\begin{align}\label{eq: main energy}
    \int_{B \setminus K} \mathbb{C} e(u) \colon e(u) \, {\rm d}x + \mathcal{H}^1(K \cap B). 
\end{align}

\noindent{\bf Competitors.}
Let $(u,K)$ be an admissible pair and let $x \in \Omega$ and $r > 0$ {be} such that ${B(x,r)} \subset \subset \Omega$. We say that $(v,L)$ is a \emph{competitor} of $(u,K)$ in $B(x,r)$ if  $(v,L)$ is  an admissible pair  such that
\begin{equation}\label{eq: competttt}
    L \setminus B(x,r) = K \setminus B(x,r) \quad \text{and} \quad v = u \  \ \text{a.e. in} \ \Omega \setminus \big(K \cup B(x,r)\big).
\end{equation}

\noindent{\bf Local minimizers and almost-minimizers.}
We call a \emph{gauge} a non-decreasing function $h: (0,+\infty) \to [0,+\infty]$ such that $\lim_{t \to 0^+} h(t) = 0$.
Most of our results hold  for any gauge but we will use the explicit  form $h(t) = h(1) t^{\alpha}$, where $\alpha \in (0,1)$, for the final proof  of the epsilon-regularity result. \EEE  Moreover, we say that a pair $(u,K)$ is \emph{coral} provided that for all $x \in K$ and for all $r > 0$,
\begin{equation}\nonumber
    \HH^{1}(K \cap B(x,r)) > 0.
\end{equation}
We say that an admissible, coral pair $(u,K)$ with locally finite energy is a \emph{Griffith local almost-minimizer with gauge $h$} in $\Omega$
if, given any  $x \in \Omega$ and  $r > 0$ with ${B(x,r)} \subset \subset \Omega$, all competitors $(v,L)$ of $(u,K)$ in $B(x,r)$ satisfy 
\begin{multline}
    \int_{B(x,r) \setminus K} \mathbb{C} e(u)\colon e(u)    \dd{x} + \HH^{1}\big(K \cap B(x,r)\big) \le  \int_{B(x,r) \setminus L} \mathbb{C} e(v)\colon e(v)    \dd{x} + \HH^{1}\big(L \cap B(x,r)\big) + h(r) r. \label{amin}
\end{multline}
In the case $h = 0$, we say that $(u,K)$ is a \emph{Griffith local minimizer}. In the following, we frequently omit the words `local' and `gauge', and  refer only to Griffith almost-minimizers for simplicity. 

We remark that the coral assumption is used simply  to choose a good representative for the crack $K$, analogous to choosing a good representative of a Sobolev function. In particular, it ensures that  (\ref{eqn:AhlforsReg}) below holds for \textit{all} points $x\in K$.

\noindent{\bf Ahlfors regularity.}
By the results in \cite{FLS2}, if $(u,K)$ is an almost-minimizer of the Griffith energy with any gauge $h$, then $K$ is (locally) Ahlfors-regular. We remark that  this is proven for Griffith minimizers in \cite{Iu3,  CrismaleCalcVar, Conti-Focardi-Iurlano:19}.  Precisely, {\cite{FLS2} shows} that there exist  constants $C_{\rm Ahlf} \geq 1$ and $\varepsilon_{\rm Ahlf}  \in  (0,1) \EEE $ (depending only on $\mathbb{C}$) such that for all $x \in K$, $r > 0$ with $B(x,r) \subset \Omega$ and $h(r) \leq \varepsilon_{\rm Ahlf}$, we have
\begin{equation}\label{eqn:AhlforsReg}
    \frac{1}{C_{\rm Ahlf}}r \leq \mathcal{H}^1\big(K\cap B(x,r)\big)\leq {C_{\rm Ahlf}}r.
\end{equation}
As a direct consequence {of their definition}, we mention that almost-minimizers satisfy
\begin{equation}\label{eqn:AhlforsReg2}
    \int_{B(x,r) \setminus K}  |e(u)|^2 \, \dd x     \leq {\bar{C}_{\rm Ahlf}}r  
\end{equation}
for all $x\in K$ and $r>0$  with  $B(x,r) \subset \Omega$ and $h(r) \le \varepsilon_{\rm Ahlf}$,  where $\bar{C}_{\rm Ahlf}$ depends only on $\mathbb{C}$. Indeed, this follows from \eqref{amin} by comparing $(u,K)$ with $v = u \chi_{\Omega \setminus B(x,\rho)}$  and $L = (K \setminus B(x,\rho)) \cup \partial B(x,\rho)$ {for $\rho < r$, and then letting $\rho \to r$}. (If $B(x,r) \subset \subset \Omega$, one can directly take $\rho = r$.) 

\noindent{\bf Examples.}  
Minimizers of the Griffith energy with a prescribed Dirichlet boundary condition are natural examples of Griffith local minimizers.
Their existence was established in \cite{CrismaleCalcVar} for Lipschitz Dirichlet data. The inclusion of a gauge in our framework also allows us to cover a wide range of functionals: for example, our regularity results apply to energies with more general bulk densities of the form
\begin{equation}\nonumber
    \int_{\Omega \setminus K} f(x,e(u)) \dd{x} + \HH^{1}(K),
\end{equation}
where $f \colon \Omega \times \R^{2 \times 2}_{\mathrm{sym}} \to  [0,\infty)$ is a Borel function satisfying
\begin{equation}\nonumber
    \abs{f(x,\xi) - \mathbb{C} \xi:\xi} \leq C \left(1 + \abs{\xi}^{q}\right) \quad \text{ for all } x \in \Omega \text{ and } \xi \in \R^{2\times 2}_{\mathrm{sym}},
\end{equation}
for some given $C \geq 1$ and $q \in (0,2)$, i.e., for large $\xi$ the density $f(\xi)$ behaves like  $\mathbb{C} \xi\colon\xi$ up to a lower order term.  We refer to the proof of \cite[Theorem 2.7]{FLS2} for the justification that a local minimizer of such a functional is a Griffith almost-minimizer with gauge $h(r) \leq C r^{1-q/2}$ for $r \leq 1$. Another notable application is to energies of the form
$$\int_{\Omega\setminus K} |e(u){-S}|^2 \dd{x} + \mathcal{H}^{1}(K)$$
with prestrain $S$, which  arise in the dual formulation of the biharmonic optimal compliance problem \cite{LemPakzad}.


Our epsilon-regularity result will rely on the fact that in a neighborhood $B(x_0,r_0)$ of a jump point $x_0 \in K$, $r_0 \ll 1$,  the set $K \cap B(x_0,r_0)$  is suitably flat and the elastic energy inside $B(x_0,r_0)$ is suitably small. Moreover, we need a quantified lower bound on the jump height at $x_0$, related to the difference of the values of $u$ in half balls around $x_0$. To this end, we introduce three basic quantities that will  be central in the paper, namely the normalized {elastic} energy, the flatness, and the normalized jump.

\noindent{\bf The  normalized {elastic} energy.} Let $(u,K)$ be an admissible pair. For any $x_0 \in \Omega$ and $r_0 > 0$ such that $B(x_0,r_0) \subset  \Omega$, we define the \emph{normalized {elastic} energy} of $u$ in $B(x_0,r_0)$ as
\begin{equation}\label{eq: main omega}
    \omega_u(x_0,r_0) :=   \frac{1}{r_0}\int_{B(x_0,r_0) \setminus K} \abs{e(u)}^2 \dd{x}.
\end{equation}
When there is no ambiguity, we write simply $\omega(x_0,r_0)$ instead of $\omega_u(x_0,r_0)$.

\noindent{\bf The flatness.} 
{{Let $x_0 \in \R^2$, $r_0 > 0$, and $K$} be a relatively closed subset of $B(x_0,r_0)$ containing $x_0$.}
We define the {\it  flatness} of $K$ in $B(x_0,r_0)$ by
\begin{equation}\label{eq_beta}
    \beta_K(x_0,r_0) := \frac{1}{r_0} \inf_\ell \sup_{y \in K \cap B(x_0,r_0)}{\rm dist }(y,\ell),
\end{equation}
where the infimum is taken over all lines $\ell$ through $x_0$.
An equivalent definition is that $\beta_K(x_0,r_0)$ is the infimum of all $\varepsilon > 0$ for which there exists a line $\ell$ through $x_0$ such that
\begin{equation}\label{eq_beta2}
    K \cap B(x_0,r_0) \subset \set{y \in B(x_0,r_0) \colon \,  \mathrm{dist}(y,\ell) \leq \varepsilon r_0}.
\end{equation}
It is easy to check that the infimum in (\ref{eq_beta})  is attained. When there is no ambiguity, we often write $\beta(x_0,r_0)$ instead of $\beta_K(x_0,r_0)$. For each $(x_0,r_0)$, we choose   (possibly not uniquely) a  unit normal vector  $\nu(x_0,r_0) \in \mathbb{S}^1$ of an approximating line $\ell$ in the sense of \eqref{eq_beta}. For $t >0$, we define the  sets \EEE
\begin{align}\label{eq :Bx0}
    D_{t}^\pm(x_0,r_0)  := \big\{ x \in B(x_0,r_0) \colon     \pm  (x - x_0) \cdot \nu(x_0,r_0) > t \big\},  
\end{align}
where we note that $K \cap (D_{t}^+(x_0,r_0) \cup D_{t}^-(x_0,r_0)) = \emptyset$ for all $t \ge \beta_K(x_0,r_0)r_0$.  We use the letter $D$ to remind the reader of the shape of the sets. We frequently omit writing $(x_0,r_0)$ if no confusion arises.

Note that there are many different notions of flatness, see e.g.\ \cite[Equation (8.2)]{Ambrosio-Fusco-Pallara:2000}. More precisely, $\beta_K$ could be called \emph{unilateral} flatness as later we also consider the bilateral flatness, see \eqref{eq: bilat flat} below.

\noindent{\bf The normalized jump.}
{Let $(u,K)$ be an admissible pair. For $x_0 \in K$, $r_0 > 0$ such that $B(x_0,r_0) \subset \Omega$ and $\beta_K(x_0,r_0) \le 1/2$}, \EEE we define the \emph{mean infinitesimal rotations and translations} on   $D_{\beta_K(x_0,r_0)r_0}^\pm$ (omitting $(x_0,r_0)$) by 
\begin{align}\label{eq: optimal values}
    A^\pm(x_0,r_0) = \dashint_{D_{\beta_K(x_0,r_0)r_0}^\pm} \frac{1}{2} (\nabla u - \nabla u^T) \dd{x}, \quad b^\pm(x_0,r_0) = \dashint_{D_{\beta_K(x_0,r_0)r_0}^\pm}  \big(u(x) - A^\pm(x_0,r_0) x\big)  \dd{x}
\end{align}
We define the \emph{normalized jump} of $u$ in $B(x_0,r_0)$ by 
\begin{equation}\label{def:normalizedJump}
    J_u(x_0,r_0) := \frac{1}{\sqrt{r_0}}  \inf_{y \in K \cap B(x_0,r_0) }  \big|(A^+(x_0,r_0) - A^{-}(x_0,r_0)) y + (b^+(x_0,r_0) - b^{-}(x_0,r_0)  )\big|, 
\end{equation}
which corresponds to the minimal jump between the  rigid motions $y \mapsto  A^\pm(x_0,r_0)y + b^\pm(x_0,r_0)$ {along the portion of the crack $K \cap { B(x_0,r_0)}$}. As before, we drop the subscript $u$ if no confusion arises.  We use here a more involved definition of the normalized jump compared to regularity results for the Mumford--Shah energy, see for instance \cite{Lemenant}, due to the invariance of the energy under  rigid motions. We also mention that a notion of this type is not necessary for epsilon-regularity results of Griffith  under a connectivity or separation property, as used in \cite{BIL} or \cite{CL}. In a subsequent work \cite{FLS3}, it will  also be \EEE instrumental to derive porosity properties.

\begin{remark}[Normalization]\label{rem: normalization}
{\normalfont

    The above quantities are called \emph{normalized} as they are invariant under rescaling. Indeed, given an admissible pair $(u,K)$ and a ball $B(x_0,r_0) \subset \Omega$, the pair {$(\tilde u,\tilde K)$ in $B(0,1)$} defined by $\tilde u(x):=  {r_0^{-1/2} u(x_0 + r_0 x)}  $ and $\tilde K := {r_0^{-1}(K-x_0)}$ satisfies
$$\omega_u(x_0,r_0) = \omega_{\tilde u}(0,1), \quad \beta_K(x_0,r_0) = \beta_{\tilde K}(0,1), \quad J_u(x_0,r_0) = J_{\tilde u}(0,1).  $$ {We emphasize that this is the natural rescaling of the problem as $(\tilde u, \tilde K)$ is an almost-minimizer with gauge $\tilde{h}(\cdot) := h(r_0 \cdot)$ in $B(0,1)$.}

}
\end{remark}

\begin{remark}[Scaling and shifting properties for $\beta$ and   $\omega$]\label{rmk_beta}
    \normalfont
    For all  $x_0 \in K$  and  $0 < r \leq r_0$ with  $B(x_0,r_0) \subset \Omega$, it holds 
    \begin{equation*}
        \beta_K(x_0,r) \leq \frac{r_0}{r} \beta_K(x_0,r_0),
    \end{equation*}
    and for all {balls} $B(x,r) \subset B(x_0,r_0)$ with $x \in K$,   one can easily check that  \EEE
    \begin{equation*}
        \beta_K(x,r) \leq \frac{2r_0}{r} \beta_K(x_0,r_0).
    \end{equation*}
    This follows directly from the definition.  {Likewise}, for each ball $B(x,r) \subset B(x_0,r_0)$, we have
    \begin{equation*}
        \omega_u(x,r) \leq  \frac{r_0}{r}  \omega_u(x_0,r_0).
    \end{equation*}
\end{remark}

\noindent{\bf The separation property.} We {recall} a notion that was central in \cite{CL}. 
\begin{definition}\label{defi_separation}
    {Let $x_0 \in \R^2$, $r_0 > 0$, and $K$ be a relatively closed subset of $B(x_0,r_0)$ containing $x_0$.} We say that $K$  \emph{separates}  $B(x_0,r_0)$ if {$\beta_K(x_0,r_0) \leq 1/2$} and the sets $D^+_{\beta_K(x_0,r_0) r_0}(x_0,r_0)$ and   $D^-_{\beta_K(x_0,r_0) r_0}(x_0,r_0)$ introduced in \eqref{eq :Bx0}  lie  in distinct connected components of $B(x_0,r_0) \setminus K$.
\end{definition}
We emphasize that,  under the assumption $\beta_K(x_0,r_0) \leq 1/2$, the separation property does not depend on the chosen line $\ell$ in $D^{\pm}_{r_0/2}(x_0,r_0)$, see  {\cite[Remark 2.8]{CL}}.  In \cite{CL}, a conditional epsilon-regularity result was shown by assuming  that \EEE this property {holds} for a Griffith almost-minimizer. A  main novelty of our contribution lies in removing this topological condition.

\subsection{Main result}

We let $\Omega \subset \R^2$ be an open set. We recall the definitions of the Griffith energy in \eqref{eq: main energy} and  of almost-minimizers in \eqref{amin}, as well as the notions in \eqref{eq: main omega}, \eqref{eq_beta}, and \eqref{def:normalizedJump}. We now formulate the main result of the paper.

\begin{theorem}[Epsilon-regularity]\label{th: eps_reg} 
    For each choice of exponent $\alpha \in (0,1)$, there exist  $\eps_0 > 0$ and $\gamma \in (0,1)$ (both depending on $\mathbb{C}$ and $\alpha$) such that the following holds.
    Let $(u,K)$ be an almost-minimizer of the Griffith energy \eqref{eq: main energy} with gauge $h(t) = h(1) t^\alpha$. 
    For all $x_0 \in K$ and  $r_0 > 0$ such that $B(x_0,r_0) \subset \Omega$ and
    \begin{align}\label{eq: smallli epsi}
        \beta_K(x_0,r_0) + J_u(x_0,r_0)^{-1} + h(r_0) \le \eps_0, 
    \end{align}
    the set $K \cap B(x_0,\gamma r_0)$ is a   $C^{1,\alpha/2}${-graph}.
\end{theorem}

\begin{remark}[Exponent $\alpha$,   condition \eqref{eq: smallli epsi}, initialization]\label{rmk:alpha12}
    {\normalfont

        (i)     For  $\alpha \in (0,1)$, the $C^{1,\alpha/2}$-regularity is expected to be optimal. In the case $\alpha = 1$, we still conclude that $K$ is Hölder differentiable in a smaller ball but our method fails to reach the exponent $C^{1,1/2}$. In the case that $h \equiv 0$ and $(u,K)$ is a Griffith minimizer, we recover the statement with $C^{1,1/2}$. This essentially amounts to using elliptic  regularity \EEE once we know the surface is $C^{1,1/4}$.  The details are spelled out \EEE in Remark \ref{rmk_alpha}.

        (ii) We mention that \eqref{eq: smallli epsi} is satisfied for $\mathcal{H}^1$-a.e.\ $x \in K$  for \EEE  $r_0>0$ sufficiently small. For the flatness, we refer to \cite[Proposition 3.1]{BIL}, and for the normalized jump one can use  properties of the jump set of $GSBD$ functions, see \cite[Section 6]{DalMaso:13}. We omit details in that direction as in \cite{FLS3} we obtain a finer result.

        (iii) The control on the flatness and the normalized jump in \eqref{eq: smallli epsi} are minimal in the sense that neither condition can be removed. Without control on $\beta_K$, $x_0$ could be a triple junction, i.e., $K$ consists locally of three curves meeting in $x_0$ with angle $2\pi/3$. Without control on $J_u(x_0,r_0)^{-1}$, $x_0$ could be a crack tip, i.e., $K$ is locally a curve ending  {at} $x_0$.

        (iv)  Epsilon-regularity results for the Mumford--Shah functional usually assume an initial control on the normalized  {elastic} energy $\omega_u$ and not on the jump $J_u$, see e.g.\ \cite[Theorem 8.2]{Ambrosio-Fusco-Pallara:2000} or   {\cite[Theorem~4.12]{Lemenant}}. There are less known variants, however, requiring a control on $J_u$ and not on $\omega_u$, see e.g.\  \cite[Theorem 4.6]{DavidC1} or  \cite[Corollary 52.25]{David}, i.e., both formulations are possible.
        Due to technical reasons related to the definition of the jump in \eqref{eq: optimal values}--\eqref{def:normalizedJump}  for the Griffith functional, it appears that {our proof only works for the formulation \eqref{eq: smallli epsi} with $J_u$.}

    (v) In \cite{BIL} it is assumed that $K \cap B(x_0,r_0) = \Gamma \cap B(x_0,r_0)$ for an isolated component  $\Gamma$ of $K$. In \cite{CL}, the separation property of Definition \ref{defi_separation} is assumed (in any dimension). The novelty of our result lies in the fact that we can remove assumptions of this kind.}
\end{remark}

\section{Preliminaries}\label{sec: prelim}

\subsection{An extension result}\label{subset: ext}

In the next sections, we frequently face the following problem: given a Griffith almost-minimizer $(u,K)$, we want to replace the set $K$ in a ball $B(x_0,r_0)$ by another set $L$ with better properties. In order to use the almost-minimality property \eqref{amin}, we need to construct an admissible competitor $(v,L)$, see \eqref{eq: competttt}. This relies on a suitable modification of $u$ in small subsets of $B(x_0,r_0)$ in order to obtain a function $v \in LD(\Omega \setminus L)$. For this, we use the extension result of  \cite[Section~3]{CL}. This subsection is devoted to the formulation of this proposition, which also involves additional notation and the definition of geometric functions. In {this} whole section, we will suppose that  the crack set $K$ {separates}  $B(x_0,r_0)$ in the sense of Definition \ref{defi_separation} in order to use the results of \cite{CL}. Later, we will see how to adapt the problem to the case when $K$ does not separate. For consistency with notation in the next sections, we denote the separating set with {the} letter $E$ instead of $K$.

Let   $E$ be a relatively closed subset of $\Omega$. We let $x_0 \in E$ and  $r_0 >0$ such that $B(x_0,r_0) \subset \Omega$, $E$ {separates} $B(x_0,r_0)$, and $\beta_E(x_0,r_0) \le 1/16  $. We denote the two connected components of $B(x_0,r_0) \setminus E$ containing $D^+_{r_0/2}(x_0,r_0)$ and   $D^-_{r_0/2}(x_0,r_0)$, respectively, by $\Omega_+{(x_0,r_0)}$ and $\Omega_-{(x_0,r_0)}$, see Definition \ref{defi_separation}. {We will typically drop $(x_0,r_0)$ in the preceding notation.} 
{Note that $B(x_0,r_0) \setminus E$ may have other connected components besides $\Omega_+$ and $\Omega_-$ but they are contained in a thin strip of width $\beta_E(x_0,r_0) r_0$, see Figure~\ref{fig:geoFunc}.}

\begin{definition}[Geometric function]\label{def: geo func}
    Let $\rho \in [r_0/2,3r_0/4]$ and $\tau \in [8\beta_E(x_0,r_0),  1/2)$. We say that a $100$-Lipschitz function 
    $$\delta\colon E \cap \overline{B(x_0,\rho)} \to [0,r_0/4]  $$ 
    is a \emph{geometric function} with parameters $(\rho,\tau)$ if 
    \begin{align}\label{eq: geotau}
        \beta_E(x,r) \le \tau \quad \quad \text{for  all $x \in E\cap \overline{B(x_0,\rho)}$ and for all $r \in (\delta(x),r_0/4]$}.
    \end{align}
\end{definition}

We exclude the case $r = \delta(x)$ in the definition since in the possible case $\delta(x) = 0$, the flatness $\beta_E(x,\delta(x))$ would not be well defined. The Lipschitz condition guarantees that radii of two overlapping balls $B(x_i,\delta(x_i))$, $i=1,2$,  are comparable.  Let us provide two examples: (1) Using Remark \ref{rmk_beta},  one can check that $\delta \equiv  {2}\tau^{-1}\beta_E(x_0,r_0)r_0$ is a geometric function for $\tau$,  where we use $ \tau^{-1}\beta_E(x_0,r_0) \leq 1/8$ to guarantee that $\delta \le r_0/4$. \EEE (2) If we have that $E$ is {relatively flat at all scales and locations}, namely $\beta_E(x,r) \le \tau $ for all $x \in B(x_0,3r_0/4)$ and $0 <r \le r_0/4$,  one could choose $\delta \equiv 0$.

Given a geometric function  $\delta \colon E \cap \overline{B(x_0,\rho)} \to [0,r_0/4] $  with  parameters $(\rho,\tau)$ for $\rho \in [r_0/2,3r_0/4]$ and $\tau \in [8\beta_E(x_0,r_0),  1/2)$, we define  $r_x := 10^{-5}\delta(x)$, and  the sets 
\begin{align}\label{eq: Wdef}
    W = \bigcup_{x \in E \cap \overline{B(x_0,\rho)} }  B(x,r_x), \quad \quad \quad  W_{10} = \bigcup_{x \in E \cap \overline{B(x_0,\rho)} }  B(x,10 r_x).
\end{align}
 Here, \EEE it is essential that $r_x$ is small compared to the Lipschitz constant of $\delta$  since this ensures that   balls $B(x,10 r_x)$ for adjacent points $x$  overlap in a suitable way. We \EEE  choose the constant $10^{-5}$ for definiteness. 
Moreover, we define the subset of $W_{10}$ consisting of balls close to the boundary, namely
\begin{align}\label{eq: 10bry}
    W^{\rm bdy}_{10} = \bigcup_{x \in \mathcal{W}^{\rm bdy} }  B(x,10 r_x), \quad  
    \text{where $\mathcal{W}^{\rm bdy}:=  \big\{ x\in  E \cap \overline{B(x_0,\rho)}\colon \,   B(x,50 r_x) \cap \partial B(x_0,\rho) \neq \emptyset  \big\}$}.
\end{align}
We define the extended domains
\begin{align}\label{exsets}
    \Omega^{\rm ext}_\pm :=  \Omega_\pm \cup W, 
\end{align}
see  again Figure \ref{fig:geoFunc}.  {A crucial point} is that the sets $\Omega^{\rm ext}_+ \cup \Omega^{\rm ext}_-$ completely cover  {$B(x_0,\rho) \setminus E  $} and that we can modify $u$  {on all of $B(x_0,r_0)$ by extending $u$ to the sets $\Omega^\pm_{\rm ext}$}. This is achieved  {by} the following result.

\begin{figure}
    \begin{tikzpicture}
        \draw[fill,green!10] (-0.5,0.05) circle (0.1);
        \draw[fill,green!10] (-0.7,0.07) circle (0.2);
        \draw[fill,green!10] (-1,0.1) circle (0.4);
        \draw[fill,green!10] (-2,0.2) circle (0.6);
        \draw[fill,green!10] (-1.5,0.15) circle (0.5);
        \draw[fill,green!10] (-2.5,0.0) circle (0.6);
        \draw[fill,green!10] (-1.9,-0.14) circle (0.5);
        \draw[fill,green!10] (1.2,-0.1) circle (0.2);
        \draw[fill,green!10] (1.5,-0.1) circle (0.4);
        \draw[fill,green!10] (1.8,0.2) circle (0.4);
        \draw[fill,green!10] (2.3,-0.1) circle (0.5);
        \draw[fill,green!10] (2.75,0.01) circle (0.5);
        \draw[fill,green!10] (-3.7,-1.5) circle (0.25);
        \node[] at (-4.5,-1.5) {$W = $};
        \draw[very thick](0,0) circle (3.0);
        \draw [red,very thick] plot [smooth] coordinates {(-3.0,0.1)  (-2.5,0.0) (-2,0.2) (0,0) (1,-0.1)};
        \draw [red,very thick] plot [smooth] coordinates {(-2.5,0.0)  (-2.2,-0.2) (-1,0.11) };
        \draw [red,very thick] plot [smooth] coordinates {(1,-0.1) (1.5,-0.1) (1.8,0.2)};
        \draw [red,very thick] plot [smooth] coordinates {(1.6,0.0) (2.5,-0.1) (3.0,0.2)};
        \draw[dotted] (-3,0.2) -- (3,0.2);
        \draw[dotted] (-3,-0.2) -- (3,-0.2);
        \draw [decorate,decoration={brace,amplitude=4pt,raise=4pt},yshift=0pt] (3.3,0.2) -- (3.3,-0.2) node [black,midway,xshift=1.4cm] {$2\beta_{E}  (x_0,r_0)r_0$};
        \node[] at (0, 1.8) {$\Omega_+$};
        \node[] at (0, -1.8) {$\Omega_-$};
        \node (a0) at (0,0) {};
        \draw[fill] (a0) circle [radius=1pt];
        \draw[] (-0.5,0.05) circle (0.1);
        \draw[] (-0.7,0.07) circle (0.2);
        \draw[] (-1,0.1) circle (0.4);
        \draw[] (-2,0.2) circle (0.6);
        \draw[] (-1.5,0.15) circle (0.5);
        \draw[] (-2.5,0.0) circle (0.6);
        \draw[->] (-3.5,1.5) -- (-2.35,0.75);
        \node[left] at (-3.5,1.5) {$B(x,\delta(x))$};
        \draw[->] (2.8,-2) -- (0.5,-0.15);
        \node[right] at (2.8,-2) {$E$};
        \draw[] (-1.9,-0.14) circle (0.5);
        \draw[] (1.2,-0.1) circle (0.2);
        \draw[] (1.5,-0.1) circle (0.4);
        \draw[] (1.8,0.2) circle (0.4);
        \draw[] (2.3,-0.1) circle (0.5);
        \draw[] (2.75,0.01) circle (0.5);
    \end{tikzpicture}
    \caption{The above figure illustrates how the geometric function  $\delta \colon E \to [0,r_0/4]$ \EEE is used to cover-up the nonflat parts of the (separating) crack $E$ with balls, which then make up $W$ as defined in \eqref{eq: Wdef}. Here we see that in the middle of the ball, as the crack is quite flat, the geometric function vanishes. Note the drawing is not scale accurate.}
    \label{fig:geoFunc}
\end{figure}
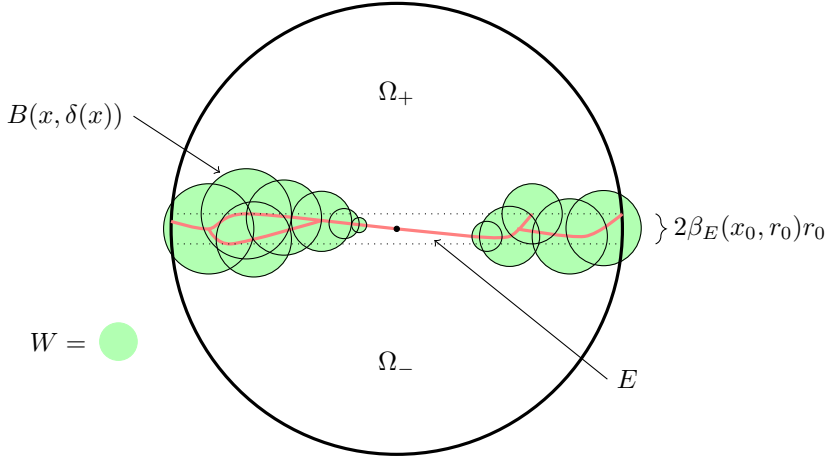

\begin{proposition}[Extension]\label{th: extension} 
    Under the above notation, for a separating set $E$ with $\beta_E(x_0,r_0) \le 10^{-8}/8$, there exists a universal constant $C \ge 1$ such that for each geometric function  with  parameters $(\rho,\tau)$ for $\rho \in [r_0/2,3r_0/4]$ and $\tau \in [8\beta_E(x_0,r_0),  10^{-8}  ] \EEE$ the following holds: for all functions $u \in LD(B(x_0,r_0) \setminus E)$ there exist two functions $v_ \pm \in LD_{\rm loc}(\Omega_\pm^{\rm ext})$ and relatively closed sets $S_\pm \subset {\Omega^{\rm ext}_{\pm}}$ such that 
    $$W \subset S_\pm \subset W_{10}, \quad \quad \quad   v_\pm = u \text{ on } \Omega_\pm^{\rm ext} \setminus S_\pm    $$
    and 
    $$\int_{ (B(x_0,\rho) \cap \Omega_\pm^{\rm ext})  \setminus W_{10}^{\rm bdy}   }  |e(v_\pm)|^2 \, {\rm d}x \le C \int_{ B(x_0,\rho) \cap \Omega_\pm  }  |e(u)|^2 \, {\rm d}x.   $$
    Moreover,  $\Omega^{\rm ext}_\pm$ satisfy 
    \begin{align}\label{extiii}
        \overline{B(x_0,\rho)}   \subset  {E} \cup \Omega^{\rm ext}_+ \cup \Omega^{\rm ext}_-.  
    \end{align}
\end{proposition}

The result can be found in {\cite[Lemma 3.1 and Remark 3.1]{CL}}.
Essentially, one can think of $S_\pm$ being equal to $W_{10}$. The distinction here is made only to ensure that $S_\pm$ is relatively closed.
{As $W \subset S_{\pm}$, we have $\Omega^{\rm ext}_{\pm} \setminus S_{\pm} \subset \Omega_{\pm}$ so the condition $v_{\pm} = u$ in $\Omega^{\rm ext}_{\pm} \setminus S_{\pm}$ makes sense. 
    Note that the extension $v_{\pm}$  coincides  with $u$ along $\partial B(0,\rho) \setminus W^{\rm bdy}_{10}$. Precisely, it follows from the definitions that $W_{10} \setminus B(0,\rho) \subset W^{\rm bdy}_{10}$ and thus }
    $$S_{\pm} \setminus B(0,\rho) \subset W^{\rm bdy}_{10}, \quad \quad \text{and} \quad \quad   {v_{\pm} = u \quad \text{in} \quad \Omega^{\rm ext}_{\pm} \setminus \big(B(x_0,\rho) \cup W^{\rm bdy}_{10}\big).}  $$
Note also that the energy of the extension $v_\pm$ cannot be controlled close to the boundary in the set $ W_{10}^{\rm bdy} $. This is due to the fact that the function $u$ is changed in adjacent balls  of comparable size contained in $W_{10}$ by using a partition of unity, but this procedure breaks down near $\partial B(x_0,\rho)$. {A possible} idea could be to refine the radii of the balls in $W$ towards the boundary, as in a Whitney-type covering.
{Note, however, that the existence of a geometric function satisfying $\delta(x) \to 0$ for $x \to \partial B(x_0,\rho)$ is already a kind of regularity property of {$E$} near $\partial B(x_0,\rho)$ that may not hold a priori.}
Therefore, whenever we later use the extension result to construct competitors, cf.\ \eqref{eq: competttt}, the lack of control {of $u$} in $W_{10}^{\rm bdy}$ {must be addressed}. To this end, we {will} choose a set $S_{\rm wall} \supset W_{10}^{\rm bdy}$ with a suitable control on $\mathcal{H}^1(\partial S_{\rm wall})$ and {`brutally'} define the competitor {as zero} in $S_{\rm wall}$, {which {decreases} the elastic energy but adds an additional `wall set' $\partial S_{\rm wall}$ to the crack/discontinuity set}.

\begin{figure}
    \begin{tikzpicture}[x=1cm,y=1cm]
        \draw[very thick](0,0) circle (3.0);
        \draw [red,very thick] plot [smooth] coordinates {(-3.0,0.1)  (-2.5,0.0) (-2,0.2) (0,0) (1,-0.1)};
        \draw [red,very thick] plot [smooth] coordinates {(-2.5,0.0)  (-2.2,-0.2) (-1,0.11) };
        \draw [red,very thick] plot [smooth] coordinates {(1,-0.1) (1.5,-0.1) (1.8,0.2)};
        \draw [red,very thick] plot [smooth] coordinates {(1.6,0.0) (2.5,-0.1) (3.0,0.2)};
        \draw[dotted] (-3,0.2) -- (3,0.2);
        \draw[dotted] (-3,-0.2) -- (3,-0.2);
        \draw[] (0,0) circle (1.7);
        \draw[] (1.7,0) circle (0.2);
        \draw[] (-1.7,0) circle (0.2);
        \begin{scope}[xshift=8cm]
            \draw[very thick](0,0) circle (3.0);
            \draw [red,very thick] plot [smooth] coordinates {(-3.0,0.1)  (-2.5,0.0) (-2,0.2) (0,0) (1,-0.1)};
            \draw [red,very thick] plot [smooth] coordinates {(-2.5,0.0)  (-2.2,-0.2) (-1,0.11) };
            \draw [red,very thick] plot [smooth] coordinates {(1,-0.1) (1.5,-0.1) (1.8,0.2)};
            \draw [red,very thick] plot [smooth] coordinates {(1.6,0.0) (2.5,-0.1) (3.0,0.2)};
            \draw[] (0,0) circle (1.7);
            \draw[fill,white] (0,0) circle (1.69);
            \draw[fill,white] (1.7,0) circle (0.19);
            \draw[fill,white] (-1.7,0) circle (0.19);
            \draw[blue,very thick] (1.7,0) circle (0.2);
            \draw[blue,very thick] (-1.7,0) circle (0.2);
            \draw[red,very thick] (-1.5,0) -- (1.5,0);
            \draw[dotted] (-3,0.2) -- (3,0.2);
            \draw[dotted] (-3,-0.2) -- (3,-0.2);
        \end{scope}
        \draw [->] plot [smooth] coordinates {(3.0,1.0) (4.0,1.2) (5.0,1.0)};
        \node[above] at (4.0,1.2) {Modify crack};
        \draw [->] (4.6,-1.5) -- (6.13,-0.22);
        \draw [->] (4.6,-1.5) -- (9.49,-0.12);
        \node[left] at (4.6,-1.5) {Wall-set};
    \end{tikzpicture}
    \caption{We often modify the pair $(u,K)$ via a direct construction. To gain some freedom to modify the crack in the interior and additionally preserve geometric properties of the crack or control the elastic energy, we must include a wall-set in the modified crack. One  of the simplest examples is drawn above.}
    \label{fig:simpleWallSet}
\end{figure}
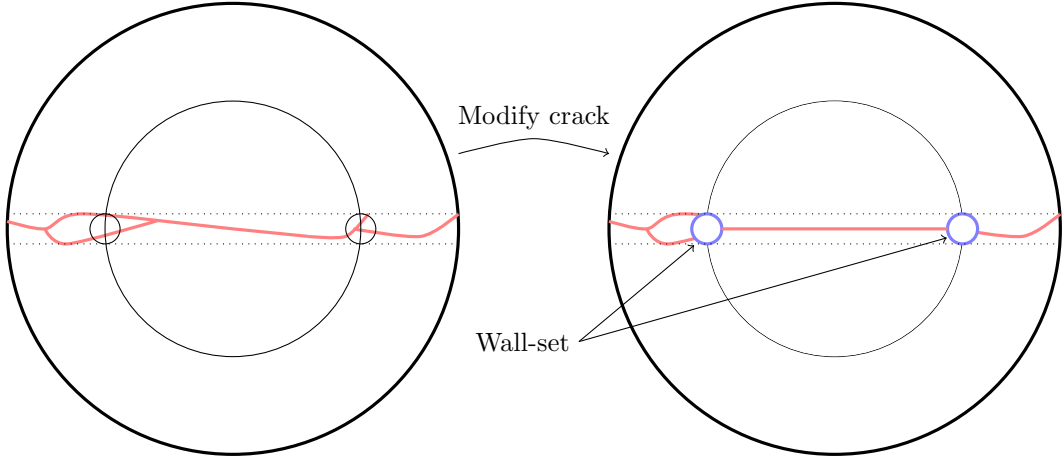

Since $\delta \equiv  {2} \tau^{-1}\beta_E(x_0,r_0)r_0$ is a possible geometric function, we see that a possible choice of $S_{\rm wall}$ is the union of two balls of radius $\sim \beta_{E}(x_0,r_0)r_0$ close to the boundary, {see Figure \ref{fig:simpleWallSet}.} A wall set with $\mathcal{H}^1(\partial S_{\rm wall}) \sim \beta_E(x_0,r_0)r_0$ is typical in the construction of extensions, see e.g.\ the simple example in \cite[Lemma 4.2]{BIL}. We emphasize that a more delicate variant with {an arbitrary small} wall set derived  in \cite[Lemma 4.5]{BIL} is not sufficient for our purposes since it provides a suboptimal scaling in terms of the elastic energy. As in \cite{CL}, in the present work we resort to the concept of \emph{bad mass}, see  \eqref{eq_goodball1}--\eqref{eq_goodball2} below, which allows for even more refined estimates on wall sets based on a suitable choice of a geometric function.

\begin{remark}[Refinement]\label{rem: refinement}
    {\normalfont
        In {\cite[Remark 3.2]{CL}}, the following refined estimate compared to Proposition~\ref{th: extension}  has been established: 
        $$
        \int_{ S_\pm  \setminus W_{10}^{\rm bdy}   }  |e(v_\pm)|^2 \, {\rm d}x \le C \int_{  E_{\rm tube}  }  |e(u)|^2 \, {\rm d}x 
        ,$$
        where $E_{\rm tube}$ denotes a `tubular' neighborhood around $E$, defined by
        \begin{align}\label{eq: tube}
            E_{\rm tube} := \big\{ y \in {B(x_0,\rho)} \colon  \exists \, x \in E \cap \overline{B(x_0,\rho)} \text{ s.t.\ } y \in B(x,50r_x) \text{ and } {\rm dist}(y,E) \ge r_x \big\}. 
        \end{align} 
    }
\end{remark}

\subsection{Piecewise Korn-Poincar\'e inequality}\label{korn-sec}

In this subsection, we collect some results for sets of finite perimeter and affine maps, and we recall the piecewise Korn inequality proved in \cite{Friedrich:15-4}. The excited reader may wish to skip this subsection and return as needed.

\textbf{Sets of finite perimeter.}  For a set of finite perimeter ${P}$, we denote by $\partial^* {P}$ its \emph{essential boundary} and by ${P}^{(1)}$ its \emph{points of density $1$}, see \cite[Definition 3.60]{Ambrosio-Fusco-Pallara:2000}.  By ${\rm diam}({P}) = {\rm ess \, sup}\lbrace |x-y| \colon x,y \in P \rbrace$ we denote the essential diameter of $P$. \EEE  A set of finite perimeter $P$ is called \emph{indecomposable} if it cannot be written as   {$P_1 \cup P_2$ with $P_1 \cap P_2 = \emptyset$, $\mathcal{L}^{ 2 }(P_1), \mathcal{L}^{ 2 }(P_2) >0$, and $\mathcal{H}^{ 1 }(\partial^* P) = \mathcal{H}^{ 1 }(\partial^* P_1) + \mathcal{H}^{ 1 }(\partial^* P_2)$}.  This notion generalizes the concept of  connectedness to sets of finite perimeter.  
We will need the following two lemmas on sets of finite perimeter.

\begin{lemma}\label{lemma: diam}
    Let ${P} \subset \R^2$ be a bounded set that has finite perimeter and is indecomposable. Then,  ${\rm diam}({P}) \le \mathcal{H}^1(\partial^* {P})$.
\end{lemma}

\begin{lemma}\label{lemma: maggi}
    Let ${U} \subset \R^2$ {be an} open, bounded {set} with Lipschitz boundary. Let  $r>0$ and   {$\lambda \in (0,1)$}. Then, for all  sets $P \subset r U := \lbrace r x \colon x \in U\rbrace $ with finite perimeter one has either
    \begin{align*}
        {\rm (i)}  \ \ \mathcal{L}^2({P}) > {\lambda} \mathcal{L}^2(r{U}) \ \ \ \text{ or } \ \ \ 
        {\rm (ii)}  \ \ {\mathcal{H}^1(\partial^* P)} \le C   \mathcal{H}^1(\partial^* {P} \cap r{U})
    \end{align*}
    for some constant $C > 0$ only depending on $U$ and {$\lambda$}.
\end{lemma}

The proof of Lemma \ref{lemma: diam} can be found in \cite[Proposition 12.19, Remark 12.28]{maggi}. Noting that $\mathcal{H}^1(\partial^* P\cap \partial (rU))\leq C{\rm diam}(P)$ for some constant $C>0$ depending only on $U$, Lemma \ref{lemma: maggi} is a consequence of \cite[Lemma 4.6]{Solombrino}, which was stated with $r=1$ and $1/2$ in place of {$\lambda$}. Inspection of the proof, however,  shows that any value  {$ \lambda \in (0,1)$}  works, and that the statement is scaling invariant.  (Strictly speaking, if $P$ is not indecomposable,  we apply \cite[Lemma 4.6]{Solombrino}  for its indecomposable components.)   \EEE

We also recall the  structure theorem of the boundary of planar sets   $E$  of finite perimeter in  \cite[Corollary~1]{Ambrosio-Morel}: there exists a unique countable decomposition  of $\partial^* E$ into pairwise almost disjoint rectifiable Jordan curves. Here, we  say that $\Gamma \subset \R^2$ is  a rectifiable Jordan curve if $\Gamma = {\Gamma}([a,b])$ for some $a < b$, {where by an abuse of notation the parametrization of $\Gamma$ is a  Lipschitz continuous map also denoted by $\Gamma$ that is one-to-one on $[a,b)$ and such that  $\Gamma(a) = \Gamma(b)­$.}

\textbf{Caccioppoli partitions.} We say that a partition $(P_j)_{j \ge 1}$ of an open set $U\subset \R^2$ is a \textit{Caccioppoli partition} of $U$ if  $\sum\nolimits_j \mathcal{H}^1(\partial^* P_j\cap U) < + \infty$.   The  local structure of Caccioppoli partitions can be characterized as follows, see \cite[Theorem 4.17]{Ambrosio-Fusco-Pallara:2000}.
\begin{theorem}\label{th: local structure}
    Let $(P_j)_j$ be a Caccioppoli partition of $U$. Then 
    $$\bigcup\nolimits_j P_j^{(1)} \cup \bigcup\nolimits_{i \neq j} (\partial^* P_i \cap \partial^* P_j)$$
    contains $\mathcal{H}^{1}$-almost all of $U$, where $P^{(1)}$ denote the points of density $1$.
\end{theorem}

\textbf{Rigid motions.}   We {recall} that $a\colon \R^2 \to \R^2$ is an (infinitesimal) \emph{rigid motion} if $a$ is affine and its \emph{symmetrized gradient} satisfies $ e(a) = ( \nabla a + (\nabla a)^{T})/2  = 0$. The following lemma will be used several times. For a proof and more general statements of this kind, we refer, e.g.,  to  \cite[Lemma~3.4]{FM}.

\begin{lemma}\label{lemma: rigid motion}
    Let $x_0 \in \R^2$, $R, \theta >0$, {and $p\in   [1 \EEE ,\infty)$}.   Let $a\colon \R^2 \to \R^2$ be affine, defined by $a(x) = A\,x + b$ for $x \in \R^2$,   and let ${U} \subset  B(x_0,R) \EEE \subset \R^2$ with  $\mathcal{L}^2({U}) \ge \theta \mathcal{L}^2(B(x_0,R))$.  Then, there exists a constant $c>0$ only depending on  $\theta$ and $p$  such that 
    {
        \begin{align*}
            \Vert a\Vert_{L^p(B(x_0,R))} \le (\sqrt{\pi}  R)^{2/p}  \Vert a\Vert_{L^\infty(B(x_0,R))}  \le c \Vert  a \Vert_{L^p({U})}, \quad \quad \quad    |A| \le   c   R^{-(2/p+1)}   \Vert  a \Vert_{L^p({U})}.
        \end{align*}
    }
\end{lemma}

  \textbf{Korn inequalities.} \EEE   As mentioned, we will rely on   Korn-Poincar\'e inequalities for $GSBD$ functions.  There has  been a variety of efforts to overcome a loss of control due to the combination of frame indifference and fracture, see e.g.\ \cite{CCS, Chambolle-Conti-Francfort:2014, Conti-Iurlano:15.2, Friedrich:15-3,  Friedrich:15-4}.  Of these variants, the inequality which provides the most information on the function is the \emph{piecewise Korn inequality}, proven   in dimension two, see  \cite[Theorem 2.1]{Friedrich:15-4} for the case $p=2$ and \cite[Remark 5.6]{Friedrich:15-4} for general $p \in (1,\infty)$. We state the result for $GSBD^p$-functions \cite{DalMaso:13}, but the reader can simply replace $u \in GSBD^p(\Omega)$ by an admissible pair $(u,K)$, with the relation $J_u = K$.

\begin{proposition}[Piecewise Korn-Poincar\'e inequality]\label{th: kornpoin-sharp-old}
    Let $\Omega \subset \R^2$ be open, bounded with Lipschitz boundary, and let $p \in (1,\infty)$.   Then, there exist constants $c>0$ depending on $p$,  and   $C_{\rm {K}orn}$ depending additionally on $\Omega$   such that for  each $u \in GSBD^p(\Omega)$   there is a Caccioppoli partition $\Omega = \bigcup^\infty_{j=1} P_j$ and corresponding   rigid motions $(a_j)_j$ such that   $v := u - \sum\nolimits^\infty_{j=1} a_j \chi_{P_j} \in SBV(\Omega; \R^2) \cap { L^\infty}(\Omega;\R^2)$
    and 
    \begin{align}\label{eq: small set main}
        {\rm (i)} & \ \   \sum\nolimits_{j=1}^\infty \mathcal{H}^1( \partial^* P_j ) \le  c(\mathcal{H}^1(J_u) + \mathcal{H}^1(\partial \Omega)),\notag\\
        {\rm (ii)}  &\ \ \Vert v \Vert_{L^{\infty}(\Omega)} \le C_{\rm {K}orn} \Vert  e(u) \Vert_{L^{p}(\Omega)},\notag\\ 
        {\rm (iii)}  &\ \   \Vert \nabla v\Vert_{L^1(\Omega)} \le C_{\rm {K}orn} \Vert  e(u) \Vert_{L^{p}(\Omega)}.
    \end{align}
\end{proposition}

 The constant $c$ is invariant under rescaling of the domain. \EEE The result says   that for a function belonging to $GSBD^p$ there is a piecewise rigid motion such that the full gradient of the function minus this rigid motion is controlled by the elastic energy.
We use the above inequality within the proof of epsilon-regularity for a decay estimate on a quantitative measure of the crack's failure to topologically separate.

    \section{Epsilon-regularity result: Proof of Theorem \texorpdfstring{\ref{th: eps_reg}}{2.5}}\label{sec: eps_reg}

    This section is devoted to the proof of Theorem \ref{th: eps_reg}. We first present the main proof strategy, postponing the proof of all technical intermediate steps to subsequent subsections.

    \subsection{Main auxiliary steps and general proof strategy}\label{sec: outline}

    The main idea of our proof originates from \cite{David} and  \cite{Lemenant} for the Mumford--Shah functional, which then has been adapted successfully  to the Griffith setting in \cite{BIL} (for a connected crack $K$ in dimension $2$)  and \cite{CL}  (under  the  separation property of Definition \ref{defi_separation} in all dimensions). The key point is the observation that Griffith almost-minimizers behave like almost-minimimal sets in {regimes} of low elastic energy.
    To establish {this connection},  we  first introduce   the {\it bilateral flatness} of $K$ in $B(x_0,r_0) \subset  \Omega$ {defined} by 
    \begin{equation}\label{eq: bilat flat}
        \beta^{\rm bil}_K(x_0,r_0) := r_0^{-1} \inf_\ell \max \Set{ \sup_{y \in K \cap B(x_0,r_0)} \mathrm{dist}(y,\ell), \sup_{y \in \ell \cap B(x_0,r_0)} \mathrm{dist}(y,K)},
    \end{equation}
    where the infimum is taken over all lines $\ell$ through $x_0$. Compare this definition to  \eqref{eq_beta} and observe that $ \beta^{\rm bil}_K(x_0,r_0) \ge  \beta_K(x_0,r_0)$. Now, we say that $K$ is \emph{Reifenberg-flat} in $B(x_0,r_0)$ {with parameter $\eps>0$} if $ \beta^{\rm bil}_K(x,r) \le \eps$ for all $x \in K \cap  {B(x_0,9r_0/10)}$  and all $r \le r_0/10$. In \cite{DPT, Reifenberg},  it is shown that for $\eps$ small enough $K$ is locally a $C^\alpha$-curve for $\alpha = 1- C\eps$ in $B(x_0,r_0/2)$. In the sequel, we will use the following stronger version.

    \begin{proposition}[Reifenberg parametrization theorem]\label{lem_reifenberg_C1alpha}
        Let $K \subset \R^2$ be  a \EEE closed set and $x_0 \in K$ be such that there exist constants $C_{\rm Reif} \geq 1$, $r_0 > 0$, and $0 < \alpha \leq 1$ with
        \begin{equation*}
            \beta^{\rm bil}_K(x,r) \leq C_{\rm Reif} \Big(\frac{r}{r_0}\Big)^\alpha \ \text{for all $x \in K \cap B(x_0,r_0)$ and all $0 < r \leq r_0$}. 
        \end{equation*}
        Then there exists $\gamma \in (0,1)$ depending on $C_{\rm Reif}$ and $\alpha$ such that $K \cap B(x_0,\gamma r_0)$ is a $10^{-2}$-Lipschitz graph and a $C^{1,\alpha}$-curve   with  a universal {(normalized)} H\"older constant. \EEE
    \end{proposition}

    For a proof we refer the reader to  \cite[Lemma 6.4]{BIL}. Note that there the notation $\beta$ is used for the bilateral flatness, in place of $ \beta^{\rm bil}_K \EEE$. Moreover, the result is only stated for   $x_0 = 0$, but readily extends to the general case. An alternative proof is given in \cite[Section 10]{DPT}. The result says that,  if we just have  {rough} information about the geometry of $K$ at \emph{all} locations and scales, then we actually have a fine control on the structure and the regularity.   On first sight, this result seems not expedient as initially we only control $\beta_K$, see \eqref{eq: smallli epsi}. Indeed, in the following it will be easier to obtain a decay for $\beta_K$ as we will frequently use the relation  $\beta_K \leq \beta_E$ if $K \subset E$, which does not holds for $\beta^{\rm bil}$ instead. 
    The following result, however, states that for separating  sets both notions are comparable.  

    \begin{lemma}[Unilateral vs.\ bilateral flatness]\label{lem_bilateral_flatness}
        {{Let $x_0 \in \R^2$, $r_0 > 0$, and $K$} be a relatively closed subset of $B(x_0,r_0)$.}
        If $K$ {separates} $B(x_0,r_0)$, then
        \begin{equation*}
            \beta^{\rm bil}_K(x_0,r_0) \le  4\beta_K(x_0,r_0).
        \end{equation*}
    \end{lemma}

    The result is proved in Subsection \ref{sec: bila}.  Summarizing Proposition \ref{lem_reifenberg_C1alpha} and Lemma \ref{lem_bilateral_flatness},  given an almost-minimizer $(u,K)$,  it is enough to establish a decay on $\beta_K$ and to show that $K$ locally satisfies the separation condition {to obtain regularity of $K$}.

    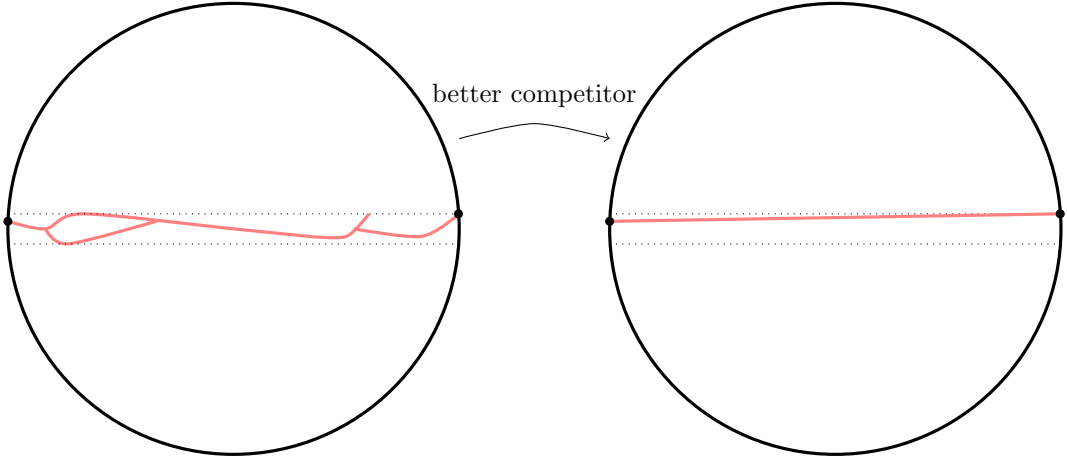
\begin{figure}
        \begin{tikzpicture}[x=1cm,y=1cm]
            \draw[very thick,dotted] (3,0) to[out=-180,in=0] (-3,0);
            \draw [very thick,red,fill=blue!10] (-3,0) to[out=90,in=180] (0,3.0) to[out=0,in=90] (3,0) to[out=180,in=10] (2.5,-0.20) to[out=190,in=0] (2.0,0.3) to[out=180,in=-20] (0.0,0.0) to[out=160,in=30] (-0.5,0.0) to[out=210,in=-30] (-1.0,0.0) to[out=150,in=-30] (-1.2,0.1) to[out=150,in=0] (-1.8,0.4) to[out=180,in=20] (-2.4,0.2)  to[out=200,in=10] (-3.0,0.0);
            \draw [very thick,red,fill=green!10] (3,0) (-1.2,0.1) to[out=150,in=0] (-1.8,0.4) to[out=180,in=20] (-2.4,0.2)  to[out=-70,in=180] (-2.0,-0.3) to[out=0,in=250] (-1.2,0.1);
            \draw[very thick,red,fill=green!10] (1.0,0.2) circle (0.1);
            \draw[very thick,dotted] (3,0) to[out=-180,in=0] (-3,0);
            \draw [very thick,red] (3,0) to[out=180,in=10] (2.5,-0.20) to[out=190,in=0] (2.0,0.3) to[out=180,in=-20] (0.0,0.0) to[out=160,in=30] (-0.5,0.0) to[out=210,in=-30] (-1.0,0.0) to[out=150,in=-30] (-1.2,0.1) to[out=150,in=0] (-1.8,0.4) to[out=180,in=20] (-2.4,0.2)  to[out=200,in=10] (-3.0,0.0);
            \draw [very thick,red] (3,0) (-1.2,0.1) to[out=150,in=0] (-1.8,0.4) to[out=180,in=20] (-2.4,0.2)  to[out=-70,in=180] (-2.0,-0.3) to[out=0,in=250] (-1.2,0.1);
            \node[] at (0, 1.8) {$\Omega_+^K$};
            \node[] at (0, -1.8) {$\Omega_-^K$};
            \draw [very thick] (-3,0) to[out=90,in=180] (0,3.0) to[out=0,in=90] (3,0);
            \draw [very thick] (3,0) to[out=-90,in=0] (0,-3.0) to[out=180,in=-90] (-3,0);
            \begin{scope}[xshift=8cm]
                \draw [very thick,red,fill=blue!10] (-3,0) to[out=90,in=180] (0,3.0) to[out=0,in=90] (3,0) to[out=-180,in=0] (-3,0);
                \draw [very thick,dotted] (3,0) to[out=180,in=10] (2.5,-0.20) to[out=190,in=0] (2.0,0.3) to[out=180,in=-20] (0.0,0.0) to[out=160,in=30] (-0.5,0.0) to[out=210,in=-30] (-1.0,0.0) to[out=150,in=-30] (-1.2,0.1) to[out=150,in=0] (-1.8,0.4) to[out=180,in=20] (-2.4,0.2)  to[out=200,in=10] (-3.0,0.0);
                \draw [very thick,dotted] (-2.4,0.2)  to[out=-70,in=180] (-2.0,-0.3) to[out=0,in=250] (-1.2,0.1);
                \draw[very thick,dotted] (1.0,0.2) circle (0.1);
                \draw [very thick,red] (3,0) to[out=-180,in=0] (-3,0);
                \draw [very thick] (-3,0) to[out=90,in=180] (0,3.0) to[out=0,in=90] (3,0);
                \draw [very thick] (3,0) to[out=-90,in=0] (0,-3.0) to[out=180,in=-90] (-3,0);
                \node[] at (0, 1.8) {$\Omega_+^L$};
                \node[] at (0, -1.8) {$\Omega_-^L$};
            \end{scope}
            \draw [->] plot [smooth] coordinates {(3.0,1.2) (4.0,1.4) (5.0,1.2)};
            \node[above] at (4.0,1.5) {better geometry};
        \end{tikzpicture}
        \caption{
        The above picture illustrates the idea that, if the crack $K$ is not flat, then it is not near minimal, and there is a closer to minimal `separation' competitor $L$. To obtain decay of the flatness $\beta_K$, we take advantage of the converse implication:  sets that are nearly minimal are quantifiably almost flat in dimension two. However, this is complicated by the geometry's interaction with the elastic term as the domain $\Omega^L_{ \pm}$ may intersect with multiple connected  components \EEE of $B(x_0,r_0)\setminus K$. We will use the extension Proposition \ref{th: extension} to find a low energy extension of the function  $u|_{\Omega_{\pm}^K}$ onto  $\Omega_{\pm}^L$.}
        \label{fig:minimalToFlat}
    \end{figure}

    We now describe  how to obtain a decay on $\beta_K$. Suppose first for simplicity that  a Griffith almost-minimizer $(u,K)$ is given such that $K$ {separates}  $B(x_0,r_0) \subset \Omega$ in the sense of Definition \ref{defi_separation} and such that  $K \cap B(x_0,r_0)$ is a curve with two endpoints on $\partial B(x_0,r_0)$. Intuitively, by a contradiction argument we see that  $\beta_K(x_0,r_0)$ cannot be too big, say $\beta_K(x_0,r_0) \ge \lambda >0$, as otherwise we could find a competitor curve $L$ {(e.g., a straight line)} in $B(x_0,r_0)$ with the same start and endpoint on  $\partial B(x_0,r_0)$ such that 
    \begin{align}\label{eq: lambdastar}
        \mathcal{H}^1\big(L\cap B(x_0,r_0) \big) \le \mathcal{H}^1\big(K\cap B(x_0,r_0) \big) - \lambda_* r_0
    \end{align}
    for some $\lambda_*$ depending only on $\lambda$ in a quantified way. This intuition is depicted in Figure \ref{fig:minimalToFlat} and made precise below in terms of separation competitors, see Subsection \ref{section_flatness}. Then, the strategy lies in finding a modification $v \in LD(\Omega \setminus L)$ with $v = u$ a.e.\ in $\Omega \setminus (K \cup B(x_0,r_0))$ and
    \begin{equation*}
        \int_{B(x_0,r_0)\setminus L} \mathbb{C} e(v) \colon e(v) \dd{x} \leq C \int_{B(x_0,r_0)\setminus K}  \mathbb{C} e(u) \colon e(u)  \dd{x}
    \end{equation*}
    for some $C \geq 1$. By comparing the Griffith energy of $(u,K)$ and $(v,L)$, and using the almost-minimality property of $(u,K)$ in \eqref{amin} we thus obtain 
    \begin{equation}\label{eq_v_comparison2}
        \HH^{1}\big(K \cap B(x_0,r_0)\big) \leq \HH^{1}\big(L \cap B(x_0,r_0)\big) +  C  \omega_u(x_0,  r_0) r_0 + h(r_0) r_0,
    \end{equation}
    which yields a contradiction to \eqref{eq: lambdastar}, provided that $\omega_u( x_0,  r_0) + h(r_0)$ is small enough. Formulated as a positive statement, this suggests that the flatness $\beta_K(x_0,r_0)$ should be controlled suitably in terms of  $\omega_u( x_0,  r_0)$ and $h(r_0)$. The exact realization of the strategy, however, is intricate as in  the construction of $v$ two additional crack sets need to be taken into account: (i) If $K$ does not  separate \EEE $B(x_0,r_0)$, an additional error term on the right-hand side of \eqref{eq: lambdastar} appears measuring the `size of holes' of the set $K$ that need to be `filled' such that the set becomes a curve  separating \EEE $B(x_0,r_0)$. (ii) The construction of  $v$ relies on the extension result in Proposition \ref{th: extension} and introduces an additional \emph{wall set} $\partial S_{\rm wall}$ on which $v$ may jump, i.e., $v \in LD(\Omega \setminus (L\cup \partial S_{\rm wall}))$. Thus,  {one pays the price of an additional error} term $\mathcal{H}^1(\partial S_{\rm wall})$ on the right-hand side of \eqref{eq_v_comparison2}. 

    We refer to Figures \ref{fig:basiccoarea} and \ref{fig:simpleWallSet} for an illustration of the two phenomena. 
    Both errors need to be quantified. For their  exact formulation, we need to introduce two additional quantities.

    \noindent \textbf{Length of holes:} 
    {Let $x_0 \in K$ and  $r_0 > 0$ be such that $B(x_0,r_0) \subset \Omega$ and $\beta_K(x_0,r_0) \le 1/2$}.  We  define the  \emph{normalized length of holes} by
    \begin{align}\label{eq: eta definition}
        \eta_K(x_0,r_0):= \frac{1}{r_0}\min\big\{ \mathcal{H}^1\big( E   \setminus K\big) \colon E \text{ satisfies \eqref{eq: main assu} in }  B(x_0,r_0) \big\}, 
    \end{align}
    where  
    \begin{equation}\label{eq: main assu}
        \begin{gathered}
            \text{$E$ is a relatively closed {and coral} subset of $B(x_0,r_0)$ such that $K \cap B(x_0,r_0)  \subset E$},\\
            \text{$\beta_E(x_0,r_0) = \beta_K(x_0,r_0)$, and $E$ {separates} $B(x_0,r_0)$.}
        \end{gathered}
    \end{equation}
    If no confusion arises, we simply write $\eta$ in place of $\eta_K$.  For each $B(x_0,r_0)$, we denote by $E(x_0,r_0)$  a minimal \emph{separating extension}  of $K$ in $B(x_0,r_0)$, i.e., $E(x_0,r_0)$ satisfies \eqref{eq: main assu} and 
    \begin{align}\label{eq: eta definition2}
        {\frac{1}{r_0}}\mathcal{H}^1\big( E(x_0,r_0)   \setminus K\big) = \eta_K(x_0,r_0).
    \end{align}
    Note that the choice is in general {non-unique} and we choose an arbitrary minimal separating extension.
    {The proof of existence of such a minimal extension is standard but is detailed in Appendix \ref{sec:MinimalExtension} for the interested reader.}

    \noindent \textbf{Bad mass:}    {We define
    \begin{equation}\label{eqn:constTau}
        \tau :=  10^{-10}
    \end{equation}
    for brevity.}  {Consider again $x_0 \in K$ and  $r_0 > 0$ such that $B(x_0,r_0) \subset \Omega$ and $\beta_K(x_0,r_0) \leq 1/2$. Let $E(x_0,r_0)$ be a minimal separating extension of $K$ in $B(x_0,r_0)$.  For $x\in K \cap {B(x_0,9r_0/10)}$ and $t \in (0,r_0/10)$, we say that $B(x,t)$ is a \emph{good ball} provided that {$K$ is flat and nearly separating in the sense that}  
        \begin{equation}\label{eq_goodball1}
            \beta_K(x,t) \leq \tau
        \end{equation}
        and  
        \begin{equation}\label{eq_goodball2}
            t^{-1} \HH^1\big( (E(x_0,r_0) \cap B(x,t)) \setminus K\big) \leq \tau.
        \end{equation}
        We define the \emph{stopping time function}  $\sigma \colon K \cap {B(x_0,9r_0/10)} \to [0,+\infty]$ \EEE by
        \begin{align}\label{defdx}
            \sigma(x) := \inf \big\{ r > 0 \colon \,  \text{$B(x,t)$ is a good ball for all $t \in [r, r_0 / 10]$}\big\}.
        \end{align}
        The name is derived from considering the radius $t$ as time variable which stops once  \eqref{eq_goodball1} or \eqref{eq_goodball2} fails, i.e., when the flatness or the size of holes get too big.   
        Then, we define the \emph{bad set} in  $B(x_0,r_0)$ by 
        \begin{equation}\label{eq: the R}
            R(x_0,r_0):=  B(x_0,r_0) \cap \bigcup_{x \in \mathcal{R}(x_0,r_0)}  B\big(x, M   \sigma(x)\big),
        \end{equation}
        where  
        \begin{equation}\label{eqn:constA}
            M := 10^6
        \end{equation} 
        and
        $$
        \mathcal{R}(x_0,r_0) := \big\{x \in K \cap {B(x_0,9r_0/10)} \colon \, \sigma(x)>0\big\}.
        $$
        We remark that $M$ is chosen to be large relative to $10^5$, as appears in the definition of $r_x$ preceding (\ref{eq: Wdef}), and the specific choice is only for definiteness. Correspondingly, we define the \emph{bad mass} of $K$ in $B(x_0,r_0)$ by
        \begin{equation}\label{eq_defi_bad_mass}
            m_K(x_0,r_0):= \frac{1}{r_0} \mathcal{H}^1\big(K \cap R(x_0,r_0)\big).
        \end{equation}
        When there is no ambiguity, we write $m(x_0,r_0)$ instead of $m_K(x_0,r_0)$.

        {We will see later in Remark \ref{rmk_betaF} that for all $x \in K \cap {B(x_0,9r_0/10)}$ and $t \in (\sigma(x),r_0/10)$, we have $\beta_K^{\rm bil}(x,r) \leq 16 \tau$. In particular, this means that, if  $m_K \equiv 0$ and thus $\sigma \equiv 0$, then $K$   is a Reifenberg-flat  set {with parameter $16\tau$}.  Hence,  the bad mass can be regarded as a quantification of how much $K$ differs from being Reifenberg-flat.}
        In Subsection \ref{sec: stop time}  below, we will construct a geometric function (see Definition \ref{def: geo func}) based on the stopping time \eqref{defdx}. {Herein}, we will see that  the set $W$ in \eqref{eq: Wdef} and in particular the set $W_{10}^{\rm bdy}$ in \eqref{eq: 10bry}, on which the extension of Proposition \ref{th: extension} cannot be controlled, are related to $R(x_0,r_0)$. Therefore,  $R(x_0,r_0)$ is important for the construction of a  wall set $S_{\rm wall} \supset W_{10}^{\rm bdy}$, and $\mathcal{H}^1(\partial S_{\rm wall})$ will be controlled in terms of $m_K(x_0,r_0)$.

        \begin{remark}[Normalization]\label{norm2}
            {\normalfont We observe that the quantities $\eta_K$ and $m_K$ are also normalized in the sense of Remark \ref{rem: normalization}.} 
        \end{remark}

        After the definition of  $\eta_K$ and  $m_K$, \EEE we now come to the main statement on the control of  the flatness.

        \begin{proposition}[Control of the flatness]\label{prop_flatness_decay}
            Let $(u,K)$ be a Griffith almost-minimizer. There exist a universal constant $\eps_{\rm flat} > 0$ and a constant  $C_{\rm flat} \geq 1$ depending only on $\mathbb{C}$  such that for each  $x_0 \in K$ and  $r_0 > 0$ with  $B(x_0,r_0) \subset \Omega$,  $h(r_0) \leq \varepsilon_{\rm Ahlf}$,  and ${\beta_K}(x_0,r_0) \leq \eps_{\rm flat}$, and for each $0 < b \le 10^{-7}$ it holds that 
            \begin{equation}\label{eq: control beta}
                \beta_K(x_0, b r_0)^2 \leq C_{\rm flat}b^{-2} \Big(\omega_u(x_0,r_0) + \beta_K(x_0,r_0) m_K(x_0,r_0) + \eta_K(x_0,r_0) + h(r_0)\Big).
            \end{equation}
        \end{proposition}

        For the proof we refer to Subsection \ref{section_flatness} below. Comparing with the discussion below \eqref{eq_v_comparison2}, we see that, due the application of the extension result, also $m_K$ and $\eta_K$ appear on the right-hand side of \eqref{eq: control beta}.     In view of Proposition \ref{prop_flatness_decay}, {the question of decay for} $\beta_K$ transfers to understanding the decays of $\omega_u$, $m_K$, and $\eta_K$. This is  a delicate issue and we proceed in two steps. We first show a weaker estimate, namely that all quantities $\beta_K$, $\omega_u$, $m_K$, and $\eta_K$ remain small on smaller balls when they are already small in a big ball. This will allow us to show  $m_K(x_0,r) = \eta_K(x_0,r) = 0$ for sufficiently small radii $r>0$, i.e., on small balls $B(x_0,r)$ the set $K$ {separates} $B(x_0,r)$ and no wall sets are needed for the extension. Then, \eqref{eq: control beta} simplifies significantly, and in a second step we show that a decay on $\omega_u$ implies the desired decay of $\beta_K$.

        Let us formulate the relevant intermediate results, namely the controls on the length of the holes $\eta_K$, the decay of the bad mass $m_K$, and the decay of the  elastic \EEE energy $\omega_u$.

        \begin{proposition}[Filling the holes]\label{lem_F}
            Let $(u,K)$ be a Griffith almost-minimizer. Let  $\gamma \in (0,1]$ and  $\kappa \in (0,1]$. Then, there exists $\eps_{\rm hole} \in (0,1/100)$ only depending on $\gamma$, and  $C_{\rm hole}>0$ only depending on $\kappa$  and $\mathbb{C}$  such that the following holds:  for each  $x_0 \in K$ and  $r_0 > 0$  such that $B(x_0,r_0) \subset \Omega$, $h(r_0) \leq \varepsilon_{\rm Ahlf}$, and $\beta_K(x_0,r_0) \leq \eps_{\rm hole}$,   it holds that
            \begin{equation}\label{eqn:hole filling}        
                \eta_K(x_0, r_0) \leq    \kappa +  C_{\rm hole}      J_u(x_0,r_0)^{-1},
            \end{equation}
            and 
            \begin{equation}\label{eqn:hole filling-scale}        
                \eta_K(x_0, br_0) \leq   2b^{-1}\eta_K(x_0, r_0)   \quad \quad \text{for all $b \in [\gamma, 1]$} . 
            \end{equation}

        \end{proposition}

        The proof is given in Subsection \ref{sec:fillingHoles}. The main strategy of the proof is to show that we can construct a closed curve $\Psi$ which {separates} $B(x_0,r_0)$  and satisfies  $\mathcal{H}^1(\Psi \setminus K) \leq r_0 \left(\kappa  +  C_{\rm hole}      J_u(x_0,r_0)^{-1}\right)$, from which we may take $E (x_0,r_0):= (\Psi \cup K) \cap B(x_0,r_0)$ to conclude Proposition \ref{lem_F}. This construction fundamentally relies on the piecewise Korn-Poincar\'e inequality given in Proposition \ref{th: kornpoin-sharp-old} since the curve $\Psi$ will be constructed  from the boundary of the partition controlled in \eqref{eq: small set main}(i).

        \begin{proposition}[Control of the bad mass]\label{prop_badmass_decay}
            Let $(u,K)$ be a Griffith almost-minimizer. Let $\gamma \in (0,1]$. Then, there exists $\eps_{\rm mass}>0$ only depending on $\gamma$  such that the following holds: for  $x_0 \in K$ and  $r_0 > 0$ such that $B(x_0,r_0) \subset \Omega$, {$h(r_0) \leq \varepsilon_{\rm Ahlf}$}, and $\beta_{K}(x_0,r_0) + \eta_{K}(x_0,r_0) \leq \eps_{\rm mass}$,  we have 
            \begin{equation}\label{eq: mass scale-wo}
                m_K(x_0,r_0/4)\leq C_{\rm mass} \Big(\omega_u(x_0,r_0) + \beta_K(x_0,r_0) m_K(x_0,r_0) + \eta_K(x_0,r_0) + h(r_0)\Big),
            \end{equation}
            where $C_{\rm mass} \geq 1$ is a constant depending on $\mathbb{C}$, and it holds that 
            \begin{align}\label{eq: mass scale}
                m_K(x_0,  br_0)   \leq   C_{\rm mass}  b^{-1}\big(m_K(x_0, r_0/4)  +\eta_K(x_0, r_0) \big)  \quad \quad \text{for all $b \in [\tfrac{1}{4}\gamma,\tfrac{1}{4}]$} . 
            \end{align}
        \end{proposition}

        The result in proven in Subsection \ref{section_badmass}.  The main idea of the proof is the following. A ball contributes to the bad mass if \eqref{eq_goodball1} or \eqref{eq_goodball2} fails. In the first case, $K$ is not flat in $B(x,\sigma(x))$ and we can use a competitor as described in \eqref{eq: lambdastar} to estimate for how many balls this can occur. {On the other hand}, the quantity of balls for which \eqref{eq_goodball2} fails can directly be estimated in terms of $\eta_K$.

        \begin{proposition}[Decay of the elastic energy]\label{prop_energy_decay} 
            Let $(u,K)$ be a Griffith almost-minimizer. For all $0 < b \leq 1$, there exist constants  $C^b_{\rm el} \ge 1$ {and} $\eps_{\rm el} > 0$, {depending on $\mathbb{C}$ and $b$,} and a constant $C_{\rm el} \geq 1$ {depending on $\mathbb{C}$} such that the following holds: for each  $x_0 \in K$ and $ r_0 >0 $ such that $B(x_0,r_0) \subset \Omega$,  $h(r_0) \leq \varepsilon_{\rm Ahlf}$,  and  $\beta_K(x_0,r_0) \le \eps_{\rm el}$, we have
            \begin{equation}\label{eq_energy_decay0}
                \omega_u(x_0,b r_0) \leq C_{\rm el} \,  b \,  \omega_u(x_0,r_0) +C^b_{\rm el} \Big( m_K(x_0,r_0)\beta_K(x_0,r_0) + \eta_K(x_0,r_0) + h(r_0)\Big).
            \end{equation}
        \end{proposition}

        The proof is given in Subsection \ref{sec: energydecay}  and relies on a contradiction-compactness argument. \EEE Decay on $\omega_{u}$ is usually based on elliptic regularity, taking into account that  $u$ solves an elliptic system with a {Neumann boundary condition} on each side of $K$.  The challenge here lies in the fact that  a priori  no regularity of the set $K$ is known. Therefore, we do not obtain decay on $\omega_{u}$ directly, but the estimate also involves the geometry of $K$ in terms of $\beta_K$, $m_K$, and $\eta_K$.  In contrast to the corresponding proof in \cite{BIL} for connected $K$, the energy decay fundamentally relies on the extension result (Proposition \ref{th: extension}).

        The previous three results allow {us} to control the right-hand side of \eqref{eq: control beta}. However, a new issue appears since the estimate on $\eta_{K}$ in \eqref{eqn:hole filling}   contains the normalized jump $J_u$. Therefore, we also need the following control on $J_u$.

        \begin{proposition}[Control of the jump]\label{lem_jump}
            Let $(u,K)$ be a Griffith almost-minimizer. Let  $\lambda >0$ and $\gamma \in (0,1]$.   Then, there exist constants $\eps^\lambda_{\rm jump}>0$,  depending on $\mathbb{C}$ and $\lambda$, and $\eps^\gamma_{\rm jump}>0$  depending on $\gamma$ such that for each  $x_0 \in K$ and  $r_0 > 0$  with  $B(x_0,r_0) \subset \Omega$, $h(r_0) \leq \varepsilon_{\rm Ahlf}$,   $J_u(x_0,r_0) \ge \lambda$,   $\beta_K(x_0,r_0) \le \eps^\gamma_{\rm jump}$, and  $\omega_u(x_0,r_0) \leq \eps_{\rm jump}^\lambda$,   it holds that
            \begin{equation*}   
                J_u(x_0,br_0)  \ge  \frac{1}{2} b^{-1/2} J_u(x_0,r_0)  \quad \quad \text{for all $b \in [\gamma , 1]$.}  
            \end{equation*}
            Moreover, as $\lambda \to + \infty$ we can choose $\eps_{\rm jump}^\lambda \to \infty$.
        \end{proposition}
  The {above} result  shows that the quantity $J_u^{-1}$ is easier to handle compared to $\beta_K$, $m_K$, and $\omega_u$ since its decay does not depend on the other quantities. \EEE         The proof is based on estimating the difference of  rigid motions, cf.\ \eqref{eq: optimal values}, on balls of different size, see Subsection \ref{sec: the jump} below  for the proof. \EEE   

        Combining the previous propositions we can then prove the following.

        \begin{lemma}[Joint smallness of all  quantities  at all scales]\label{lem_decay}
            Let $(u,K)$ be a Griffith almost-minimizer.   Let  $x_0 \in K$ and  $r_0 > 0$ be such that $B(x_0,r_0) \subset \Omega$.
            For all $\varepsilon_1 >0$ there exists $\varepsilon_2 \leq \varepsilon_1$ {depending on $\mathbb{C}$} and $\varepsilon_1$ such that, if
            \begin{equation}\label{eq: thats the start}
                \beta_K(x_0,r_0)  +  \omega_u(x_0,r_0) + m_K(x_0,r_0) + \eta_K(x_0,r_0) +   J_u(x_0,r_0)^{-1} + h(r_0) \leq \varepsilon_2,
            \end{equation}
            then for all $0 < r \leq r_0$ it holds that 
            \begin{equation}\label{eq: thats the start2}
                \beta_K(x_0,r) + \omega_u(x_0,r) + m_K(x_0,r) + \eta_K(x_0,r) +    J_u(x_0,r)^{-1} + h(r) \leq  \varepsilon_1.
            \end{equation}
        \end{lemma}

        We refer to Subsection \ref{sec: joint smallness} for the proof.  Note that in  Theorem \ref{th: eps_reg}  we have an initial control only on $\beta_K$, $J_u$, and $h$, and not on all quantities in \eqref{eq: thats the start}. Using the above propositions, the other quantities  $\eta_K$, $\omega_u$, and $m_K$ can be initialized too, see Lemma \ref{lem_initialization} below for details.  In order to obtain the decay  not only around $x_0$ but in  balls centered in arbitrary points of $K \cap B(x_0,r_0)$ for $r_0$ small enough, as required  in Proposition \ref{lem_reifenberg_C1alpha}, we will also need the following   shifting properties.

        \begin{lemma}[Shifting properties]\label{lemma: scaling properties}
            Let $\eps_{\rm hole}$, $\eps_{\rm mass}$, $\eps_{\rm jump}^\lambda$, and $\eps_{\rm jump}^\gamma$ be the constants of Propositions \ref{lem_F}, \ref{prop_badmass_decay}, and \ref{lem_jump} applied for $\lambda=1$ and $\gamma=1/2$.   Consider  $x_0 \in K$ and  $r_0 > 0$ satisfying $B(x_0,r_0) \subset \Omega$, {$h(r_0) \leq \varepsilon_{\rm Ahlf}$} as well as $\beta_K(x_0,r_0) \leq \min\lbrace \eps_{\rm hole}, \eps_{\rm mass}{/2},  \eps^\gamma_{\rm jump}\rbrace$, $\eta_K(x_0,r_0) \leq \eps_{\rm mass}{/2}$,   $\omega_u(x_0,r_0) \leq \eps_{\rm jump}^\lambda$, and $J_u(x_0,r_0) \ge 1$. Then, for each   $x \in K \cap B(x_0,r_0/4)$ \EEE it holds that  
            \begin{align*}
                J_u(x,  r_0/2)^{-1}  &\le  C_{\rm shift}  J_u(x_0,r_0)^{-1}, \quad  \eta_K(x,  r_0/2)  \le  C_{\rm shift}   \eta_K(x_0,r_0),  \\
                m_K(x,  r_0/2) & \le C_{\rm shift}   \big( m_K(x_0,r_0) + \eta_K(x_0,r_0)   \big),  
            \end{align*}
            where $C_{\rm shift} \ge 4$ is a universal constant. 
        \end{lemma}

        The corresponding  properties for $\beta_K$ and  $\omega_u$ are direct consequences of  their respective \EEE definition, see Remark~\ref{rmk_beta}. For $\eta_K$, $m_K$, and $J_u$, instead, some more work  is  required. We will give the precise arguments in the subsequent subsections.

        Smallness of all quantities implies that $\eta_K$ and $m_K$  vanish, as the following result shows.

        \begin{lemma}[Vanishing $\eta$ and $m$]\label{lem_vanish}
            Let $(u,K)$ be a Griffith almost-minimizer.  Then,  there exists a {universal} constant $\eps_{\rm van} > 0$  such that  for  all  \EEE $x_0 \in K$ and  $r_0 > 0$  such that $B(x_0,r_0) \subset \Omega$, the condition
            \begin{align}\label{eq: another assumption}
                \beta_K(x,r) + \eta_K(x,r) \leq \eps_{\rm van} \quad \quad \text{for all $x \in K \cap B(x_0,r_0/2)$ and $0 < r \le r_0/2$}
            \end{align}
            implies  $m_K(x,r) = \eta_K(x,r) = 0$ for all $x \in K \cap B(x_0,r_0/20)$ and $0 < r \le r_0/20$.   
        \end{lemma} 

        The proof is given in Subsection \ref{sec: bila}. The idea is  as follows: we first show that \eqref{eq: another assumption} also yields a quantitative control on the bilateral flatness $\beta^{\rm bil}_K$. A simple geometric argument will then show that $K$ {must separate} $B(x,r)$, and thus $\eta_K(x,r)=0$. At this point,  given $\eps_{\rm van} \le \tau$, \EEE the fact that the bad mass vanishes is almost immediate from the definition, cf.\ \eqref{eq_goodball1}--\eqref{eq_goodball2}.

        With  Lemma \ref{lem_vanish} at hand, the decay of the flatness $\beta_K$ in \eqref{eq: control beta} is significantly simpler and we conclude the decay of $\beta_K$ by the decay on $\omega_u$ stated in Proposition \ref{prop_energy_decay}. Lemma \ref{lem_vanish} also shows that locally the separation property holds and therefore Lemma \ref{lem_bilateral_flatness} yields the equivalence of $\beta_K$ and $\beta_K^{\rm bil}$. This allows us to conclude the proof by Proposition  \ref{lem_reifenberg_C1alpha}. 

        After this overview of the proof, the presentation is now organized as follows:

        \begin{itemize}
            \item In Subsection \ref{sec: joint smallness}, we suppose that the auxiliary results (Lemmas \ref{lem_bilateral_flatness}, \ref{lemma: scaling properties}, \ref{lem_vanish}, and Propositions \ref{prop_flatness_decay}, \ref{lem_F},  \ref{prop_badmass_decay}, \ref{prop_energy_decay}, \ref{lem_jump}) hold true, and we first show the joint smallness (Lemma \ref{lem_decay}).  Then,  we show the main statement Theorem \ref{th: eps_reg}  by resorting to the Reifenberg parametrization theorem,  see Proposition \ref{lem_reifenberg_C1alpha}.
            \item In  Subsection \ref{sec: bila} we address the relation between flatness, bilateral flatness, and the separation property. In particular, we show   
                Lemma \ref{lem_bilateral_flatness} and Lemma \ref{lem_vanish}.
            \item In Subsection \ref{sec: the jump}, we control the jump by proving    Proposition \ref{lem_jump}.
            \item Subsequently, in Subsection \ref{sec:fillingHoles}, we address the construction how to fill holes and prove  Proposition \ref{lem_F}. 
            \item In Subsection \ref{sec: stop time}, we use the stopping time to construct a geometric function which will be needed for the extension result.  
            \item Subsection \ref{section_flatness} is devoted to the control of the flatness. Here, we show Proposition \ref{prop_flatness_decay}.
            \item  In Subsection \ref{section_badmass} we then collect properties of the bad mass and show Proposition \ref{prop_badmass_decay}.
            \item  Eventually, Subsection \ref{sec: energydecay} is devoted to the decay of the   elastic \EEE energy, i.e.,  Proposition \ref{prop_energy_decay}.
            \item In   Subsections   \ref{sec: the jump}, \ref{sec:fillingHoles},   and \ref{section_badmass}, we will also show the shifting properties stated in Lemma~\ref{lemma: scaling properties}. 
        \end{itemize}

        In the sequel, we  write $\beta$, $\beta^{\rm bil}$, $m$,   $\omega$, $J$, and $\eta$,  omitting the subscripts $K$ and $u$, respectively. {Throughout the whole proof, the letter $C$ is a generic constant larger than $1$  whose value may change from one line to another and which may depend on the tensor $\mathbb{C}$.}

        \subsection{Joint smallness  at all scales and  proof of the main result}\label{sec: joint smallness}

        In this subsection, we first prove Lemma \ref{lem_decay} and show how to initialize the quantities $m$, $\omega$, and $\eta$. Based on this, we give the proof of the main result Theorem \ref{th: eps_reg}.

        \begin{proof}[Proof of Lemma \ref{lem_decay}] 
            First, we fix a constant $b > 0$ sufficiently small such that 
            \begin{align}\label{eq: choice of b}
                b \leq 10^{-7}, \quad \quad  b \leq \frac{1}{4C_{\rm el}},
            \end{align}
            where $C_{\rm el}$ is the constant of Proposition \ref{prop_energy_decay}, {which depends only on $\mathbb{C}$}. Let $\varepsilon^* \in (0,1/2)$ {be} such that 
            \begin{align}\label{eq: vareps*}
                \varepsilon^* \le \min \Big\{ \eps_{\rm el}, \eps_{\rm hole}, \frac{1}{2}\eps_{\rm mass}, \frac{1}{2}\eps^\lambda_{\rm jump}, \frac{1}{2}  \eps^\gamma_{\rm jump},  \frac{b}{4C^2_{\rm mass}},  \frac{b}{16C^2_{\rm mass}C_{\rm el}^b}, \eps_{\rm flat}, \varepsilon_{\rm Ahlf}  \Big\}, 
            \end{align}
            where $\eps_{\rm el}$ and    $C^b_{\rm el}$ are the constants of Proposition~\ref{prop_energy_decay}, $\eps_{\rm mass}$  and  $C_{\rm mass}$ are the constants from Proposition~\ref{prop_badmass_decay},  $\eps_{\rm hole}$ is the constant from Proposition \ref{lem_F},  $\eps^\lambda_{\rm jump}$, $\eps^\gamma_{\rm jump}$ are the constants from Proposition \ref{lem_jump}, and $\eps_{\rm flat}$ is the constant from Proposition \ref{prop_flatness_decay}, all applied for $\gamma = b$ and $\lambda = 1$. Moreover, $\varepsilon_{\rm Ahlf}$ is given in \eqref{eqn:AhlforsReg}.   
          We emphasize that \EEE   {the constants $b$ and $\varepsilon^*$ only depend on $\mathbb{C}$.}

            \emph{Step 1: Single iteration.}    Fix any $\epsilon_1 \in (0,\varepsilon^*]$ to be specified at the end of the proof depending on $\eps_1$ {and $\mathbb{C}$}.  Along the proof, we will subsequently choose  $(\epsilon_j)_{2 \le j \le 5} \in (0,\epsilon_1]$, where each $\epsilon_j$ will {only depend on $\mathbb{C}$} and $(\epsilon_i)_{1 \le i < j}$, i.e., the parameters are chosen in the order   $\epsilon_1$, $\epsilon_2$, $\epsilon_3$, $\epsilon_4$, and $\epsilon_5$.  In particular, between $\epsilon_2$ and $\epsilon_3$ we suppose the relation
            \begin{align}\label{eq: 12}
                \epsilon_3 = \frac{\epsilon_2 b}{4C^2_{\rm mass}}.
            \end{align}
            The main part  of the proof consists in showing that the conditions
            \begin{equation}\label{eq: iteration start}
                \beta(x_0,r_0) \leq \epsilon_1, \ \   m(x_0,r_0)  \leq \epsilon_2,\ \  \omega(x_0,r_0)  \leq \epsilon_3,\ \  \eta(x_0,r_0) \leq \epsilon_4,  \ \  {J}(x_0,r_0)^{-1} + h(r_0) \leq \epsilon_5
            \end{equation}
            imply  
            \begin{equation}\label{eq_iterate}
                \beta(x_0,b r_0) \leq \epsilon_1, \ \  m(x_0,b r_0)  \leq \epsilon_2,\ \    \omega(x_0,b r_0)  \leq \epsilon_3,\ \  \eta(x_0,br_0) \le  \epsilon_4, \ \ {J}(x_0,br_0)^{-1} + h(br_0) \leq \epsilon_5.
            \end{equation}
            We start with the decay of the     bad mass.   By \eqref{eq: vareps*} and \eqref{eq: iteration start} we have  $\beta(x_0,r_0) + \eta(x_0,r_0) \leq \epsilon_1 + \epsilon_4 \le 2{\eps^*} \le \eps_{\rm mass}$ and $h(r_0) \leq \varepsilon_{\rm Ahlf}$.  Then we apply  Proposition \ref{prop_badmass_decay} and    by \eqref{eq: iteration start} we find
            \begin{equation*}
                m(x_0,b r_0) \leq C^2_{\rm mass} b^{-1} (\epsilon_3 + \epsilon_1 \epsilon_2 +2  \epsilon_4 + \epsilon_5) \le \frac{1}{4} \epsilon_2 + \frac{1}{4} \epsilon_2 + C^2_{\rm mass} b^{-1} (2\epsilon_4 + \epsilon_5),
            \end{equation*}
            where the second {inequality} follows from $\epsilon_3 = \epsilon_2 b/(4 C_{\rm mass}^2)$, see \eqref{eq: 12}, and {$\epsilon_1 \le \varepsilon^* \leq b/(4 C_{\rm mass}^2)$}, see \eqref{eq: vareps*}. We choose $\epsilon_j \le \epsilon_2 b/ (8 C^2_{\rm mass}) \EEE$ for $j = 4,5$, and  we find $ m(x_0,br_0) \leq \epsilon_2$.

            Next, we address the      normalized  {elastic} energy.  
            By \eqref{eq: vareps*} and \eqref{eq: iteration start} we have  $\beta(x_0,r_0) \leq \epsilon_1 \le \varepsilon^* \le \eps_{\rm el}$. Thus, we can  apply Proposition \ref{prop_energy_decay} to find, again using \eqref{eq: iteration start}, that
            $$  \omega(x_0,br_0) \leq C_{\rm el} b \epsilon_3 + C^b_{\rm el} (\epsilon_1 \epsilon_2 + \epsilon_4 + \epsilon_5). $$
            Then, using $b \leq 1/(4 C_{\rm el})$ from \eqref{eq: choice of b}, the relation $\epsilon_2 = 4 b^{-1} C_{\rm mass}^2 \epsilon_3$ given in \eqref{eq: 12}, and the fact that $\epsilon_1 \leq \varepsilon^* \leq b/(16 C_{\rm mass}^2 C_{\rm el}^b)$, see \eqref{eq: vareps*}, we obtain
            $$  \omega(x_0,br_0) \leq  \tfrac{1}{4} \epsilon_3 +  C^b_{\rm el} (  4 b^{-1} C^2_{\rm mass} \epsilon_1 \epsilon_3  + \epsilon_4 + \epsilon_5) \le   \tfrac{1}{2} \epsilon_3 +  C^b_{\rm el} ( \epsilon_4 + \epsilon_5) . $$      
            Then, by choosing   $\epsilon_j \le \epsilon_3/(4C_{\rm el}^b)$ for $j = 4,5$ we conclude  $\omega(x_0,br_0)  \le \epsilon_3$.

            Next, we establish a control on $\eta$.   By \eqref{eq: vareps*} and \eqref{eq: iteration start} we have  $\beta(x_0,r_0)  \le \epsilon_1 \le {\eps^*} \le  \eps_{\rm hole}$. Then, due to Proposition  \ref{lem_F}  and \EEE \eqref{eq: iteration start}, for given $\kappa >0$, we have 
            \begin{equation*}
                \eta(x_0, br_0) \leq 2 b^{-1} \big( \kappa +  C_{\rm hole}    \epsilon_5\big),
            \end{equation*}
            where $C_{\rm hole}  $ depends on $\kappa$ {and $\mathbb{C}$}.     Now, we choose $\kappa = b\epsilon_4/4$ and then $\epsilon_5$ depending on $\epsilon_4$ small enough such   that $2b^{-1}C_{\rm hole} \epsilon_5 \leq \epsilon_4/2$. Then, $\eta(x_0,br_0) \leq \epsilon_4$ follows.

            To control the normalized jump, we apply Proposition \ref{lem_jump} (for $\gamma = b$ and $\lambda = 1$) and use $b \le 1/4$, see \eqref{eq: choice of b}, to get
            $$J(x_0,br_0)^{-1}  \le  2 b^{1/2} J(x_0,r_0)^{-1} \le J(x_0,r_0)^{-1}.$$
            Note that  Proposition \ref{lem_jump}  is applicable since by \eqref{eq: vareps*} and \eqref{eq: iteration start} it holds that   $\beta(x_0,r_0) + \omega(x_0,r_0) \leq \epsilon_1 + \epsilon_3 \le 2  \eps^* \EEE \le \min\lbrace \eps^\lambda_{\rm jump}, \eps^\gamma_{\rm jump}\rbrace $  by \eqref{eq: iteration start}. \EEE This along with the fact that $h$ is {non-}decreasing gives $J(x_0,br_0)^{-1} + h(br_0) \leq \epsilon_5$.

            We finally deal with the flatness.     According to Proposition \ref{prop_flatness_decay}  (recall that  $b \leq 10^{-7}$ by \eqref{eq: choice of b} and ${\beta_K}(x_0,r_0) \leq \eps_{\rm flat}$ by \eqref{eq: vareps*} and \eqref{eq: iteration start}),  we have
            $$
            \beta(x_0, b r_0)^2 \leq C_{\rm flat}b^{-2} (\epsilon_3 + \epsilon_1 \epsilon_2 + \epsilon_4 + \epsilon_5).
            $$
            Thus, choosing  $\epsilon_j \le b^2 \epsilon_1^2 /(4C_{\rm flat})$ for $j = 2,\ldots,5$, we conclude $\beta(x_0, b r_0) \le \epsilon_1$.

            \emph{Step 2: Conclusion.} We are now ready to prove the statement.  Let $\eps_2 = \min_{1 \le j \le 5} \epsilon_j$, and suppose that  \eqref{eq: thats the start} holds. In particular, \eqref{eq: iteration start} is satisfied, and we can we iterate (\ref{eq_iterate}) to obtain   for all $n \geq 0$
            $$
            \beta(x_0,b^n r_0) \leq \epsilon_1, \  m(x_0,b^n r_0) \leq \epsilon_2,\   \omega(x_0,b^n r_0) \leq \epsilon_3, \ \eta(x_0,b^n r_0) \leq \epsilon_4, \  {J}(x_0,b^n r_0)^{-1} + h(b^nr_0) \leq \epsilon_5.
            $$ 
            In view of the scaling properties for $\beta$ and $\omega$ in Remark \ref{rmk_beta}, for $\eta$ and $m$ in \eqref{eqn:hole filling-scale}   and \eqref{eq: mass scale}, respectively, and for $J$ in Proposition \ref{lem_jump},  we deduce that for all $0 < r \leq r_0$,
            \begin{equation*}
                \beta(x_0,r) + m(x_0,r) + \omega(x_0,r) + \eta(x_0,r) + J(x_0,r)^{-1}  + h(r) \leq \hat{C} \epsilon_1
            \end{equation*}
            for some constant $\hat{C} \geq 1$ that depends only on $\mathbb{C}$. Here, we used that $\epsilon_1 = \max_{1 \le j \le 5} \epsilon_j$. Now, we eventually choose $\epsilon_1$ small enough such that $ \hat{C} \epsilon_1 = \eps_1$.
        \end{proof}

        In the main statement, we have an initial control only on $\beta$, $J$, and $h$. Therefore, in order to use the previous result, we need to initialize  $\eta$, $\omega$, and $m$.

        \begin{lemma}[Initialization]\label{lem_initialization}
            Let $x_0 \in K$ and  $r_0 > 0$ be such that $B(x_0,r_0) \subset \Omega$. For all $\varepsilon_3>0$ there exists $\varepsilon_4 \leq \varepsilon_3$ and $\theta \in (0,1)$ depending on $\mathbb{C}$  and $\varepsilon_3$  such that, if
            \begin{equation}\label{eq: again this assu}
                \beta(x_0,r_0) +  {J}(x_0,r_0)^{-1}  + h(r_0)\leq \eps_4,
            \end{equation}
            then
            \begin{equation}\label{eq: again this assuXXX}
                \beta(x_0, \theta r_0) + \omega(x_0, \theta r_0) + m(x_0, \theta r_0) + \eta(x_0,\theta r_0) +  {J}(x_0, \theta r_0)^{-1} + h( \theta r_0) \leq \eps_3.
            \end{equation}
        \end{lemma}
        \begin{proof}
            {Let $\epsilon_1{:= }\min(1,\eps_3/6)$.}
            We will show that we can choose  constants $0< \epsilon_3 \le \epsilon_2 \le \epsilon_1$, where $\epsilon_2$ {depends on $\mathbb{C}$}, $\epsilon_1$, and  $\epsilon_3$ {depends on $\mathbb{C}$}, $\epsilon_1$, $\epsilon_2$, such that if $\eps_4$ is small enough, depending on $\mathbb{C}$ and $(\epsilon_i)_{1 \leq i \leq 3}$, then
            \begin{align}\label{eq: on eta}
                \eta(x_0,  r_0) \leq \epsilon_3\epsilon_2,
            \end{align}
            \begin{align}\label{eq: on omega}
                \omega(x_0, br_0) \leq \epsilon_2,
            \end{align}
            \begin{align}\label{eq: on m}
                m(x_0,  br_0/4) \leq \epsilon_1,
            \end{align}
            where    $b :=\frac{1}{2}\epsilon_2 C_{\rm el}^{-1} {\bar{C}_{\rm Ahlf}}^{-1}$, \EEE with ${\bar{C}_{\rm Ahlf}}$ being the constant in \eqref{eqn:AhlforsReg2}, and $C_{\rm el}$ from \eqref{eq_energy_decay0}.
            Let $\varepsilon^* \in (0,1/2)$ be  such that 
            \begin{align}\label{eq: vareps*-new}
                \varepsilon^* \le \min \big\{ \eps_{\rm el}, \eps_{\rm hole}, \varepsilon_{\rm Ahlf}  \big\}, 
            \end{align}
            where $\eps_{\rm el}$  is the constant of Proposition  \ref{prop_energy_decay} applied for $b$  and      $\eps_{\rm hole}$  is the constant  in Proposition~\ref{lem_F} applied for {$\gamma=b$}, and $\varepsilon_{\rm Ahlf} $ is given in \eqref{eqn:AhlforsReg}.  We can suppose that $\eps_4 \le \eps^*$.

            We start with $\eta$. Note that  $\beta(x_0,r_0) \leq \eps_4 \le \eps^* \le   \eps_{\rm hole}$ and $h(r_0) \leq \varepsilon_{\rm Ahlf}$  by \eqref{eq: again this assu} and \eqref{eq: vareps*-new}. We apply Proposition~\ref{lem_F}    for $\kappa = \epsilon_2{\epsilon}_3/2$, and find
            \begin{equation*}
                \eta(x_0,r_0) \leq  \kappa + C_{\rm hole} J(x_0,r_0)^{-1}  = {\epsilon}_3\epsilon_2/2  + C_{\rm hole} J(x_0,r_0)^{-1},
            \end{equation*}
            where $C_{\rm hole}$ depends on {$\mathbb{C}$}, $\epsilon_2$ and $\epsilon_3$.  We  choose $\eps_4$  small depending on ${\epsilon}_2$ and ${\epsilon}_3$  such  that  \eqref{eq: again this assu} yields \eqref{eq: on eta}.

            We continue with $\omega$. Since $\beta(x_0,r_0) \le \eps_4 \le \eps^* \le \eps_{\rm el}$ by \eqref{eq: again this assu} and \eqref{eq: vareps*-new}, we can  apply  Proposition~\ref{prop_energy_decay} for   $b =\frac{1}{2}\epsilon_2 C_{\rm el}^{-1}{\bar{C}_{\rm Ahlf}}^{-1}$ \EEE  and get
            \begin{align*}
                \omega(x_0,br_0) &\leq C_{\rm el}  \, b  \, \omega(x_0,r_0) +C^b_{\rm el} \Big( m(x_0,r_0)\beta(x_0,r_0) + \eta(x_0,r_0) + h(r_0)\Big)\\
                                 & \le   \tfrac{1}{2}\epsilon_2 +C^b_{\rm el} \Big( m(x_0,r_0) \eps_4 + \epsilon_3\epsilon_2  + \eps_4\Big),
            \end{align*}
            where we used \eqref{eqn:AhlforsReg2},  \eqref{eq: again this assu}, and \eqref{eq: on eta}.
            {As $h(r_0) \leq \varepsilon_{\rm Ahlf}$, the density of $K$ in $B(x_0,r_0)$ is bounded from above  by  $C_{\rm Ahlf}$, see \eqref{eqn:AhlforsReg}, and since   $m(x_0,r_0) = r_0^{-1} \HH^1(K \cap R(x_0,r_0))$ with $R(x_0,r_0) \subset B(x_0,r_0)$  by \eqref{eq_defi_bad_mass},  we also have $m(x_0,r_0) \leq C_{\rm Ahlf}$.
}
            Therefore, choosing $\epsilon_3$ and $\eps_4$ small enough depending on {$\epsilon_2$}, $C_{\rm Ahlf}$ and  $C^b_{\rm el}$ (thus, effectively on $\epsilon_2$ and $\mathbb{C}$)  we get  \eqref{eq: on omega}.

            Now we control $m$.
             Since  $\beta(x_0,r_0) \leq \varepsilon^* \leq \varepsilon_{\rm hole}$, we can apply \eqref{eqn:hole filling-scale} to estimate $\eta(x_0,br_0) \leq 2 b^{-1} \eta(x_0,r_0) \leq  2b^{-1} \epsilon_2\epsilon_3 \EEE $.  In a similar fashion, we have  $\beta(x_0,br_0) \leq b^{-1} \beta(x_0,r_0) \leq b^{-1} \varepsilon_4$ by  Remark \ref{rmk_beta}.  As $b= \epsilon_2/\hat{C}$ for a constant $\hat{C} \ge 1$ depending only on   $\mathbb{C}$, we use also   \EEE   \eqref{eq: on omega},  \eqref{eq: again this assu}, and the monotonicity of $h$ to \EEE get that
            \begin{equation}\label{eq: the scale}
                \beta(x_0,br_0) + \omega(x_0,br_0) +  \eta(x_0,br_0)   + h(br_0) \leq    \hat{C}\epsilon_2^{-1} \eps_4 + \epsilon_2 + 2 \hat{C} \epsilon_3 + \eps_4  \le 2\epsilon_2,  \EEE
            \end{equation}
            where the second inequality follows by choosing $\epsilon_3$ and  $\eps_4$ small enough depending {on $\mathbb{C}$} and $\epsilon_2$. 
            \EEE
            It is not restrictive to assume that $\epsilon_2 \le \eps_{\rm mass}/  2 \EEE $ for the universal constant $\eps_{\rm mass}$ from Proposition \ref{prop_badmass_decay} applied for {$\gamma=1/4$}.    Since $\beta(x_0,br_0) + \eta(x_0,br_0) \leq  2 \EEE \epsilon_2$  by \eqref{eq: the scale}, \EEE we can apply  Proposition~\ref{prop_badmass_decay} in $B(x_0,b r_0)$ to estimate
            \begin{equation*}
                m(x_0, b r_0/4) \leq C_{\rm mass}\big(\omega(x_0,b r_0) + \beta(x_0,br_0)m(x_0,br_0) + \eta(x_0,br_0) + h(br_0)\big).
            \end{equation*}       
            Setting $\theta = b/4$ and using as before        $m(x_0,br_0) \le C_{\rm Ahlf} $, as well as \eqref{eq: the scale}, the previous estimate yields
            \begin{equation*}
                m(x_0, \theta r_0) \leq 3 C_{\rm mass} C_{\rm Ahlf} \epsilon_2 \leq \epsilon_1,
            \end{equation*}
            where the last step follows from choosing $\epsilon_2$ small enough compared to $\epsilon_1$.
            This achieves the proof of \eqref{eq: on m}.

            Having proved \eqref{eq: on eta}--\eqref{eq: on m}, we can now conclude.             We assume  that  $\varepsilon_4$  is small enough such  that $\beta(x_0,r_0) \leq \varepsilon_4$ in \eqref{eq: again this assu} and the scaling property in Remark \ref{rmk_beta}  imply   $\beta(x_0,br_0) \leq \varepsilon_{\rm hole}$, where  $\varepsilon_{\rm hole}$ is the constant of Proposition \ref{lem_F} applied for $\gamma = 1/4$.  Then, applying  Proposition \ref{lem_F} in $B(x_0,br_0)$, in particular the scaling property \eqref{eqn:hole filling-scale},  and using  \eqref{eq: on eta}, we estimate
            \begin{equation}\label{eq_init1}
                \eta(x_0,\theta r_0) \leq 8 \eta(x_0, b r_0) \leq  8  \epsilon_2.
            \end{equation}
            By the scaling property of Remark~\ref{rmk_beta}  and \eqref{eq: the scale}  we also get
            \begin{align}
                \omega(x_0,\theta r_0) &\leq 4 \omega(x_0,b r_0) \leq 8 \epsilon_2,\label{eq_init2}\\
                \beta(x_0,\theta r_0) &\leq 4 \beta(x_0,b r_0) \leq 8 \epsilon_2\label{eq_init3}.
            \end{align}
            It is left to initialize $J^{-1}(x_0,\theta r_0)$.
            As $h(r_0) \leq \varepsilon_{\rm Ahlf}$, we have $\omega(x_0,r_0) \leq \bar{C}_{\rm Ahlf}$ by \eqref{eqn:AhlforsReg2}.
            According to Proposition \ref{lem_jump}, we can find some $\lambda > 0$ for which the constant $\varepsilon^{\lambda}_{\rm jump}$   satisfies  $\varepsilon^{\lambda}_{\rm jump} \geq \bar{C}_{\rm Ahlf}$. For such a choice of $\lambda$ and for the choice $\gamma = \theta$, Proposition \ref{lem_jump} yields a constant $\varepsilon^{\gamma}_{\rm jump}$, only depending on $\gamma = \theta$ and thus on $\mathbb{C}$ and $\epsilon_2$, such that, if $J(x_0,r_0) \geq \lambda$ and $\beta(x_0,r_0) \leq \varepsilon^{\gamma}_{\rm jump}$, then
            \begin{equation*}
                J^{-1}(x_0,\theta r_0) \leq 2 \theta^{1/2} J^{-1}(x_0,r_0).
            \end{equation*}
             The conditions $J(x_0,r_0) \geq \lambda$ and $\beta(x_0,r_0) \leq \varepsilon^{\gamma}_{\rm jump}$ above hold true provided that $\varepsilon_4$ in \eqref{eq: again this assu} is once more chosen small enough. Then,   as $\theta = b/4 = \epsilon_2/(4\hat{C})$, by \eqref{eq: again this assu}  we have
            \begin{equation}\label{eq_init4}
                J^{-1}(x_0,\theta r_0) \leq C \epsilon_2^{1/2} \varepsilon_4 \leq C \epsilon_2^{1/2}
            \end{equation}
            for some constant $C \geq 1$ which depends only on $\mathbb{C}$. Now, we fix $\epsilon_2$ small enough   such that all   right-hand sides in \eqref{eq_init1}--\eqref{eq_init4} are  less than  $\epsilon_1$. We also assume $\varepsilon_4 \leq \epsilon_1$ so that the initial assumption \eqref{eq: again this assu} and the monotonicity of $h$ yields $h(\theta r_0) \leq h(r_0) \leq \epsilon_1$. Together with \eqref{eq: on m}, this concludes \eqref{eq: again this assuXXX} {recalling that $\epsilon_1 = \min(1,\eps_3/6)$.}
        \end{proof}

        \begin{proof}[Proof of Theorem \ref{th: eps_reg}]
            We start the proof by choosing some constants. First, let  $b$ be small enough (depending only on $\mathbb{C}$ and $\alpha$) such that $b \le 10^{-7}$ and  $C_{\rm el} b \leq b^\alpha/2$, where $C_{\rm el}$ is the constant in  Proposition~\ref{prop_energy_decay}. We remark that the second inequality in  possible due to  $\alpha < 1$.
            Choose $\eps_1   $ small enough such that 
            \begin{align}\label{eq: again quite small}
                \eps_1 \le \min\big\{ \eps_{\rm Ahlf},   \eps_{\rm van}, \eps_{\rm el},  \eps_{\rm flat}, \EEE \eps_{\rm hole}, \eps_{\rm mass}/2,  \eps_{\rm jump}^\lambda, \eps_{\rm jump}^\gamma, {\tfrac{b^{\alpha}}{2 C_{\rm el}^b}}, \tfrac{1}{100} \big\},
            \end{align}
            where  $\eps_{\rm Ahlf}$ is given in \eqref{eqn:AhlforsReg},   $\eps_{\rm van}$ is the constant  in  Lemma \ref{lem_vanish},     $\eps_{\rm el}$ and $C_{\rm el}^b$ are the constants  of Proposition \ref{prop_energy_decay} applied for $b$ fixed above, and   $\eps_{\rm flat}$, \EEE $\eps_{\rm hole}$, $\eps_{\rm mass}$, $\eps_{\rm jump}^\lambda$, and $\eps_{\rm jump}^\gamma$ are the constants of Propositions \ref{prop_flatness_decay}, \ref{lem_F}, \ref{prop_badmass_decay}, and \ref{lem_jump}, respectively,  applied for $\lambda=1$ and   $\gamma=1/2$.  

            Choose $\eps_2$ small enough such that estimate \eqref{eq: thats the start2} in Lemma \ref{lem_decay} holds for $\eps_1$. Then, let $\eps_3 = \min\lbrace  \eps_2/(2C_{\rm shift}),  \eps_1 \rbrace$, with $C_{\rm shift}$ from Lemma \ref{lemma: scaling properties},  and choose $\eps_4$ small enough such that \eqref{eq: again this assuXXX} in Lemma~\ref{lem_initialization} holds for $\eps_3$.  {Now we fix $\eps_0$ from the assumption in   \eqref{eq: smallli epsi} to be $\eps_0 : = \eps_4$.} {Thus, as} 
            \begin{equation*}
                \beta(x_0,r_0)  + {J}(x_0,r_0)^{-1}  + h(r_0) \leq \varepsilon_0,
            \end{equation*}
            the initialization in Lemma \ref{lem_initialization} yields
            \begin{equation}\label{needed!}
                \beta(x_0,\theta r_0) + \omega(x_0,\theta r_0) + m(x_0,\theta r_0) + \eta(x_0,\theta  r_0) + {J}(x_0,\theta r_0)^{-1} + h(\theta  r_0) \leq  {\varepsilon}_3,
            \end{equation}    
            where $\theta \in (0,1)$ {depends on $\mathbb{C}$} and $\varepsilon_3$, and thus {only on $\mathbb{C}$} and $\alpha$. 

            With these preparations, we can now start the proof. By Remark \ref{rmk_beta} and Lemma \ref{lemma: scaling properties} (for $\theta r_0$ in place of $r_0$) and the monotonicity of $h$, we find for all $x \in K \cap B(x_0,\theta r_0/ 4)\EEE$ that 
            \begin{equation*}
                \beta(x,\theta r_0/2) + \omega(x,\theta r_0/2) + m(x,\theta r_0/2) + \eta({x},\theta r_0/2) +  {J}({x},\theta r_0/2)^{-1} + h(\theta r_0/2) \leq  2  {C_{\rm shift}\eps_3 \leq \eps_2}.
            \end{equation*}
            Note that the assumptions of the lemma are satisfies by \eqref{eq: again quite small}, \eqref{needed!}, and ${\varepsilon}_3 \le {\varepsilon}_1$.          
            By   applying  Lemma~\ref{lem_decay} we deduce
            \begin{equation}\label{eq_thm_main_varepsilon0}
                \beta(x,r) + \omega(x,r) + m(x,r) + \eta(x,r) + {J}(x,r)^{-1} + h(r) \leq  \eps_1
            \end{equation}
            for all $x \in K \cap B(x_0, \bar{r}_0/ 4) \EEE$ and all $0 < r \leq  \bar{r}_0/2$, where for convenience we define $\bar{r}_0 = \theta r_0$.  In view of Lemma \ref{lem_vanish} and the fact that $\eps_1 \le \eps_{\rm van}$, we find 
            $m(x,r) = \eta(x,r) = 0$   for all $x \in K \cap B(x_0,\bar{r}_0/ 40) \EEE$ and $0 < r \le  \bar{r}_0/40\EEE$.

            Now, we are ready to show the decay of $\omega$.  Fix  $x \in K \cap B(x_0,\bar{r}_0/ 40) \EEE$.  For all $0 < r \leq  \bar{r}_0/40\EEE$, we have $\beta(x,r) \leq \eps_1 \leq \eps_{\rm el}$  by  \eqref{eq: again quite small} and \EEE \eqref{eq_thm_main_varepsilon0}, and we apply Proposition \ref{prop_energy_decay} in $B(x,r)$, which is now simplified significantly since $ m(x,r) = \eta(x,r) = 0$. We get that for all $0 < r \leq  \bar{r}_0/40\EEE$,
            \begin{align}\label{of the form}
                \omega(x,b r) \leq  C_{\rm el} b \omega(x,r) + C^b_{\rm el} h(r),
            \end{align}
            which, by our choice $C_{\rm el} b \leq b^{\alpha}/2$,  leads \EEE to
            \begin{equation}\label{eq_omega_decay}
                \omega(x,b r) \leq  \frac{b^\alpha}{2} \omega(x,r) + C^b_{\rm el} h(r).
            \end{equation}
            From the assumption $\varepsilon_1 \leq b^{\alpha}/(2 C_{\rm el}^b)$ and (\ref{eq_thm_main_varepsilon0}), we see that $h(\bar{r}_0/40) \leq b^{\alpha}/(2 C_{\rm el}^b)$ and, since $h$ decays  as a  power $\alpha$, we have for all $k \in \N$,
            \begin{equation}\label{eq_omega_decay-new}
                h\left(b^k \tfrac{\bar{r}_0}{40}\right) = h\left(\tfrac{\bar{r}_0}{40}\right) b^{k\alpha} \leq \frac{1}{2 C_{\rm el}^b} b^{(k+1)\alpha}.
            \end{equation}
            We deduce by induction on (\ref{eq_omega_decay}) that for all $k \in \N$ it holds that 
            \begin{equation}\label{eq_omega_induction}
                \omega\big(x,b^k \tfrac{\bar{r}_0}{40} \big) \leq b^{k \alpha}.
            \end{equation}
            Indeed, the case $k = 0$ directly follows from (\ref{eq_thm_main_varepsilon0}) as $\varepsilon_1 \leq 1$. If (\ref{eq_omega_induction}) holds at rank $k$, it also holds at rank $k+1$ because by \eqref{eq_omega_decay}--\eqref{eq_omega_decay-new} 
            \begin{equation*}
                \omega\big(x,b^{k+1} \tfrac{\bar{r}_0}{40} \big) \leq \frac{b^\alpha}{2} \omega\left(x,b^k \tfrac{\bar{r}_0}{40}\right) + C^b_{\rm el} h\left(b^k \tfrac{\bar{r}_0}{40}\right) \leq b^{(k+1)\alpha}.
            \end{equation*}
            Taking   Remark \ref{rmk_beta} into account, and using that  $h$ decays as a power $\alpha$,    we conclude for all $0 < r \leq \bar{r}_0/40$
            \begin{equation*}
                \omega(x,r) \leq C \left( \frac{r}{\bar{r}_0}\right)^\alpha, \quad \quad \quad      h(r) = h(\bar{r}_0) \left(\frac{r}{\bar{r}_0}\right)^{\alpha} \leq \left(\frac{r}{\bar{r}_0}\right)^{\alpha},
            \end{equation*}
            for some constant $C \geq 1$ which depends on $b$ and $\alpha$ (and thus only on $\mathbb{C}$ and $\alpha$).     Therefore,  applying Proposition \ref{prop_flatness_decay} for $b \le 10^{-7}$, we can estimate for all $x \in K \cap B(x_0,\bar{r}_0/  40) \EEE$ and all $0 < r \leq  \bar{r}_0/40 \EEE $
            \begin{equation}\label{eqn:betaDecayFor12}
                \beta(x,b r)^2 \leq C\big(\omega(x,r) + h(r)\big) \leq C \left(\frac{r}{\bar{r}_0}\right)^\alpha
            \end{equation}
            for a larger constant $C$.  (Note that the estimate is applicable due to \eqref{eq: again quite small} and \eqref{eq_thm_main_varepsilon0}.) \EEE
            For all $x \in K \cap B(x_0,\bar{r}_0/40)$ and $0 < r \leq \bar{r}_0/40$, we have shown above that  the set $K$ {separates} $B(x,r)$ and we can thus apply Lemma \ref{lem_bilateral_flatness} to find $\beta^{\rm bil}(x,r) \le 4\beta(x,r)$.
            Setting $ r_1 = b \bar{r}_0/40  = b \theta {r}_0/40 \EEE$, we conclude that for all $x \in B(x_0,r_1)$ and $0 < r \leq r_1$,
            \begin{equation}\nonumber
                \beta^{\rm bil}(x,r) \leq C \left(\frac{r}{r_1}\right)^{\alpha/2},
            \end{equation}
            for a still larger constant $C$. The theorem is finally an application of the Reifenberg parametrization theorem stated in Proposition~\ref{lem_reifenberg_C1alpha} (with exponent $\alpha/2$ instead of $\alpha$ and radius $r_1$ instead of $r_0$).    
        \end{proof}

        \begin{remark}\label{rmk_alpha}
            \normalfont
            In the case $\alpha = 1$, a decay property of the form  \eqref{of the form}  does not appear sufficient to prove that $\omega(x,r) \leq C r/r_0$. Hence, our method fails to reach the regularity $C^{1,1/2}$.

            In the absence of a gauge  ($h\equiv0$), \EEE we can however take advantage of the Euler--Lagrange equation to improve the regularity.
            We can first apply our epsilon-regularity Theorem \ref{th: eps_reg} with $\alpha = 1/2$ to obtain that $K$ a is $C^{1,1/4}$ graph in a smaller ball $B(x_0,\gamma r_0)$.
            
            The proof also shows that $m(x,r) = \eta(x,r) = 0$ for all balls $B(x,r) \subset B(x_0,\gamma r_0)$ with $x \in K$.
            Since $u$ solves an elliptic equation with Neumann boundary conditions on both sides of $K$, we can use Schauder estimates  to deduce that
            \begin{equation}\label{eqn:DecayEstimatU}
                \sup_{{B(x_0,\gamma r_0/2)}} |e(u)|^2\leq C \dashint_{B(x_0,\gamma r_0)} |e(u)|^2 \dd{x}.
            \end{equation}
            By \eqref{eqn:AhlforsReg2}, {if we restrict $x\in B(x_0,\gamma r_0/4)$ and $B(x,r) \subset B(x_0,\gamma r_0)$, the above estimate directly implies that} $\omega(x,r)\leq C r/r_0$ {for} $C>0$ depending {only} on $\bar{C}_{\rm Ahlf}$ and $\gamma$ (i.e., only on $\mathbb{C}$).
            Applying Lemma~\ref{lem_bilateral_flatness} and Proposition~\ref{prop_flatness_decay} in $B(x,r)$, and recalling that $m(x,r) = \eta(x,r) = 0$ with $h \equiv 0$, this shows that $\beta^{\rm bil}({x},10^{-7} r) \leq 4\beta({x}, 10^{-7} r) \le C  (r / r_0)^{1/2}$. An application of Proposition~\ref{lem_reifenberg_C1alpha} concludes the stated {$C^{1,1/2}$ regularity} in a smaller ball.
            
We note that the estimate \eqref{eqn:DecayEstimatU} can be proven by straightening $K$ via a change of variable and extending the solution on the other side by reflection (the coefficients of the equations are reflected accordingly). Then the problem {reduces to} an interior regularity {estimate} which can be dealt with {via a standard} `freezing the coefficients' trick.
    This procedure is detailed in the scalar case in \cite[Theorem 7.53]{Ambrosio-Fusco-Pallara:2000} but also works for general elliptic systems. In particular, it has been adapted to Lamé's equations in \cite[Theorem 3.18]{FFLM07}. 
    {There, the} authors prove that weak solutions of Lamé's equations with a Neumann boundary condition are Hölder differentiable up to the boundary but the proof also yields the above Schauder estimate (see \cite[Lemma A.1]{CL2} for the missing details).            
            
            \EEE

        \end{remark}

        \subsection{Flatness, bilateral flatness, and the separation property}\label{sec: bila}

        This subsection is devoted to the proofs of  Lemma \ref{lem_bilateral_flatness} and Lemma \ref{lem_vanish}.  We also recall the bilateral flatness $\beta^{\rm bil}$ introduced in \eqref{eq: bilat flat} and observe that the scaling properties in  Remark \ref{rmk_beta} also hold for $\beta^{\rm bil}$.  We proceed with two lemmas on the separation property.

        \begin{lemma}[Separation]\label{lem_reifenberg}
            Let  $x_0 \in \R^2$, \EEE $r_0 > 0$, and $0 < \varepsilon_0 < 1/100$ and let $K$ be a relatively closed subset of $B(x_0,r_0)$  containing $x_0$. \EEE We assume that for all $x \in K \cap B(x_0,r_0/2)$ and for all $0 < r \leq r_0/2$, we have $\beta^{\rm bil}(x, r) \le \varepsilon_0$. Then,  $K$ separates $B(x_0,r_0/10)$.
        \end{lemma}

        {The separation property above can be deduced from the $C^\alpha$-version of the Reifenberg parametrization theorem \cite{DPT, Reifenberg}.
        However, we present a short self-contained proof within the Appendix \ref{sec:lemmReif} for the sake of completeness.}

        \begin{lemma}[Separation in smaller balls]\label{lem_separation}
            Let $x_0 \in \R^2$, $r_0 > 0$, and $E$ be a relatively closed subset of $B(x_0,r_0)$ that contains $x_0$ and such that $E$ {separates} $B(x_0,r_0)$.
            \begin{enumerate}
                \item For all $x \in E \cap B(x_0,r_0)$ and $\gamma > 0$ such that $B(x,\gamma r_0) \subset B(x_0,r_0)$ and $\beta(x_0,r_0) \leq \gamma/4$, the set $E$ still {separates} $B(x,\gamma r_0)$.
                \item We assume that $\beta(x_0,r_0) \leq 1/100$. Let $x \in E \cap {B(x_0,9r_0/10)}$ and $r^x_0 \ge r_0/10$ be such that  $B(x,r^x_0) \subset B(x_0,r_0)$. Let $0 < r \leq r^x_0$ be such that   $\beta(x,t) \leq {1/8}$ for all $t \in [r,r^x_0]$. Then, $E$ still {separates} $B(x,r)$.
            \end{enumerate}
        \end{lemma}

        \begin{proof}
            We start with (1).
            Let $\ell_0$ be a line  passing through $x_0$ which achieves the minimum in the definition of $\beta(x_0,r_0)$, see \eqref{eq_beta2}. Set $\varepsilon_0 := \beta(x_0,r_0) \le \gamma /4$. 
            Let $\ell$ be the line parallel to $\ell_0$ and passing through $x$.
            Since $\dist(x,\ell_0) \le  \varepsilon_0 r_0$,   the parallel lines  $\ell_0$ and $\ell$ have distance smaller or equal to  $\varepsilon_0 r_0$. Then, as  $B(x,\gamma r_0) \subset B(x_0,r_0)$,   it follows that
            \begin{align*}
                {E} \cap B(x,\gamma r_0) &\subset \set{y \in B(x,\gamma r_0) \colon \,  \mathrm{dist}(y, \ell_0) \leq \varepsilon_0 r_0} \subset \set{y \in B(x,\gamma r_0) | \mathrm{dist}(y, \ell) \leq \varepsilon (\gamma r_0)},
            \end{align*}
            where $\varepsilon := 2 \varepsilon_0 / \gamma \leq 1/2$.
            Since the two connected components of $B(x,\gamma r_0) \setminus \set{\mathrm{dist}(\cdot, \ell) \leq \varepsilon (\gamma r_0)}$ are contained in {the} respective connected components of $B(x_0,r_0) \setminus \set{\mathrm{dist}(\cdot,\ell_0) \leq \varepsilon_0 r}$, we deduce that ${E}$ separates them as well. {It is straightforward to check that} all requirements in Definition \ref{defi_separation} are satisfied.

            Let us now show (2). Let    $x \in {B(x_0,9r_0/10)}$ and $r^x_0 \ge r_0/10$  be \EEE such that  $B(x,r^x_0) \subset B(x_0,r_0)$.  The fact that $\beta(x_0,r_0) \leq 1/100$ along with  (1)   yields that $E$ {separates} $B(x,r^x_0)$.  To deduce that $E$ {separates} $B(x,r)$ for given $r \in (0,r_0^x]$,  we iteratively check that $E$ {separates} $B(x_0,t)$ for all $r \le t \le r_0^x$.  To this end, we use the assumption  $\beta(x,t) \leq {1/8}$ for all $t \in [r,r^x_0]$, and   we may apply (1) with $\gamma =1/2$ $k$-times until $ r \in [2^{-k-1} r_0^x, 2^{-k} r_0^x]$ to conclude.    
        \end{proof}

        The next lemma states that we can bound the bilateral flatness by the flatness and the normalized size of holes in $K$.

        \begin{lemma}[Unilateral and bilateral flatness]\label{lem_bilateral_flatness-old}
              Let $x_0 \in \R^2$, $r_0 > 0$, and $K$ be a relatively closed subset of $B(x_0,r_0)$ that contains $x_0$. \EEE
            For any $\eps > 0$, the following   holds. 
            If $E$ is a relatively closed subset of $B(x_0,r_0)$ such that $ K \EEE  \subset E$ and $E$ {separates} $B(x_0,r_0)$,
            \begin{equation*}
                \beta_{E}(x_0,r_0) \leq \eps, \qquad r_0^{-1} \HH^1(E \setminus K) \leq \eps,
            \end{equation*}
            then we have
            \begin{equation*}
                \beta^{\rm bil}_K(x_0,r_0) \leq 4 \eps.
            \end{equation*}
        \end{lemma}
        \begin{proof}
            {First of all, observe that the statement is trivial for $\eps \geq 1/4$ because we always have $\beta^{\rm bil}_K(x_0,r_0) \leq 1$. Let us now assume $\eps \leq 1/4$.}
            In view of \eqref{eq_beta2},  we can consider a line $\ell$ passing through $x_0$ such that
            \begin{equation*}
                E \subset S := \set{x \in B(x_0,r_0) \colon \,  \mathrm{dist}(x,\ell) \leq \eps r_0}.
            \end{equation*}
            Recalling the definition in \eqref{eq: bilat flat},  it suffices to show that
            \begin{equation*}
                \ell \cap B(x_0,r_0) \subset \set{x \in B(x_0,r_0) \colon \, \mathrm{dist}(x,K) \leq 4 \eps r_0}.
            \end{equation*}
            For $y \in \ell \cap B(x_0,  (1 - 2\eps) r_0)$, we proceed by contradiction and assume that $B(y,2 \eps r_0)$ does not meet $K$. We observe that the connected components of $B(y,2\eps r_0) \setminus \set{\mathrm{dist}(\cdot,\ell) \geq \eps r_0}$ are contained in the respective connected components of $B(x_0,r_0) \setminus S$.
            We deduce that $E$ separates $B(y, 2\eps r_0)$ but since $E \subset S$, the set $E \cap B(y, 2 \eps r_0)$ must have a length of at least $2 \sqrt{3} \eps r_0$.
            But  this contradicts the  assumption $   \HH^1(E \setminus K) \leq \eps r_0$ and the fact that $B(y,2 \eps r_0) \cap K = \emptyset$.  
            For   $y \in \ell \cap (B(x_0, r_0) \setminus B( x_0,  (1 - 2\eps) r_0))$, we can find $z \in \ell \cap B(x_0, (1 - 2\eps)r_0)$ such that $\abs{y - z} \leq 2 \eps r_0$.   The result follows by noting that $B({z},2 \eps r_0) \cap K \neq \emptyset$, as shown above.
        \end{proof}

        In fact, {this} result can be understood as a generalization of Lemma \ref{lem_bilateral_flatness}. We now give its proof as a corollary of Lemma \ref{lem_bilateral_flatness-old} and afterwards we conclude the subsection  with the proof of   Lemma \ref{lem_vanish}.

        \begin{proof}[Proof of  Lemma \ref{lem_bilateral_flatness}]
            {The result  is a direct application of Lemma~\ref{lem_bilateral_flatness-old} for the choice $\eps = \beta_K(x_0,r_0)$ and $E = K$.}
        \end{proof}

        \begin{proof}[Proof of Lemma \ref{lem_vanish}]
            Let us start by showing $\eta(x,r) = 0$ for all $x \in K \cap B(x_0,r_0/20)$ and $0 < r \le r_0/20$.  Using the assumption \eqref{eq: another assumption}, Lemma \ref{lem_bilateral_flatness-old} shows that $\beta^{\rm bil}_{{K}}(x,r) \le 4\eps_{\rm van}$ for all $x \in K \cap B(x_0,r_0/2)$ and $0 < r \le r_0/2$. Then, provided $\eps_{\rm van}$ is chosen sufficiently small, Lemma~\ref{lem_reifenberg} implies that $K$ {separates} $B(x_0,r_0/10)$.
            {Choosing $\eps_{\rm van}$ again sufficiently small, it follows from \eqref{eq: another assumption} that $\beta_K(x_0,r_0/10) \leq 1/100$ and that $\beta_K(x,r) \leq 1/8$ for all $x \in K \cap B(x_0,r_0/20)$ and for all $r > 0$ such that $B(x,r) \subset B(x_0,r_0/10)$. We deduce by application of Lemma \ref{lem_separation}(2) that, for all $x \in K \cap B(x_0,r_0/20)$ and $0 < r \le r_0/20$, the set $K$ separates in $B(x,r)$ and thus $\eta(x,r) = 0$.}

            By possibly passing to a smaller $\eps_{\rm van}$, we can also assume that $\eps_{\rm van} \le \tau$, with $\tau$ as  defined in \eqref{eqn:constTau}.  Therefore, $\beta_K(x,r)\le  \tau$ for all $x \in K \cap B(x_0,r_0/2)$ and $0 < r \le r_0/2$. In view of \eqref{eq_goodball1}--\eqref{eq_goodball2} and   $\eta(x,r) = 0$ for all $x \in K \cap B(x_0,r_0/20)$ and $0 < r \le r_0/20$, this shows that $m(x,r)   = 0$ for all $x \in K \cap B(x_0,r_0/20)$ and $0 < r \le r_0/20$.  This concludes the proof.  
        \end{proof}

        \subsection{The jump}\label{sec: the jump}
        This subsection is devoted to the proof of  Proposition \ref{lem_jump} on the normalized jump $J$ introduced in \eqref{def:normalizedJump}. Moreover, we show the shifting property for $J$ in Lemma \ref{lemma: scaling properties}. Note that here we do not discuss the initialization of $J$ since $J^{-1}$ is {assumed to be small initially}, cf.\ \eqref{eq: smallli epsi}. The initialization of $J$, however, {will be relevant for our subsequent study on the size of the singular set \cite{FLS3},} and will require understanding how the jump relates to the  \textit{normalized $p$-elastic energy} defined  by
        \begin{equation}\label{eqn:pelastic}
            \omega_p(x_0,r_0) : = \left(r_0^{p/2} \dashint_{B(x_0,r_0)} |e(u)|^p\, {\rm d} x\right)^{2/p},
        \end{equation}
        for  $B(x_0,r_0) \subset \Omega$ and $p \in (1,2]$.   Specifically, Theorem \ref{th: eps_reg} will only rely on the behavior of the $2$-elastic energy (given by $\omega_2 \equiv \omega_u$), {but within \cite{FLS3}, we will rely on the more general notion $\omega_p$.} As it does not take us too far afield, in this subsection, some  estimates  are  already derived for general $p \in (1,2]$.   The exponent on the radius $r_0$ in (\ref{eqn:pelastic}) is chosen in such a way that $\omega_p$ is invariant under rescaling, as in Remark \ref{rem: normalization}. Note that for all balls $B(x,r) \subset B(x_0,r_0)$ we have
        \begin{equation}\label{eqn:naiveUpEst}
            \omega_p(x,r) \leq  {\left(\frac{r_0}{r}\right)^{4/p - 1}} \omega_p(x_0,r_0).
        \end{equation}
        {Now suppose $\beta(x_0,r_0) \leq 1/2$, and recall} the definition of the sets $D^\pm_{\beta(x_0,r_0)r_0}$ in \eqref{eq :Bx0}  (omitting $(x_0,r_0)$ in the notation) {along with} the fact that    $K \cap (D_{\beta(x_0,r_0)r_0}^+\cup D_{\beta(x_0,r_0)r_0}^-) = \emptyset$.  We define $A^\pm(x_0,r_0)$ and $ b^\pm(x_0,r_0)$ as in \eqref{eq: optimal values}.
         The \EEE corresponding  rigid motions are denoted by
        \begin{align}\label{eq: the rigid motions}
            a^\pm_{x_0,r_0}(y) = A^\pm(x_0,r_0)\,y + b^\pm(x_0,r_0) \quad \text{for $y \in \R^2$}. 
        \end{align} 
    When there is no confusion, we will drop the explicit dependence on $x_0$ and $r_0$ in the notation.}
    By Korn's and Poincar\'e's inequality applied on $D^\pm_{\beta(x_0,r_0)r_0}$ we find
    \begin{align}\label{eq: korns} 
        \int_{D_{\beta(x_0,r_0)r_0}^\pm}  \big(| \nabla u - A^\pm|^p + r_0^{-p} |u -a^\pm|^p \big) \, {\rm d}x   \le C\int_{B(x_0,r_0) \setminus K}  |e(u)|^p \, {\rm d} x =  C r_0^{2-p/2} \omega_p(x_0,r_0)^{p/2}  
    \end{align}
    for a constant $C = C(p)>0$, where we used {\eqref{eqn:pelastic}} in the last step. 
    Note  that \EEE the constant is indeed independent of {$\beta(x_0,r_0) \leq 1/2$} as the sets $D_{\beta(x_0,r_0)r_0}^\pm$ can be transformed to a half ball by a  Bilipschitz map with uniformly bounded derivative.
    We start by estimating the difference of $J$ on balls of different size. 

    \begin{lemma}[Balls of different size]\label{lem_aA}
        Let $p\in (4/3,2].$ Let $x_0 \in K$ and  $r_0 > 0$ be such that $B(x_0,r_0) \subset \Omega$ and   {$h(r_0) \leq \varepsilon_{\rm Ahlf}$.} Let $\gamma \in (0,1)$ be such that $\beta(x_0,t) \leq 1/8$ for all $t \in [\gamma r_0,r_0]$.
        Then, there exists a constant $C \geq 1$, depending on $p$ and $\mathbb{C}$, such that for all $r \in [\gamma r_0,r_0]$ we have
        \begin{equation}\label{eqn:bAControl}
            | a^\pm_{x_0,r_0}(y) -  a^\pm_{x_0,r}(y)|    \leq C r_0^{1/2}\omega_p(x_0,r_0)^{1/4} \quad \quad   \text{for all $y \in B(x_0,r)$}.   
        \end{equation}
    \end{lemma}

    \begin{proof}
         We first focus  on the special case   $r \in [r_0/2,r_0]$ and $\beta(x_0,r) \leq 1/8$, $\beta(x_0,r_0) \leq 1/8$. We apply \eqref{eq: korns}  in the balls $B(x_0,r_0)$ and   $B(x_0,r)$. It is elementary to check that  $D^\pm := D_{\beta(x_0,r)r}^\pm(x_0,r) \cap D_{\beta(x_0,r_0)r_0}^\pm(x_0,r_0)$ satisfies $\mathcal{L}^2(D^\pm) \ge cr_0^2$ for a universal constant $c >0$ that is independent of the two lines $\ell$    given in  \eqref{eq_beta}.  Using  \EEE \eqref{eq: korns} this shows 
            \begin{align*}
                \int_{D^\pm} \abs{a^\pm_{x_0,r} - a^\pm_{x_0,r_0}}^p \, \dd x &\leq 2^{p-1}\int_{D^\pm} \left(\abs{u - a^\pm_{x_0,r}}^p + \abs{u - a^\pm_{x_0,r_0}}^p\right) \dd{x}                                            \leq C r_0^{2+p/2} \omega_p(x_0,r_0)^{p/2}. 
            \end{align*}
            In particular,   Lemma \ref{lemma: rigid motion}   implies 
            \begin{equation}\label{eqn:affControl} | a^\pm_{x_0,r}(y) -  a^\pm_{x_0,r_0}(y)   | \le C   r_0^{1/2} \omega_p(x_0,r_0)^{1/2} \quad \text{for all $y \in B(x_0,r_0)$}.  
            \end{equation}
\EEE        Now, we  come \EEE to the general case $r \in [\gamma r_0,r_0]$ and we choose $k \in \N$ such that  $r \in [2^{-(k+1)}r_0, 2^{-k}r_0]$. To emphasize dependence on the radius, let us write $a^\pm(t;y) = a^\pm_{x_0,t}(y)$ for $t >0$ and $y \in \R^2$.
        Note by H\"older's inequality and \eqref{eqn:AhlforsReg2} that  we have the uniform bound 
        $\omega_p(x_0,t) \leq \bar C$ for all $t \in [0,r_0]$,   where $\bar C$ depends only on  $\mathbb{C}$. We apply the above formula (\ref{eqn:affControl}) to find for each $i=0,\ldots,k-1$ and for all $y \in B(x_0, 2^{-i}r_0)$  
        \begin{align*}
            |a^\pm(2^{-(i+1)}r_0;y) -   a^\pm(2^{-i}r_0;y)| &  \leq   C(2^{-i} r_0)^{1/2}    \omega_p\big(x_0,2^{-i} r_0\big)^{1/2}  \\
                                                            & = Cr_0^{1/2} (2^{-i})^{1/2}       \omega_p\big(x_0,2^{-i} r_0\big)^{1/4}    \omega_p\big(x_0,2^{-i} r_0\big)^{1/4} \\
                                                            & \leq  C r_0^{1/2}(2^{-i})^{\frac{3}{4} - \frac{1}{p}}   \omega_p(x_0,r_0)^{1/4},
        \end{align*} 
        where the last step follows from $\omega_p(x_0,2^{-i} r_0) \leq \bar C$ and {$\omega_p(x_0,2^{-i}r_0) \leq (2^{-i})^{1 - 4/p} \omega_p(x_0,r_0)$} by \eqref{eqn:naiveUpEst}.   In a similar fashion, we get 
        \begin{equation*}
            \abs{a^\pm(r;y) - a^\pm(2^{-k}r_0;y)} \leq C r_0^{1/2} (2^{-k})^{\frac{3}{4} - \frac{1}{p}}   \omega_p(x_0,r_0)^{1/4} \quad \text{for all $y \in B(x_0,r)$}.
        \end{equation*}
        Summation over $i$ yields 
        \begin{equation*}
            | a^\pm_{x_0,r_0}(y) -  a^\pm_{x_0,r}(y)   |  \leq C r_0^{1/2} \sum_{i=0}^\infty  (2^{-i})^{\frac{3}{4} - \frac{1}{p}}  \omega_p(x_0,r_0)^{1/4} \le Cr_0^{1/2} \omega_p(x_0,r_0)^{1/4}
        \end{equation*}
        for all $y \in B(x_0,r)$, where   $p>4/3$  ensures that the infinite sum is finite.     This concludes the proof.  
    \end{proof}

    \begin{remark}[Varying centers]\label{rem: vary}
        \normalfont
        An inspection of the proof shows that for given  $B(x_0,r_0) \subset \Omega$  and another ball $B(x,r)\subset B(x_0,r_0)$ with $x_0,x \in K$ and $\beta(x_0,r_0) \leq r_0/(16 r)$, we also have 
        $$  { | a^\pm_{x_0,r_0}(y) -  a^\pm_{x,r}(y)|    \leq C_*  r_0^{1/2}\omega_p(x_0,r_0)^{1/4} \quad \quad   \text{for all $y \in B(x,r)$},} $$
        where $C_*$  depends on $r_0/r$.
        The essential point is that $\beta \leq 1/8$ at all intermediate scales between $B(x,r)$ and $B(x_0,r_0)$.
        This guarantees that, when passing from a scale $t$ to a scale $t/2$, the sets $D_{\beta(x,t)t}^\pm(x,t) \cap D_{\beta(x,t/2)t/2}^\pm(x,t/2) $ have volume comparable to $t^2$,  and thus the difference of the rigid motions can be estimated.
        Using Lemma \ref{lemma: rigid motion} this also shows
        $$   | A^\pm_{x_0,r_0} -  A^\pm_{x,r}|    \leq C_*  r_0^{-1/2}\omega_p(x_0,r_0)^{1/4}. $$
    \end{remark}
    \EEE

    As  the above lemma was the only place where we will make use of further information on the $p$-elastic energy in \cite{FLS3}, we return to the case when $p$ is fixed to be $2$. We  proceed with the proofs of Proposition \ref{lem_jump} and Lemma \ref{lemma: scaling properties}.

    \begin{proof}[Proof of Proposition \ref{lem_jump}]
        Given two parameters $\lambda$ and $\gamma$, we suppose that $J(x_0,r_0) \ge \lambda$,   $\beta(x_0,r_0) \le \eps^\gamma_{\rm jump}$, and  $\omega(x_0,r_0) \leq \eps_{\rm jump}^\lambda$, where we    choose $\eps^\lambda_{\rm jump}$ and  $\eps^\gamma_{\rm jump}$  sufficiently small depending on $\lambda$,   $\gamma$ such that
        \begin{align}\label{eq: eps00-new}
            C (\eps^\lambda_{\rm jump})^{1/4} \le \frac{1}{4} \lambda , \quad \quad \eps^\gamma_{\rm jump} \le \frac{\gamma}{ 16 \EEE}. 
        \end{align}
        Here,  $C$ denotes the constant in \eqref{eqn:bAControl} for the choice $p=2$. Fix   $r \in [\gamma r_0,r_0]$. {By Remark \ref{rmk_beta} we have  $\beta(x_0,t) \le\frac{r_0}{t} \beta(x_0,r_0) \le \eps^\gamma_{\rm jump} / \gamma$  for all $t \in [r,r_0]$,   and  thus $\beta(x_0,t) \le 1/8$  by  \eqref{eq: eps00-new}.} We can then apply (\ref{eqn:bAControl}) for $r \in [\gamma r_0,r_0]$ and  find by \eqref{eq: eps00-new}
        \begin{equation}\label{eq: ttt}
            \abs{a^\pm_{x_0,r_0}(y) - a^\pm_{x_0,r  }(y) }   \le C r_0^{1/2} \omega(x_0,r_0)^{1/4} \le \frac{1}{4}  \lambda  \sqrt{r_0} \quad \text{for all $ y  \in K \cap B(x_0,r)$},
        \end{equation}
        where we used that $\omega(x_0,r_0) \le \eps^\lambda_{\rm jump}$ by assumption.  By the definition  of $J$ in  \eqref{def:normalizedJump} and $ J(x_0,r_0) \geq \lambda$   we have   $\abs{a^+_{x_0,r_0}(y)  - a^-_{x_0,r_0}(y)} \ge J(x_0,r_0) \sqrt{r_0} \ge \frac{1}{2}(J(x_0,r_0) + \lambda) \sqrt{r_0}$ for all $y \in K \cap B(x_0,r_0)$.  Thanks to \eqref{eq: ttt}, this   implies 
        $\abs{a^+_{x_0,r}(y)  - a^-_{x_0,r}(y)} \ge \frac{1}{2} J(x_0,r_0)  \sqrt{r_0}$ for all $y \in K \cap B(x_0,r)$ for $r \in [\gamma r_0,r_0]$. Therefore, for each  $r \in [\gamma r_0,r_0]$ we get 
        $$\sqrt{r}{J}(x_0,r) \ge  \frac{1}{2} J(x_0,r_0)  \sqrt{r_0},   $$    
        which concludes the proof. 
    \end{proof}

    \begin{proof}[Proof of Lemma \ref{lemma: scaling properties}]
        Fix $x \in K \cap B(x_0,  r_0/4) \EEE $. By the definition  of $J$ in  \eqref{def:normalizedJump} and $ J(x_0,r_0) \geq 1$ we have $\abs{a^+_{x_0,r_0}(y)  - a^-_{x_0,r_0}(y)} \ge  {J(x_0,r_0)} \sqrt{r_0} \ge \frac{1}{2}(J(x_0,r_0) + 1) \sqrt{r_0}$ for all $y \in K \cap B(x_0,r_0)$. We can assume that $\eps^\lambda_{\rm jump}$ in \eqref{eq: eps00-new} (for $\lambda = 1$) is chosen small enough such that also $C_* (\eps^\lambda_{\rm jump})^{1/4} \le 1/4$, where $C_*$ is the constant in Remark~\ref{rem: vary} {for $r := r_0/2$}. Then by  Remark~\ref{rem: vary}    we get that $\abs*{a^{\pm}_{x,r_0/2}(y) - a^{\pm}_{x,r_0}(y)} \leq  \frac{1}{4}\sqrt{r_0} \EEE$ for all $y \in B(x,r_0/2)$ and thus
        $|a^+_{x,r_0/2}(y)  - a^-_{x,r_0/2}(y)| \ge \frac{1}{2} J(x_0,r_0)  \sqrt{r_0}$ for $y \in K \cap B(x,r_0/2)$.
        {Note  that   $\beta(x_0,r_0) \leq r_0/(16r)$   is satisfied because $r_0/r = 2$ and $\beta(x_0,r_0) \le \eps_{\rm jump}^\gamma$, where $\eps^{\gamma}_{\rm jump}$ is defined in \eqref{eq: eps00-new} for $\gamma = 1/2$, see the proof of Proposition \ref{lem_jump}.} \EEE In view of  the definition of $J$, this yields $J(x,r_0/2) \ge J(x_0,r_0)/\sqrt{2}$ and concludes the proof.   
    \end{proof}

    \subsection{Minimal separating extensions and  construction to fill  holes}\label{sec:fillingHoles}

    In this subsection, we provide basic properties of minimal separating extensions  and prove a control  for $\eta_K$ defined in \eqref{eq: eta definition}. We start with two important auxiliary results, and afterwards we present the construction to fill holes in $K$, which is fundamental for the proof of Proposition \ref{lem_F}.

    \begin{lemma}[Ahlfors-regularity of extensions]\label{lem_AFF}
        {{Let $x_0 \in K$ and $r_0 > 0$} be such that $B(x_0,r_0) \subset \Omega$, $\beta_K(x_0,r_0) \leq 1/2$, and $h(r_0) \le \eps_{\rm Ahlf}$.}
        Let $E$ be a minimal separating extension of $K$ in $B(x_0,r_0)$. Then, the following holds:
        \begin{enumerate}
            \item For all $x \in E \setminus K$ and for all $r > 0$ such that $B(x,r) \cap K = \emptyset$ {and $B(x,r) \subset B(x_0,r_0)$,} we have
                \begin{equation}\label{eq_AFF0}
                    \HH^1(E \cap B(x,r)) \geq  {2 r.} 
                \end{equation}
            \item For all $x \in E$ and for all $r > 0$ such that $B(x,r) \subset B(x_0,r_0)$, we have
                \begin{equation*}
                    {C}^{-1}   r \leq \HH^1(E \cap B(x,r)) \leq C r,
                \end{equation*}
                for some constant $C \geq 1$ that only depends on the Ahlfors-regularity constant $C_{\rm Ahlf}$ of $K$.
        \end{enumerate}
    \end{lemma}

    The proof is rather standard and we include it {within Appendix \ref{sec:MinimalExtension}} for convenience of the reader.

    \begin{lemma}[Scaling properties of $\eta$]\label{lem_eta_scaling}
        {{Let $x_0 \in K$ and $r_0 > 0$} be such that $B(x_0,r_0) \subset \Omega$ and $\beta(x_0,r_0) \leq 1/100$.}
        Let $E$ be a minimal separating extension of $K$ in $B(x_0,r_0)$. Let $x \in E \cap {B(x_0,9r_0/10)}$ and $r^x_0 \ge r_0/10$ be such that  $B(x,r^x_0) \subset B(x_0,r_0)$. Let $0 < r \leq r^x_0$ be such that   $\beta_E(x,t) \leq {1/8}$ for all $t \in [r,r^x_0]$. Then, we have
        \begin{equation}\label{etaYYY} 
            \eta(x,r) \leq   \frac{2 r_0}{r}  \eta(x_0,r_0).
        \end{equation}
    \end{lemma}

    In particular, we observe that this lemma implies Lemma \ref{lemma: scaling properties} for $\eta$  (we choose \EEE  $\eps_{\rm hole} \le 1/100$).    For the proof, we use the following lemma on reducing the flatness of curves which will also be instrumental later.

    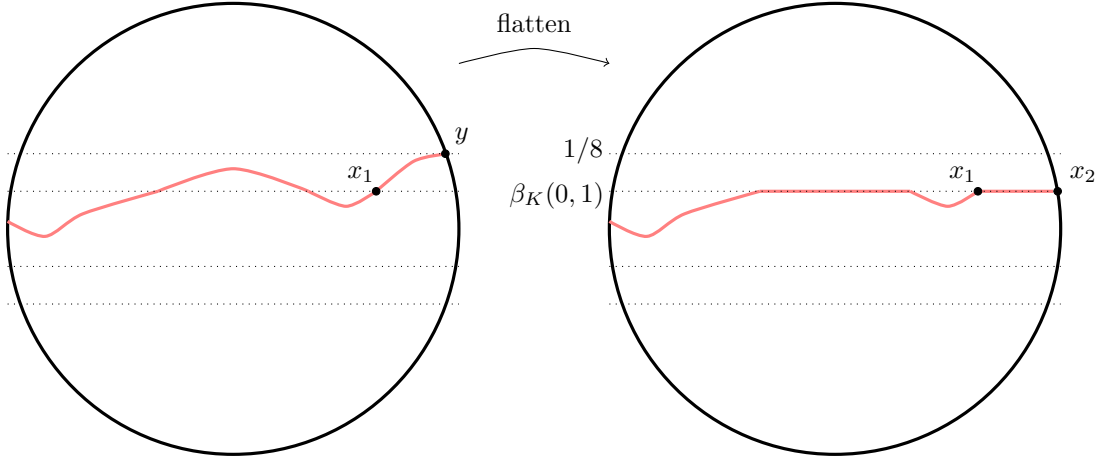
\begin{figure}
        \begin{tikzpicture}[x=1cm,y=1cm]
            \draw[very thick](0,0) circle (3.0);
            \draw [red,very thick] plot [smooth] coordinates {(-3.0,0.1)  (-2.5,-0.1) (-2,0.2) (-1,0.5)};
            \draw [red,very thick] plot [smooth] coordinates {(-1,0.5) (0,0.8) (1,0.5)};
            \draw [red,very thick] plot [smooth] coordinates {(1,0.5) (1.5,0.3) (1.9,0.5)};
            \draw [red,very thick] plot [smooth] coordinates { (1.9,0.5) (2.4,0.9) (2.85,1)};
            \draw[dotted] (-3,0.5) -- (3,0.5);
            \draw[dotted] (-3,-0.5) -- (3,-0.5);
            \draw[dotted] (-3,1) -- (3,1);
            \draw[dotted] (-3,-1) -- (3,-1);
            \begin{scope}[xshift=8cm]
                \draw[very thick](0,0) circle (3.0);
                \draw [red,very thick] plot [smooth] coordinates {(-3.0,0.1)  (-2.5,-0.1) (-2,0.2) (-1,0.5)};
                \draw [red,very thick] plot [smooth] coordinates {(-1,0.5) (1,0.5)};
                \draw [red,very thick] plot [smooth] coordinates {(1,0.5) (1.5,0.3) (1.9,0.5)};
                \draw [red,very thick] plot [smooth] coordinates { (1.9,0.5) (2.9,0.5)};
                \draw[dotted] (-3,0.5) -- (3,0.5);
                \draw[dotted] (-3,-0.5) -- (3,-0.5);
                \draw[dotted] (-3,1) -- (3,1);
                \draw[dotted] (-3,-1) -- (3,-1);
            \end{scope}
            \draw [->] plot [smooth] coordinates {(3.0,2.2) (4.0,2.4) (5.0,2.2)};
            \node[above] at (4.0,2.5) {flatten};
            \node[above left] at (2.0,0.5) {$x_1$};
            \draw[fill=black](1.9,0.5) circle(.05);
            \node[above left] at (10.0,0.5) {$x_1$};
            \draw[fill=black](9.9,0.5) circle(.05);
            \node[above] at (11.3,0.5) {$x_2$};
            \draw[fill=black](10.96,0.5) circle(.05);
            \node[above right] at (2.82,1.0) {$y$};
            \draw[fill=black](2.82,1.0) circle(.05);
            \node[below left] at (2.82,0.5) {$z$};
            \draw[fill=black](2.82,0.5) circle(.05);
            \draw[dotted](2.82,0.5) to (2.82,1.0);
            \node[above] at (4.65,0.75) {$1/8$};
            \node[above] at (4.3,0.15) {$\beta_K(0,1)$};
        \end{tikzpicture}
        \caption{We show that given a separating curve, we may make it flatter by potentially increasing the length. The possible increase in length only comes from the portions of the curve that meet the boundary of the ball.}
        \label{fig:curveFlatten}
    \end{figure}

    \begin{lemma}[Reducing flatness of curves]\label{lemma: replace curve}
        Let $K$ be a relatively closed set in $B(x_0,r_0)$  with $\beta_K(x_0,r_0) \le 1/8$, \EEE and let $\Gamma$ be a rectifiable curve which {separates} $B(x_0,r_0)$ with $\beta_\Gamma(x_0,r_0) \le 1/8$. Then, there exists a closed  rectifiable curve $\Psi$ separating in $B(x_0,r_0)$ such that $\Psi \cap D_{\beta_K(x_0,r_0)r_0}^\pm =\emptyset$   and 
        $$\mathcal{H}^1(\Psi \setminus K) \le 2 \mathcal{H}^1(\Gamma \setminus K). $$
    \end{lemma}

    \begin{proof}
        Without restriction we suppose $x_0 = 0$, $r_0=1$,  and  that an optimal line $\ell$ for $\beta_K (0,1)  $ in \eqref{eq_beta} {is given by $\ell = \mathbb{R}{\rm e}_1$}. The idea is to replace the curve $\Gamma$ by a straight segment whenever it leaves the set $\Pi:=\R \times [-\beta_K(0,1), \beta_K(0,1)]$. More precisely, we replace each connected component of  $\Gamma   \setminus  \Pi $   by a segment in $\R \times \lbrace -\beta_K(0,1)\rbrace $ or  $\R \times \lbrace \beta_K(0,1)\rbrace $, respectively. If this segment does not intersect $\partial B(0,1)$, the length of the curve is decreased by the triangle inequality. We now argue that segments intersecting $\partial B(0,1)$ are at most   a factor of $2$ longer,  {see Figure \ref{fig:curveFlatten}}.  We denote the horizontal segment by $\Lambda_1 = [x_1;x_2]$ with $x_2 \in \partial B(0,1)$ and compare its length to the part $\Lambda_2 \subset \Gamma$ connecting $x_1$ with a point $y \in \partial B(0,1)$ with $|{y \cdot e_2}| \le 1/8$. Denote the orthogonal projection of $y$ onto $[x_1;x_2]$ by $z$. Let $\varphi$ be the angle at $y$ in the triangle $x_1,z,y$, and $h = |y-z|$. Then, elementary geometric considerations yield
        $$\mathcal{H}^1(\Lambda_2) \ge h(\cos\varphi)^{-1}, \quad \mathcal{H}^1(\Lambda_1) = \mathcal{H}^1([x_1;z])  + \mathcal{H}^1([z;x_2]) \le h\tan\varphi + h \tan(\sin^{-1}(1/8)), $$
        where the estimate on  $\mathcal{H}^1([z;x_2])$ follows from $|y  \cdot e_2| \le 1/8$. Then, we calculate the factor
        $$ \mathcal{H}^1(\Lambda_1)  / \mathcal{H}^1(\Lambda_2) \le \tan\varphi \cos\varphi + \tfrac{1}{3\sqrt{7}}    \cos\varphi  \le 2 \quad \text{for all $\varphi \in [0,2\pi]$}.$$
    We repeat this procedure for each connected component,  possibly a countable number of times, which by a compactness argument leads to a modified curve which we denote by $\Psi$. Clearly, we have  $\Psi \cap D_{\beta_K(x_0,r_0)r_0}^\pm =\emptyset$  and that $\Psi$ {separates} $B(x_0,r_0)$. As all connected components of  $\Gamma   \setminus  \Pi $    do not intersect $K$  we get $\mathcal{H}^1( {\Psi} \setminus K    \big) \le  2\mathcal{H}^1\big( \Gamma \setminus K)$.
\end{proof}

\begin{proof}[Proof of Lemma \ref{lem_eta_scaling}]
    According to Lemma  \ref{lem_separation}, the set $E$ still {separates} $B(x,r)$   but the difficulty here is that we may not have $\beta_E(x,r) = \beta_K(x,r)$. We choose a curve $\Gamma \subset E \cap B(x,r)$ such that $\Gamma$  {separates} $B(x,r)$. For instance, one can consider the connected component $P$ of $B(x,r) \setminus E$ containing $D^+_{r/8}(x,r)$, {which is a set of finite perimeter by \cite[Proposition 3.62]{Ambrosio-Fusco-Pallara:2000},} and observe that $\partial^* P$ contains  a Jordan curve {that separates $B(x,r)$} due to the structure theorem of the boundary of planar sets of finite perimeter, see  \cite[Corollary 1]{Ambrosio-Morel}. {More precisely, the aforementioned curve can be found by taking the indecomposable component of $P$ containing $D^+_{r/8}(x,r)$, saturating this set \cite[Definition 5.2]{Ambrosio-Morel}, and then applying the structure theorem.} Taking $\Gamma$ to be the part of the Jordan curve inside of $B(x,r)$,   we apply Lemma \ref{lemma: replace curve} to {modify $\Gamma$ and} we  find $\Psi$ which {separates} $B(x,r)$ and satisfies $\Psi \cap D_{\beta_K(x,r)r}^\pm =\emptyset$.  Then, we see that $E' :=  (\Psi \cup K) \cap B(x,r)$ satisfies \eqref{eq: main assu}   on $B(x,r)$ and is thus a competitor for \eqref{eq: eta definition}. This implies
    $$\eta(x,r) \le \frac{1}{r} \mathcal{H}^1(E' \setminus K) = \frac{1}{r} \mathcal{H}^1(\Psi \setminus K) \le  \frac{2}{r} \mathcal{H}^1(\Gamma \setminus K) \le  \frac{2}{r} \mathcal{H}^1\big( E  \setminus K\big) =  2\frac{r_0}{r}    \eta(x_0,r_0),   $$
    and concludes the proof.  
\end{proof}

{We now turn to proving} Proposition \ref{lem_F}. Given a ball $B(x_0,r_0)$ and $\kappa \in (0,1]$, the idea is to construct a separating curve $\Psi$ with $\Psi \cap D_{\beta_K(x_0,r_0)r_0}^\pm =\emptyset$   and 
\begin{align*}
    {\mathcal{H}^1\big( {\Psi} \setminus K    \big)   \le   \left(\kappa  +  C(\kappa)J(x_0,r_0)^{-1}\right) r_0}
\end{align*}
such that the set $E =  (\Psi \cup K) \cap B(x_0,r_0)$ {separates} $B(x_0,r_0)$.  To this end, we first find a Caccioppoli partition with the {piecewise} Korn- Poincar\'e  inequality such that the length of the boundaries is controlled by a uniform constant and such that  there are two large pieces corresponding to above and below the crack, denoted $P^+$ and $P^-$ respectively.  From here, we jump into the proof of Theorem \ref{lem_F}. Roughly speaking, the idea is to choose $\Psi$ as a subset of $\partial^* P^+$. The rigorous argument groups elements of the partition together if their infinitesimal rotation coming from the Korn's inequality is sufficiently close or disconnects them if rotations are too different. This eventually allows us to pick up a small multiplicative factor $\kappa$ in our estimate at the cost of the inverse jump $J^{-1}$.

We start with the construction of a Caccioppoli partition of a ball $B(x_0,r_0)$,   relying on the piecewise  Korn-Poincar\'e  inequality stated in Proposition \ref{th: kornpoin-sharp-old}. The estimate  keeps track of the full length of the boundary for elements of the partition and will be the starting point to construct the separating curve.

\begin{lemma}[Caccioppoli partitions]\label{lemma: Cacc2}
    Let $(u,K)$ be a Griffith almost-minimizer. Let $x_0 \in K$ and  $r_0 > 0$ be such that $B(x_0,r_0) \subset \Omega$, $h(r_0) \le \eps_{\rm Ahlf}$,  and $\beta(x_0,r_0)\leq 1/8$. Then, there  exist a   constant $C>0$ \EEE  only depending on $\mathbb{C}$  and a Caccioppoli partition of $B(x_0,r_0)$ consisting of  sets $P^+$, $P^-$, and   $(P_j)_{j \geq 1}$,   and corresponding  rigid motions  $(a_j)_{j \ge 1}$,   such that 
    \begin{align}\label{eq: appl-korn3}
        {\rm (i)} & \ \  {\mathcal{H}^1\big(    \partial^* P^+  \big) + \mathcal{H}^1\big(    \partial^* P^-  \big)+}\sum\nolimits_{j\ge 1} \mathcal{H}^1\big(    \partial^* P_{j}  \big)  \le     C r_0 ,\notag\\
    {\rm (ii)} & \ \ \min\Big\{\mathcal{L}^2\big(P^+ \cap  D_{r_0/8}^+(x_0,r_0)\big), \mathcal{L}^2\big(P^- \cap  D_{r_0/8}^-(x_0,r_0)\big) \Big\} \ge  \frac{1}{2}\mathcal{L}^2\big(  D_{r_0/8}^+(x_0,r_0)\big), \EEE   \end{align}
    and 
    \begin{align}\label{eq: appl-korn5}
        {\rm (i)} & \ \  \Vert u - a_{j} \Vert_{L^\infty(P_{j})}  \le C\sqrt{r_0} \quad \quad \text{for all $j \ge 1$}, \notag\\
        {\rm (ii)} & \ \ \Vert u - a^+ \Vert_{L^\infty(P^+)}  \le C\sqrt{r_0}, \quad \Vert u - a^- \Vert_{L^\infty(P^-)}  \le C\sqrt{r_0},
    \end{align}
    where $a^\pm$ come from \eqref{eq: the rigid motions},
\end{lemma}

\begin{proof}
    In the proof, we treat the sets $D^{\pm}_{r_0/8}(x_0,r_0)$ and the strip $S_{x_0,r_0}: = B(x_0,r_0)\setminus (D^{+}_{r_0/8}(x_0,r_0) \cup D^{-}_{r_0/8}(x_0,r_0))$ separately. For notational convenience we write $D^\pm$ in place of   $D^{\pm}_{r_0/8}(x_0,r_0)$ and $S$ in place of $S_{x_0,r_0}$.

    In $S$, we directly apply the piecewise  Korn-Poincar\'e inequality of Proposition \ref{th: kornpoin-sharp-old} to $u$ with $J_u = K$ and find a Caccioppoli partition $(P^S_j)_{j\in \N}$ of $S$ such that there are rigid motions $(a^S_j)_j$  satisfying
    \begin{align}\label{eq: appl-korn7}
        {\rm (i)} \ \ & \sum\nolimits_{j} \mathcal{H}^1\big(\partial^* P^S_j \cap S \big) \leq C \left(\mathcal{H}^1(K\cap S) + \mathcal{H}^1(\partial S) \right) \leq Cr_0 , \notag \\
        {\rm (ii)} \ \ & \|u - \sum\nolimits_j  a^S_j  \chi_{P_j^S}\|_{L^\infty(S) } \leq  C_{\rm Korn} \|e(u)\|_{L^2(B(0,1))} \leq C\sqrt{r_0},
    \end{align}
    where $C$ depends on the constants in \eqref{eqn:AhlforsReg} and \eqref{eqn:AhlforsReg2}, and thus on $\mathbb{C}$.   

    We now address $D^\pm$. In  $D^\pm$,   we cannot simply apply the piecewise  Korn-Poincar\'e inequality, as we need to ensure that  there is a large element of the partition. For this, we recall (\ref{eq: korns}) for $p=2$,  which by H\"older's inequality yields 
    \begin{equation}\nonumber
        \int_{D^\pm} |\nabla u - A^\pm| \, {\rm d}x   \le Cr_0 \Big(\int_{D^\pm} |\nabla u - A^\pm|^2 \,  {\rm d}x  \Big)^{1/2} \leq Cr_0^{3/2},
    \end{equation}
    where we also used \eqref{eqn:AhlforsReg2}. Consequently, we can apply the piecewise  Poincar\'e inequality in $D^\pm$ for $u - A^\pm(\cdot)$ and a parameter $\rho >0$, see \cite[Theorem 2.3]{Friedrich:15-4}, to find constants $(b_j^\pm)_{j\in \N} \subset \R^2$ and a Caccioppoli partition $(P^\pm_j)_{j\in \N}$ with
    \begin{align}\label{eq: appl-korn8}
        {\rm (i)} \ \ & \sum\nolimits_{j  } \mathcal{H}^1(\partial^* P_j^\pm\cap D^\pm) \leq  \, r_0\rho,
        \notag \\
        {\rm (ii)} \ \ &
        \Vert u - \sum\nolimits_j a_j^\pm \chi_{P_j^\pm}\Vert_{L^\infty(D^\pm)}\leq  \frac{C}{r_0\rho} \int_{D^\pm} |\nabla u - A^\pm| \, {\rm d}x  \leq C_\rho\sqrt{r_0}, 
    \end{align}
    where $C_\rho>0$ depends on $\rho$, and  $a_j^\pm(y):  = A^\pm\,y + b_j^\pm$ for $y \in \R^2$. Here, we used that $K \cap D^\pm = \emptyset$.  By the relative isoperimetric inequality on $D^\pm$ we see that
    $${\min\big\{ \mathcal{L}^2 (P^\pm_i),   \mathcal{L}^2(D^\pm \setminus P^\pm_i)   \big\} \le C\big(\mathcal{H}^1(\partial^* P_i^\pm\cap D^\pm)\big)^2 \le Cr_0\rho \mathcal{H}^1(\partial^* P_i^\pm\cap D^\pm),} $$
    where the constant $C$ is invariant under rescaling of the domain and thus only depends on   $D^\pm_{1/8}(0,1)$. Now, if we had $\mathcal{L}^2( P^\pm _i)  \le  \mathcal{L}^2(D^\pm \setminus P^\pm_i)$ for all $i \in \N$, we would {have}
$${\mathcal{L}^2(D^\pm) = \sum\nolimits_{i \in \N} \mathcal{L}^2( P^\pm_i)  \le  Cr_0\rho  \sum\nolimits_{i \in \N} \mathcal{H}^1(\partial^* P_i^\pm\cap D^\pm\big) \le Cr_0^2\rho^2.  }$$
Thus, taking $\rho$ sufficiently small depending only on the domain $D^\pm_{1/8}(0,1)$, we obtain a contradiction. This in turn shows that there exists a unique component  $P^\pm \subset (P_i^\pm)_i$ with $\mathcal{L}^2(P^\pm \cap D^\pm) \geq \frac{1}{2}\mathcal{L}^2(D^\pm).$

Now, we define the Caccioppoli partition as the sets  $(P^S_j)_j$ and  $(P^\pm_i)_i \setminus \lbrace P^\pm \rbrace$, denoted by  $(P_j)_j$, as well as the two sets $P^+$, $P^-$. Then, \eqref{eq: appl-korn3}(i) follows from \eqref{eq: appl-korn7}(i), and \eqref{eq: appl-korn8}{(i)}. Moreover, \eqref{eq: appl-korn5}(i) follows from \eqref{eq: appl-korn7}(ii) and \eqref{eq: appl-korn8}(ii), where the rigid motions are denoted by $(a_j)_j$ accordingly. 

It remains to show \eqref{eq: appl-korn3}(ii) and \eqref{eq: appl-korn5}(ii). In fact, \eqref{eq: appl-korn3}(ii)  holds  by construction. \EEE To see \eqref{eq: appl-korn5}(ii),  we need to estimate the difference of $a^\pm$  given in  \eqref{eq: the rigid motions}    and the rigid motion $a^\pm_{j_0}$ related to the component $P^\pm = P^\pm_{j_0}$ for a suitable index $j_0$. We can use Lemma \ref{lemma: rigid motion}, the triangle inequality, as well as \eqref{eq: appl-korn8}(ii), \eqref{eq: korns},  and \eqref{eq: appl-korn3}(ii) \EEE to see that 
\begin{equation}\nonumber
    \begin{aligned}
        \|a_{j_0}^\pm - a^\pm\|_{L^\infty(B(x_0,r_0))} \leq & \frac{C}{r_0}\|a^\pm_{j_0} - a^\pm\|_{L^2(P_{j_0}^\pm)} \le  \frac{C}{r_0}\left(\|u - a^\pm_{j_0}\|_{L^2(P_{j_0}^\pm)} + \|u - a^\pm\|_{L^2( D^\pm)} \right) \leq C \sqrt{r_0}
    \end{aligned}
\end{equation}
for a constant depending on  $\bar{C}_{\rm Ahlf}$, see \eqref{eqn:AhlforsReg2}.
Consequently, we have $\|u - a^\pm\|_{L^\infty(P^{\pm})} \leq  C \sqrt{r_0}$.
\end{proof}

{We now complete the proof of} Proposition  \ref{lem_F}.  For this, we {improve the partition constructed} in Lemma~\ref{lemma: Cacc2}. As preparation, we recall the following lemma, see \cite[Lemma 4.4]{Solombrino} and  also \cite[Theorem~4.1]{Solombrino}.

\begin{lemma}\label{lemma: interface}
    Let  $K \subset \Omega$ be  relatively closed and $y \in LD(\Omega \setminus K;\R^2) \cap L^\infty(\Omega \setminus K;\R^2)$. Let  $T_1,T_2 \subset \Omega$ be sets of finite perimeter and $a_i = a_{A_i,b_i}$, $i=1,2$,  be rigid motions. Then,  there is a ball $B_{T_1,T_2} \subset \R^2$ with 
    \begin{align*}
        {\rm (i)} & \ \ {\rm diam}(B_{T_1,T_2} ) \le 4\, {\rm diam}(T_2) \,  \Vert a_1 - a_2 \Vert^{-1}_{L^\infty(T_2)} \, \sum\nolimits_{i=1,2}\Vert y - a_i \Vert_{L^\infty(T_i)}, \\ 
        {\rm (ii)} & \ \ \mathcal{H}^1\big( (\partial^* T_1 \cap \partial^* T_2) \setminus ( B_{T_1,T_2}  \cup K) \big) = 0.
    \end{align*}   
\end{lemma}

\begin{proof}[Proof of  Proposition   \ref{lem_F}]
    Our basic idea is to construct the curve $\Psi$ as a subset of $\partial^* P^+$ with $P^+$ as given in Lemma \ref{lemma: Cacc2}. However, to ensure that $\Psi \setminus K$ has small measure, we need to modify $P^+$ iteratively by adding certain components of $(P_j)_{j\ge {1}}$ if they share a common boundary with $P^+$ outside of $K$.  If this boundary is too small instead, we will remove certain  balls provided by Lemma \ref{lemma: interface}.

    Let $\kappa \in (0,1]$ be given.  Let $P^+$, $P^-$, and   $(P_j)_{j \geq 1}$  be the partition given by Lemma \ref{lemma: Cacc2}. It is not restrictive to assume that the sets $(P_j)_{j \geq 1}$ are indecomposable as otherwise we consider their indecomposable components.  In view of  Lemma \ref{lemma: diam} and \eqref{eq: appl-korn3}(i), we find a   universal constant $C_0>0$  such that 
    \begin{align}\label{eq: appl-korn7***}
        \sum\nolimits_{j \ge 1} {\rm diam}(P_j) \le \sum\nolimits_{j \ge 1}  \mathcal{H}^1(\partial^* P_j ) \le  C_0r_0 .
    \end{align}
    To construct the curve $\Psi$, we will modify the set   $P^+$. (In the same way, we could also start from the set $P^-$.)     By the structure theorem of Caccioppoli partitions, see Theorem \ref{th: local structure}, we know that
    \begin{align}\label{eq: inspect}
        \partial^* P^+   \cap B(x_0,r_0) \subset \partial^* P^- \cup \bigcup\nolimits_{j \ge 1} \partial^* P_j
    \end{align}
    up to a set of negligible $\mathcal{H}^1$-measure.

    \emph{Step 1: Index sets and rest set.}  We introduce a decomposition for the `small components' $(P_j)_{j \ge 1}$ according to the difference of  rigid motions as follows. We fix $\theta >0$ satisfying
    \begin{align}\label{eq: thet-chocie}
        \theta \le  \min \Big\{ \frac{1}{16 C_0}, \frac{1}{C^{1/2}} ,   {\sqrt{\pi c^*_D}, \EEE \frac{\kappa}{4(2C_D + 1 )}} \Big\}, 
    \end{align} 
    with the constant $C_0$ in \eqref{eq: appl-korn7***},  the constant $C$ from \eqref{eq: appl-korn5},  $c^*_D := \mathcal{L}^2(D^+_{1/8}(0,1))$, \EEE    and $C_D:= C_{D^+_{1/8}}$   the constant from Lemma \ref{lemma: maggi} {with $U = D^+_{1/8}(0,1)$ and  $\lambda = 3/4 \EEE $}. For $k \in \N$, we introduce the set of indices
    \begin{align}\label{eq: kornpoinsharp5}
&\mathcal{J}_0 = \big\{ j \ge 1\colon  \Vert a_j - a^+ \Vert_{L^\infty(P_j)} \le r_0^{1/2}  \theta^{-2} \big\}, \notag\\
& \mathcal{J}_k = \big\{ j \ge 1\colon  r_0^{1/2}    \theta^{-2k}< \Vert a_j - a^+ \Vert_{L^\infty(P_j)} \le  r_0^{1/2} \theta^{-2(k+1)} \big\},
    \end{align}
    where $a^+$ is the rigid motion associated to $P^+$.  In view of  \eqref{eq: appl-korn3}(i), we find some $L \in \N$, $1 \le L \le C/\theta$, such that $ \sum_{j \in \mathcal{J}_L} \mathcal{H}^1(\partial^* P_j )   \le \theta r_0$.  We define the \emph{rest set}
    $$R = \bigcup\nolimits_{j \in \mathcal{J}_L}  P_j  $$
    and note that
    \begin{align}\label{eq: measure-rest}
        \mathcal{H}^1\big(\partial^* R   \big) \le  \theta r_0 . 
    \end{align}
    \emph{Step 2: Adding of components.}  We define a first modification by 
    $$M_1 =   P^+ \cup  \bigcup\nolimits_{k=0}^{L-1}   \bigcup\nolimits_{j \in \mathcal{J}_k}  P_j .  $$
    Clearly, by \eqref{eq: appl-korn5} and the definition in \eqref{eq: kornpoinsharp5} we have
    \begin{align}\label{eq: appl-korn6}
        \Vert u - a^+ \Vert_{L^\infty(M_{1})}    \le   Cr_0^{1/2} + r_0^{1/2}   \theta^{-2L}\le  2r_0^{1/2} \theta^{-2L},
    \end{align}
    where the last step follows from the choice of $\theta$ and $L \ge 1$. By  \eqref{eq: inspect} and the construction, the boundary of $M_1$ satisfies {(up to a set of negligible $\mathcal{H}^1$-measure)}
    \begin{align}\label{eq: boundary-Q1}
        \partial^* M_1  \cap B(x_0,r_0)   \subset \big(   \partial^* M_1 \cap \partial^* P^-  \big) \cup  \Big( \partial^*M_1 \cap \bigcup\nolimits_{k > L}    \bigcup\nolimits_{j \in \mathcal{J}_k} \partial^* P_j    \Big)  \cup \big(  \partial^* R  \cap B(x_0,r_0)  \big).
    \end{align}
    In the above subset relation, the coincidence of $\partial^* M_1$ and $\partial^* P^-$ {will be controlled by the difference of the rotations associated with the large pieces $P^+$ and $P^-$, i.e., the normalized jump $J$}. We will be able to control the boundary coming from the rest set using \eqref{eq: measure-rest}. The middle term may be big, but in the next step, we will use the definition of $\mathcal{J}_k$ for $k > L$ to see that an element in the set has a  rigid motion \EEE far away from $a^+$, thereby allowing us to apply Lemma \ref{lemma: interface} to cover  $(\partial^* P_j)_{j \in \mathcal{J}_k}$ \EEE with small balls.

    \emph{Step 3: Removing balls}.  We define the set
    $$M_2 = M_1  \setminus \bigcup\nolimits_{k>L}  \bigcup\nolimits_{j \in \mathcal{J}_k} B_j ,   $$
    where $B_j$ denote the balls given by Lemma \ref{lemma: interface} applied for $T_1 = M_1$ and $T_2 = P_j$.  In view of Lemma~\ref{lemma: interface}, \eqref{eq: appl-korn5},  \eqref{eq: appl-korn7***}, \eqref{eq: kornpoinsharp5}, and   \eqref{eq: appl-korn6},  we get 
    \begin{align}\label{eq: balls1}
        \sum\nolimits_{k >L} \sum\nolimits_{j \in  \mathcal{J}_k } \mathcal{H}^1(\partial B_j)  \le \sum\nolimits_{j \ge 1}  4\,  {\rm diam}(P_j) \,  \big(r_0^{1/2} \theta^{-2(L+1)}\big)^{-1} \, 4r_0^{1/2} \theta^{-2L}    \le 16  C_0r_0  \theta^2 \le \theta r_0,
    \end{align}
    where the last step follows from the choice of $\theta$ in \eqref{eq: thet-chocie}. 
    Since $B_j \supset (\partial^* M_1 \cap  \partial^* P_j) \setminus K$ up to a set of $\mathcal{H}^1$-measure zero,  by  \eqref{eq: measure-rest}, \eqref{eq: boundary-Q1}, and \eqref{eq: balls1} we find
    \begin{align}\label{eq: appl-korn9}
& \mathcal{H}^1\Big(  \big(\partial^* M_2  \cap B(x_0,r_0)\big) \setminus \big(K \cup  (  \partial^* M_1 \cap \partial^* P^-  ) \big)   \Big) \notag \\ 
& \le   \mathcal{H}^1\big(\partial^* R \cap B(x_0,r_0)\big) + \sum\nolimits_{k >L} \sum\nolimits_{j \in  \mathcal{J}_k } \mathcal{H}^1(\partial B_j)    \le 2\theta r_0.
    \end{align}
    As before, we introduce the shorthand notation $D^\pm$ for $D^\pm_{r_0/8}(x_0,r_0)$.  Then, as $M_1 \supset P^+$, by \eqref{eq: appl-korn3}(ii) and \eqref{eq: balls1} we obtain 
    \begin{align}\label{eq: big vol}
        \mathcal{L}^2\big(M_2 \cap D^+ \big) \ge \mathcal{L}^2\big(P^+ \cap D^+\big) - \sum\nolimits_{k >L} \sum\nolimits_{j \in  \mathcal{J}_k } \mathcal{L}^2(B_j) \ge  \frac{c^*_D}{2} \EEE r_0^2 - {\frac{1}{4\pi}\theta^2}r_0^2 \ge  \frac{1}{4}c^*_D \EEE r_0^2 ,   
    \end{align}
    where the last inequality follows from the choice of $\theta$ in \eqref{eq: thet-chocie}.

    \emph{Step 4: Separation of big components.} We begin this step  by cutting out the intersection of $\partial^* M_1$ and {$\partial^* P^-$. Once} again by Lemma \ref{lemma: rigid motion}, recalling the definition of $J$ in \eqref{def:normalizedJump} we have 
    $$ r_0^{1/2}J(x_0,r_0)  \le   \big|(A^+-A^-)\, x_0  + (b^+-b^-)\big| \le  \Vert a^+ - a^- \Vert_{L^\infty(B(x_0,r_0))}  \le C\Vert a^+ - a^- \Vert_{L^\infty(M_1)}.$$ 
    Applying Lemma \ref{lemma: interface} for $T_1= P^-$ and $T_2 = M_1$,  by  the above estimate, \eqref{eq: appl-korn5}(ii), and (\ref{eq: appl-korn6}), we find a ball $B_*$ containing $  (\partial^* M_1 \cap \partial^* P^-)\setminus K  $  {(up to a set of negligible $\mathcal{H}^1$-measure)}  with
    \begin{equation}\label{eqn:JumpBall}
        \mathcal{H}^{1}(\partial B_*) \leq C_\theta r_0 r_0^{1/2} \Vert a^+ - a^- \Vert_{L^\infty(M_1)}^{-1}\leq C_\theta r_0 J(x_0,r_0)^{-1}
    \end{equation}
    for   $C_\theta >0$ depending on $\theta$.  If \EEE $(C_\theta J(x_0,r_0)^{-1})^2 >  c_D^* \EEE$,  the estimate (\ref{eqn:hole filling}) holds up to a  universal \EEE constant since $\eta$ is always bounded by $2$. Consequently, we suppose  that \EEE $(C_\theta J(x_0,r_0)^{-1})^2 \leq  c_D^*$. This yields  $\mathcal{L}^2(B_*) \le  \frac{c_D^*}{4 \pi } \EEE r_0^2$ and thus, letting  $M_3:= M_2 \setminus B_*$, we find by \eqref{eq: appl-korn3}(ii) and \eqref{eq: big vol} that 
    \begin{equation}\label{eqn:massM3}
        {\rm (i)} \ \ \mathcal{L}^2 \big(M_3 \cap D^+\big) \ge   \frac{c_D^*}{4} \EEE r_0^2,  \quad  \quad  {\rm (ii)} \ \  \mathcal{L}^2 \big(D^- \setminus M_3\big) \ge  \mathcal{L}^2 \big(  (P^-  \cap D^-) \EEE \setminus M_3\big) \ge  \frac{c_D^*}{2}, \EEE
    \end{equation} 
    as well as, thanks to \eqref{eq: appl-korn9} and \eqref{eqn:JumpBall}, 
    \begin{equation}\label{eq: appl-korn9c}
        \mathcal{H}^1\big( (\partial^* M_3 \cap B(x_0,r_0)) \setminus K\big) \leq 2\theta r_0 +  C_\theta  J(x_0,r_0)^{-1}r_0.
    \end{equation}
    We now fill in $M_3$ within $D^+$ and cutaway $D^-$. We let $M_4 := M_3 \cup D^+$. We apply Lemma \ref{lemma: maggi} for $ U = D^+_{1/8}(0,1))$, {$\lambda =  3/4 \EEE $}, and $r = r_0$ to each indecomposable component $\tilde{M}$ of $D^+ \setminus M_3$, which have sufficiently small volume by \eqref{eqn:massM3}(i)   to give option (ii). Specifically,  we have $\mathcal{H}^1(\partial^* \tilde M)\leq C_D\mathcal{H}^1(\partial^* \tilde M \cap D^+)$, where   $C_D  = C_{D_{1/8}^+}$. With this, we recall \eqref{eq: appl-korn9c} and $K\cap D^+ = \emptyset$ to find 
    $$ \mathcal{H}^1\big(  \partial^* ( {D^+} \setminus M_3  )  \big) \leq C_D \big(2\theta +  C_\theta J(x_0,r_0)^{-1}\big) r_0.$$
    As 
    $${\partial^* M_4  \subset \big( (\partial^* M_3 \cap (B(x_0,r_0)\setminus D^+)\big) \cup  \big(\partial^* ( {D^+} \setminus M_3) \cap (\partial D^+ \cap B(x_0,r_0))    \big),}$$
    we increase the measure of the boundary of $M_4$ by a term of size $C_D(2\theta +  C_\theta J(x_0,r_0)^{-1}) r_0$. Similarly, defining $M_5 : = M_4 \setminus D^-$, {using \eqref{eqn:massM3}(ii),} and repeating the above argument,  we  add an additional set of {length} $C_D({2}\theta +  C_\theta J(x_0,r_0)^{-1}) r_0$ on  $\partial D^- \cap B(x_0,r_0)$. More precisely, {there is a set $S$ such that}
    \begin{equation}\label{eqn:boundaryM4}
        \partial^*M_5 =  \Big(\partial^* M_3 \cap \big(B(x_0,r_0)\setminus (D^+ \cup D^-)\big) \Big) \cup S, \text{ with } \mathcal{H}^1(S) \leq 2C_D (2\theta +  C_\theta J(x_0,r_0)^{-1}) r_0.
    \end{equation}
    Let  $M_6$ be the indecomposable component of $M_5$ containing $D^+$. By  \eqref{eq: appl-korn9c} and \eqref{eqn:boundaryM4}, we have that 
    \begin{equation}\label{eqn:boundaryM5}
        \mathcal{H}^1\big(  (\partial^* M_6  \cap B(x_0,r_0))\setminus K    \big)\leq  (2C_D+1)(2\theta +  C_\theta J(x_0,r_0)^{-1}) r_0.
    \end{equation}

    \emph{Step 5: Construction of the curve}.   By the  structure theorem of the boundary of planar sets    of finite perimeter, see  \cite[Corollary 1]{Ambrosio-Morel}, we get that $\partial^* M_6$ contains a Jordan curve whose intersection with $B(x_0, r_0)$ is a curve $\hat{\Psi}$ that separates $B(x_0,r_0)$.  By \EEE \eqref{eqn:boundaryM5}  and the choice of $\theta$   in \eqref{eq: thet-chocie}, the curve $\hat{\Psi}$ \EEE satisfies
    \begin{align}\label{eq: alsmost last2}
        \mathcal{H}^1\big( \hat{\Psi} \setminus K    \big) \le   \tfrac{1}{2} \kappa  r_0  +  C_\theta r_0    J(x_0,r_0)^{-1}.
    \end{align}
    {As before, we remark that the aforementioned curve can be found by saturating the set $M_6$ (see \cite[Definition 5.2]{Ambrosio-Morel}) and then applying the structure theorem.}
    Eventually, we use Lemma \ref{lemma: replace curve}  to replace the curve $\hat{\Psi}$ by a closed  curve $\Psi$ which satisfies $\Psi \cap D_{\beta_K(x_0,r_0)r_0}^\pm =\emptyset$   and, in view of  \eqref{eq: alsmost last2},
    \begin{align*}
        \mathcal{H}^1\big( {\Psi} \setminus K    \big) &\le 2  \big( \tfrac{1}{2} \kappa     +  C_\theta    J(x_0,r_0)^{-1}\big) r_0\le   \left(  \kappa    +  C_\theta J(x_0,r_0)^{-1}\right) r_0.
    \end{align*}
    Let $E := (\Psi \cup K) \cap B(x_0,r_0)$. By construction we have $\beta_E(x_0,r_0) \le  \beta_K(x_0,r_0)$ and  that $E$ separates $B(x_0,r_0)$ in the sense of Definition \ref{defi_separation}.   Therefore, $E$ satisfies \eqref{eq: main assu} and the estimate (\ref{eqn:hole filling}) {since, in view of \eqref{eq: thet-chocie}, the} constant $C_\theta$ only depends on $\kappa$ and $\mathbb{C}$.   

    {To obtain the stated scaling within (\ref{eqn:hole filling-scale}) for general $b \in [\gamma,1]$, we use {estimate} \eqref{etaYYY},}  where we can assume  $\beta_E(x_0,br_0) \leq 1/8  $ for all $b \in [\gamma,1]$ by Remark \ref{rmk_beta}, provided that $\eps_{\rm hole}$ is chosen small enough depending on $\gamma$.  This concludes the proof. 
\end{proof}

\subsection{Stopping time and construction of a geometric function}\label{sec: stop time}

In the subsequent subsections, we want to apply the extension result in various ways to show Propositions  \ref{prop_flatness_decay},  \ref{prop_badmass_decay}, and \ref{prop_energy_decay}. To this end, we need geometric functions $\delta : E \to [0,r_0/4]$, see Definition \ref{def: geo func} defined on a separating extension $E$ of $K$.
The stopping time function $\sigma$ introduced in \eqref{defdx} already  looks like a geometric function except for the fact that it is not defined on $E$, but only on $K$, and it is not Lipschitz. We now use a Vitali covering argument to define a natural geometric function related to $\sigma$.

Let $x_0 \in K$ and $r_0 > 0$ be such that $B(x_0,r_0) \subset \Omega$.
In the following, we will work under the assumption  that 
\begin{equation}\label{eq_varepsilon00_flat}
    \beta(x_0,r_0) + \eta(x_0,r_0) \leq \bar{\eps}_{\rm flat}:=   \tau/(400 M),
\end{equation} 
where we recall  the {definitions} {$\tau = 10^{-10}$}  from \eqref{eqn:constTau} {and $M= 10^6$ from \eqref{eqn:constA}}.
In particular, the constant $\bar{\eps}_{\rm flat}$ is universal.
We let $E = E(x_0,r_0)$ denote a minimal separating extension of $K$, with respect to which the stopping time $\sigma$ is defined.

We observe that for $x \in K \cap B(x_0,9r_0/10)$ and for all $t \in (0,r_0/10)$,
\begin{equation}
    \beta(x,t) \leq 2 \frac{r_0}{t}  \beta(x_0,r_0), \quad t^{-1} \HH^1\big((E \cap B(x,t)) \setminus K \big) \leq \frac{r_0}{t} \eta(x_0,r_0),
\end{equation}
and one can deduce that
\begin{equation}\label{eq_ri_estimate}
    \sigma(x)  \leq 2 \tau^{-1} \big( \beta(x_0,r_0) + \eta(x_0,r_0) \big)   r_0.
\end{equation}
In view of \eqref{eq_varepsilon00_flat},  this guarantees that, for all $x \in K \cap B(x_0,9r_0/10)$, we have $\sigma(x) \leq r_0 / (200 M)$ and thus $B(x,10 M\sigma(x)) \subset \subset B(x_0,r_0)$. In particular, the definition of the bad set $R(x_0,r_0)$ in \eqref{eq: the R} simplifies to
\begin{equation*}
    R(x_0,r_0)=  \bigcup_{x \in \mathcal{R}(x_0,r_0)}  B\big(x, M   \sigma(x)\big).
\end{equation*}
\EEE
Applying Vitali's covering lemma,  we may find a countable  subfamily $(x^\sigma_i)_{i\in I} \subset \mathcal{R}(x_0,r_0)$, see \eqref{eq: the R},  such that balls $B_i^\sigma := B(x^\sigma_i,M\sigma(x^\sigma_i))$ with radii  $ r_i^\sigma:= M\sigma(x^\sigma_i)$ are pairwise disjoint, and  
\begin{equation}\label{def_Bi}
    \bigcup_{i\in I} B_i^\sigma \subset R(x_0,r_0) \subset \bigcup_{i \in I} 4B_i^\sigma,
\end{equation}
where for $N \in \N$, $N \, B_i^\sigma:= B(x_i,NM\sigma(x^\sigma_i))$.
Note that by   $\sigma(x) \leq r_0 / (200 M)$ we get \EEE
\begin{equation}\label{eq: in the next section}
    { r_i^\sigma   \leq r_0/200 \quad \quad \quad \text{and} \quad \quad \quad  {10 B^\sigma_i}  \subset \subset B(x_0,  r_0).}
\end{equation}
Note that the definition of the covering clearly depends on $(x_0,r_0)$ which we do not include in the notation. We  however \EEE include $\sigma$ in the notation to highlight the dependence of the covering on the stopping time.

\begin{lemma}[Geometric function]\label{lem_delta}
    Let $(u,K)$ be a Griffith almost-minimizer. 
    Let $x_0 \in K$ and  $r_0 > 0$ be such that $B(x_0,r_0) \subset \Omega$, and  $\beta(x_0,r_0) + \eta(x_0,r_0) \leq  {\bar{\eps}_{\rm flat}}$. Let $E$ be a minimal separating extension {of $K$ in $B(x_0,r_0)$}. Then, the function $\delta \colon E \cap \overline{B(x_0,3r_0/4)} \to [0,+\infty) $ defined by
    \begin{equation}\label{eq_delta}
        \delta(x) := \min \Big\{ \inf_{i \in I} \big( |x - x^\sigma_i| + r^\sigma_i\big) , \mathrm{dist}\big(x, \R^2 \setminus \bigcup\nolimits_{i\in I} 9 B^\sigma_i\big) \Big\}
    \end{equation}    
    is $1$-Lipschitz and satisfies 
    \begin{enumerate}
        \item  $\delta(x^\sigma_i) = r^\sigma_i$ for all $i \in I$;
        \item for all $x \in E \cap \overline{B(x_0,3r_0/4)}$, we have  $\delta(x) \leq \max_{i \in I} (10 r^\sigma_i) \leq { r_0/4}$;
        \item for all $x \in E \cap \overline{B(x_0,3r_0/4)}$ and for all $r \in (\delta(x),r_0/4]$, we have $\beta_E(x,r) \leq 16 \tau$.
    \end{enumerate}
    In particular,  $\delta$ is a geometric function for $E$ with parameters $(3r_0/4, 16\tau)$, see Definition \ref{def: geo func}.
\end{lemma}

The lemma is an adaptation of \cite[Lemma 5.1]{CL}. Note that we cannot directly use this lemma due to the minimal separating extension $E$ which can {differ} from $K$. Properties (2) and (3) ensure that $\delta$ is a  geometric function whereas (1) reflects the idea the $\delta$ is related to the stopping time $\sigma$, up to multiplication with the big constant $M$. By \eqref{eq_varepsilon00_flat} and $8\bar{\eps}_{\rm flat} \le 16\tau \le 10^{-8}$ we will later be able     to  apply the extension result Proposition \ref{th: extension} {(with $16 \tau$ in place of $\tau$)}.  In this context, let us mention that  in some constructions we will use variants of the geometric function $\delta$ {which maintain} the crucial properties stated in Lemma \ref{lem_delta}.  

{Before we come to the proof, we state and prove a useful property: the stopping time function of $K$ also controls the flatness of $E$.
In the scalar case, this property \cite[Definition 2.6, ``Property *'']{Lemenant} resulted from the construction of a separating extension by the coarea formula and a clever choice of level set. We provide an alternative argument based on the fact that minimal separating extensions satisfy a lower density bound.}

\begin{lemma}\label{eq_betaF}
    Suppose that $\beta_K(x_0,r_0) \le \tau/10$ and that  $E$ is a minimal separating extension in $B(x_0,r_0)$. Then, for each    $x \in K \cap {B(x_0,9r_0/10)}$ and $t \in (\sigma(x),r_0/10 ]  $, we have
    \begin{equation*}
        {\beta_K(x,t) \leq \tau \quad \text{and} \quad \beta_E(x,t)  \leq  4\tau.}
    \end{equation*}
\end{lemma}

\begin{proof}
    Let $x \in K \cap B(x_0,9r_0/10)$ and $t \in (\sigma(x),r_0/10]$.
    The fact that $\beta_K(x,t)\leq \tau$ holds  follows directly \EEE from the definition of $\sigma$. We prove the second inequality.
    If $t \geq r_0/20$, we can apply the scaling properties {of the flatness within Remark \ref{rmk_beta} and} $\beta_K(x_0,r_0) {= \beta_E(x_0,r_0)}   \le \tau/10$ {to} bound directly
    \begin{equation*}
        \beta_{E}(x,t) \leq   \frac{2 r_0}{t}  \beta_{E}(x_0,r_0) {\leq 40 \beta_E(x_0,r_0)} \leq 4 \tau.
    \end{equation*}
    In the case $\sigma(x) < t \leq r_0/20$, we have $2t \in (\sigma(x),r_0/10]$, i.e., $B(x,2t)$ is a good ball.
    {For $y \in E \cap B(x,t)$,    let  us estimate the distance of $y$ to $K$. Whenever $\rho > 0$ is such that $B(y,\rho) \cap K = \emptyset$, we have $\rho < t$ (because $\abs{x - y} < t$ and $x \in K)$ so $B(y,\rho) \subset B(x,2t)$ and then by Lemma \ref{lem_AFF}{(1) and} \eqref{eq_goodball2} {we have}  
        \begin{align*}
             2 \EEE   \rho \leq \HH^1\big(E \cap B(y,\rho)\big) \leq \HH^1\big((E \setminus K) \cap B(x,2t)\big) \leq 2 t \tau.
        \end{align*}
    This implies  $B(y,\rho) \cap K \ne \emptyset$  for each $\rho > t \tau$,   and it follows that $\mathrm{dist}(y,K) \leq t \tau$. Since $y \in B(x,t)$ and $\tau \leq 10^{-10}$, the distance $\mathrm{dist}(y,K)$ is achieved at a point of $K \cap B(x,2t)$.}
    At the same time, $\beta_K(x,2t) \leq \tau$ {by} \eqref{eq_goodball1}. Thus,  $K \cap B(x,2t)$  has distance less or equal to $t \tau$ from a line through $x$ and then  $y$ has distance less or equal to $2 t \tau$ from such a line \EEE which does not depend on $y$. Since $y \in E \cap B(x,t)$ was arbitrary,   this proves that $\beta_E(x,t) \leq 2 \tau$, {and we conclude by going up to $\beta_E(x,t) \leq 4 \tau$ for later convenience.}
\end{proof}

\begin{remark}\label{rmk_betaF}
    {\normalfont
        In view of the above  lemma \EEE and Lemma \ref{lem_separation}(2), we see that for all $x \in K \cap {B(x_0,9r_0/10)}$ and $t \in (\sigma(x),r_0/10)$, the set $E$ also {separates} $B(x,t). $ Then, by $\beta_E(x,t)  \leq  4\tau$ and  Lemma \ref{lem_bilateral_flatness-old} we get  that $\beta_K^{\rm bil}(x,t)  \leq 16 \tau$.
        Hence, the bad mass \eqref{eq_defi_bad_mass} effectively weighs {the portion of the crack failing to be Reifenberg-flat with parameter $16\tau$, see the definition preceding Proposition \ref{lem_reifenberg_C1alpha}.}}
\end{remark}

\begin{proof}[Proof of Lemma \ref{lem_delta}]
    {We omit the superscript $\sigma$ for simplicity and write $B_k$ and $x_k$,  $r_k$ (the center and radius of $B_k$) for $k \in I$. Here, we recall that $r_k = M \sigma(x_k)$.}

    (1) We readily see that $\delta$ is $1$-Lipschitz and that $\delta(x_k) \leq r_k$. To {show} $\delta(x_k) = r_k$, {one uses} the fact that the balls $(B_i)_{i \in I}$ are disjoint, i.e., for all $i \ne k$ we have 
    $$
    \abs{x_i - x_k} + r_i \geq \abs{x_i - x_k} \geq r_k, \quad \quad 
    {\mathrm{dist}\Big(x_k, \R^2 \setminus \bigcup\nolimits_{i \in I} 9 B_i\big) \geq \mathrm{dist}(x_k, \R^2 \setminus 9 B_k) \geq 9 r_k.}
    $$

    (2)    Next, we justify that for all $x \in E \cap \overline{B(x_0,3r_0/4)}$  we have $\delta(x) \leq  \max_{i \in I} (10 r_i)  $. If $x \notin \bigcup_{i \in I} 9 B_i$, then $\delta(x) = 0$. Otherwise, there exists $k$ such that $x \in 9 B_k$ and thus $\delta(x_k) \leq \abs{x - x_k} + r_k \leq 10 r_k$. The last inequality in (2) follows from \eqref{eq: in the next section}.

    (3)   Fix \EEE $x \in E  \cap \overline{B(x_0,3r_0/4)}$ and  $r \in (\delta(x),r_0/4]$.   First, if $r \in [r_0/20, r_0/4]$, (3) is  trivial because we use $\beta_E(x_0,r_0) = \beta_K(x_0,r_0) \leq {\bar{\eps}_{\rm flat}}\leq \tau/40$ to get directly
    \begin{equation*}
        \beta_E(x,r) \leq \frac{2 r_0}{r}  \beta_E(x_0,r_0) \leq 40\beta_E(x_0,r_0) \leq \tau.
    \end{equation*}
    From now on, we suppose  $r \leq r_0/20$.  If $x \in K$ and $\sigma(x) = 0$, (3) {immediately follows from Lemma \ref{eq_betaF}}.
    {Otherwise, {we either have} $x \in E \setminus K$ or $x \in K$ with $\sigma(x) > 0$.
         We justify first that   in these two cases  there exists $k$ such that  $x \in 4 B_k$. 
    If $x \in K$ and $\sigma(x) > 0$, then $x \in R(x_0,r_0)$ and  by the definition of $(B_i)_{i \in I}$, see \eqref{def_Bi}, we have   $x \in \bigcup_k  4  B_k$.}    In the case $x \in E \setminus K$, we start by estimating the distance from $x$ to $K$. For all $\rho > 0$ such that $B(x,\rho) \cap K = \emptyset$,  Lemma \ref{lem_AFF}(1)  and    the assumption   $\eta(x_0,r_0) \le  \tau/20 \EEE $   give
    \begin{align}\label{ahaha}
         2 \EEE  \rho \leq \HH^1\big(E \cap B(x,\rho)\big) \leq \HH^1(E \setminus K) \leq \eta(x_0,r_0) r_0 \leq r_0/  10. \EEE
    \end{align}
    Thus, the distance $t := \mathrm{dist}(x,K) \leq r_0/20$ is attained at some point $z \in K \cap {B(x_0,9r_0/10)}$.
    Next, we observe that $B(z,2t)$ cannot be good ball in the sense of \eqref{eq_goodball1}--\eqref{eq_goodball2}. Indeed, by letting $\rho \nearrow t$ in \eqref{ahaha},    we have
    \begin{align*}
         2 \EEE   t \leq \HH^1\big(   (E\setminus K) \EEE \cap B(x,t) \big) \leq \HH^1\big((E \cap B(z,{2t})) \setminus K\big),
    \end{align*}
    and this excludes the possibility that $(2 t)^{-1} \HH^1((E \cap B(z,2t)) \setminus K) \leq \tau$.
    It follows that $\sigma(z) \geq 2 t$ and thus, since $|x-z| = t$, $x \in B(z, \sigma(z)) \subset \bigcup_{i\in I} 4 B_i$, where the last inclusion holds due to \eqref{def_Bi}.

    Now, we have shown that for $x$  there exists   $k \in I$ such that  $x \in 4 B_k$.  Let us  check that for all $r \in (\delta(x),r_0/20]$ we have $\beta_E(x,r) \leq 16  \tau$.
    We can clearly bound
    \begin{equation*}
        \inf_{i \in I} {\left(\abs{x - x_i} + r_i \right) }\leq \abs{x - x_k} + r_k \leq 5 r_k
    \end{equation*}
    whereas
    \begin{equation*}
        {\mathrm{dist}\left(x,\R^2 \setminus \bigcup_{i \in I} 9 B_i\right)} \geq \mathrm{dist}(x, \R^2 \setminus 9 B_k) \geq 5 r_k.
    \end{equation*}
    Thus, $\delta(x) = \inf_{i \in I} {(\abs{x - x_i} + r_i)}$.
    As $r > \delta(x)$, there exists $i$ such that $\abs{x - x_i} + r_i \leq r$. In particular, $B(x,r) \subset B(x_i,2r)$ and $r \geq r_i = M\sigma(x_i) > \sigma(x_i)$.
    We have $2 r \in (\sigma(x_i),r_0/10)$ so, according to Lemma~\ref{eq_betaF},  $\beta_E(x_i,2r) \leq 4 \tau$.
    As $B(x,r) \subset B(x_i,2r)$, we can once more apply the scaling properties of $\beta$ in Remark \ref{rmk_beta} to estimate $\beta_E(x,r) \leq 16 \tau$.

    In view of Definition \ref{def: geo func}, this shows that  $\delta$ is a geometric function for $E$ with parameters $(3r_0/4, 16\tau)$. In particular, also note that  $16\tau \ge 8 \bar{\eps}_{\rm flat} \ge  8\beta_E(x_0,r_0)$ is satisfied by \eqref{eq_varepsilon00_flat}. 
\end{proof} 

We proceed with two results that will be instrumental in the next subsections. 

\begin{lemma}[Selection of good radii]\label{lem_goodradius}
    Let $(u,K)$ be a Griffith almost-minimizer.  There exists a constant $C_0 \geq 1$  depending on $\mathbb{C}$ such that  for each $x_0 \in K$ and $r_0 > 0$, with  $B(x_0,r_0) \subset \Omega$, {$h(r_0) \leq \varepsilon_{\rm Ahlf}$}, {and} $\beta(x_0,r_0) + \eta(x_0,r_0) \leq {\bar{\eps}_{\rm flat}}$, there exists ${\Xi} \subset [r_0/2, 3r_0/4]$ with $\mathcal{L}^1( [r_0/2, 3r_0/4] \setminus {\Xi}) \le r_0 /8$ such that for each $\rho \in {\Xi}$ we have  
    \begin{equation*}
        \sum_{i \in I(\rho)} r^\sigma_i  \leq C_0 \big( \beta(x_0,r_0)m(x_0,r_0) + \eta(x_0,r_0)\big)  r_0,
    \end{equation*}
    where for $t \in (0,r_0)$ we set 
    \begin{equation}\label{Irho}
        I(t) := \big\{i \in I \colon \,  10 B^\sigma_i \cap \partial B(x_0,t) \neq \emptyset\big\}.
    \end{equation}
\end{lemma}

\begin{proof}
    Let ${\Xi} \subset  [r_0/2, 3r_0/4]$ be the set of radii $\rho$  such that the {cumulative circumferences} of the balls $(10 B_i)_{i\in I}$ intersecting $\partial B(0,\rho)$ is less than $2$-times the average, i.e., 
    \begin{align}\label{eq: comcocmon1}
        \sum_{i \in I(\rho)} r^\sigma_i  \leq 2 {\dashint_{r_0/2}^{3r_0/4}} \left(\sum_{i \in I(t)} r^\sigma_i\right) \dd{t}.
    \end{align}
    By {Markov's inequality,} $\mathcal{L}^1( [r_0/2, 3r_0/4] \setminus {\Xi}) \le r_0 /8$.
    {Observe that, for a fixed $i \in I$, the set of $t > 0$ such that $i \in I(t)$ is an interval of  length at most $2 r_i^{\sigma}$.}
    Thus, by Fubini's theorem   we {have}
    \begin{align}\label{eq: comcocmon2}
        \int_{r_0/2}^{3r_0/4}  \sum_{i \in I(t)} r^\sigma_i \dd{t}   &\leq \sum_{i\in I} \int_{\set{t \colon i \in   I(t)}} r^\sigma_i \dd{t} \le \sum_{i\in I} 2(r^\sigma_i)^2 \leq C \big( \beta(x_0,r_0) + \eta(x_0,r_0) \big)r_0  \sum_{i\in I} r^\sigma_i ,
    \end{align}
    where in the last step we used   \eqref{eq_ri_estimate} with $C = C(M, \tau) \geq 1$.  The balls $(B^\sigma_i)_i$ are pairwise disjoint and contained in $R(x_0,r_0)$ by  the definition of the \EEE covering, see \eqref{def_Bi}.
    Then,  by the Ahlfors-regularity in \eqref{eqn:AhlforsReg}, it follows    
    \begin{align}\label{eq: comcocmon3}
        \sum_{i\in I} r^\sigma_i \le \sum_{i\in I} {C_{\rm Ahlf}}\mathcal{H}^1\big(K \cap B^\sigma_i\big) \le {C_{\rm Ahlf}\mathcal{H}^1}\big(K \cap R(x_0,r_0)\big) =  {C_{\rm Ahlf}}m(x_0,r_0)r_0,
    \end{align}
    where for the  last identity we refer to \eqref {eq_defi_bad_mass}. {Selecting $\rho \in \Xi$,  using that $m(x_0,r_0) \le C_{\rm Ahlf}$ by \eqref{eqn:AhlforsReg} and \eqref{eq: the R}, \EEE  and combining} \eqref{eq: comcocmon1}--\eqref{eq: comcocmon3} {concludes the proof.}
\end{proof}

\begin{corollary}[Wall set]\label{cor: wall set}
    Let $W_{10}^{\rm bdy}$ be the set in \eqref{eq: 10bry} related to the geometric function $\delta$. Then, we have
    \begin{equation}\label{eq_Z}
        W_{10}^{\rm bdy} \subset \bigcup_{i \in I(\rho)} 10 B^\sigma_i,
    \end{equation}
    where $I(\rho)$ is the set defined in \eqref{Irho}. Moreover, the wall set $\Gamma :=  \bigcup_{i \in I(\rho)} \partial (10 B^\sigma_i)$ satisfies
    \begin{align}\label{2ndpro}
        \overline{\Gamma} \setminus \Gamma \subset \partial B(x_0,\rho) \cap K. 
    \end{align}
\end{corollary}

\begin{proof}
    We consider $x \in E \cap \overline{B(x_0,\rho)}$ such that $B(x, 50 r_x) \cap \partial B(x_0,\rho) \ne \emptyset$, where $r_x = 10^{-5}\delta(x)$. 
    As $x \in E \cap \overline{B(x_0,\rho)}$ and  $\delta(x) > 0$,  it follows from the definition of $\delta$  in \eqref{eq_delta} that there exists $i$ such that $x \in 9 B^\sigma_i$ and thus
    \begin{equation*}
        \delta(x) \leq \abs{x - x^\sigma_i} + r^\sigma_i \leq 10 r^\sigma_i.
    \end{equation*}
    {It follows that $10 r_x =  10^{-4} \delta(x) \leq 10^{-3} r^{\sigma}_i$.  Since  $x \in 9 B^\sigma_i$, we find $B(x,  10 \EEE r_x) \subset 10 B_i^\sigma$.} This yields \eqref{eq_Z}. 

    Let us now show \eqref{2ndpro}.       For $x \in \overline{\Gamma} \setminus \Gamma$,      we can extract a sequence of distinct elements $(10 B^\sigma_{i_k})_{k \in \N}$ from the family $(10 {B^\sigma_i})_{i \in I(\rho)}$ such that $\mathrm{dist}(x,  \overline{10 B^\sigma_{i_k}}) \to 0$ as  $k \to +\infty$.    The balls $(B^\sigma_{i_k})_k$ are disjoint and contained in $B(x_0,r_0)$, see \eqref{def_Bi}.  Thus, their radii must tend to zero  as $k \to +\infty$.
    We deduce that $\mathrm{dist}(x,x^\sigma_{i_k}) \to 0$, where $x^\sigma_{i_k}$ denotes the center of $B^\sigma_{i_k}$. 
    But since the points $x^\sigma_{i_k}$ belong to the relatively closed set $K$ and {each ball} $10 B^\sigma_{i_k}${intersects} $\partial B(x_0,\rho)$ by definition of $I(\rho)$, we conclude $x \in K \cap \partial B(x_0,\rho)$.    
\end{proof}

{We briefly comment on the use of the previous lemma. When we apply the extension Proposition \ref{th: extension} to construct   competitors, we will be forced to introduce a set $W_{10}^{\rm bdy}$ where the extended function is not controlled. However, we will be able to cover this set with $\bigcup_{i \in I(\rho)} 10 B^\sigma_i$. Consequently, we will simply set competitors  to \EEE be $0$ in $\bigcup_{i \in I(\rho)} 10 B^\sigma_i$ and create a `wall-set' given by $\partial(\bigcup_{i \in I(\rho)} 10 B^\sigma_i)$. The competitor will pay the price of having a slightly larger crack with additional length at most  $2\pi \sum_{i \in I(\rho)} 10r^\sigma_i$. The choice of good radii in Lemma \ref{lem_goodradius} then yields a helpful estimate on the  size \EEE of this wall set. }

\subsection{Control on the flatness}\label{section_flatness}

In this subsection, we first control the flatness of separating sets via the \emph{{(normalized)} length excess}. Here, the length excess of a set $E$   separating {the} ball $B(x_0,r_0)$ is given by {$r_0^{-1} \left(\mathcal{H}^1 (E)  - \mathcal{H}^1([p;q])\right)$}, where $[p;q]$ denotes the segment between two points $p,q \in \partial B(x_0,r_0){\cap E}$. Based on this control and the notion $\eta$ measuring the difference of $K$ and a minimal separating extension, we then can {gain} control {of} the flatness (Proposition \ref{prop_flatness_decay}). Although the main arguments in this subsection are for separating sets and in this case unilateral and bilateral flatness are equivalent, see Lemma \ref{lem_bilateral_flatness-old} or {\cite[Remark 2.9]{CL}},  we recall that for the entire proof it is more convenient to work with the unilateral flatness.  

{We will frequently use the minimal separating extensions as analyzed in Subsection \ref{sec:fillingHoles}. The separation property  will allow us to argue along the lines of the constructions in \cite{CL}, always adding an additional error term $\eta_K$ which quantifies the difference between $K$ and the extension $E$, see \eqref{eq: eta definition}. Since careful adaptations are necessary (e.g., we apply the extension lemma to the minimal separating extension),    we will provide full details   to ensure a self-contained presentation. }

\begin{lemma}[Intersection of circles]\label{lem_twopoints}
    Let $x_0 \in \R^2$, $r_0 > 0$,   and let $E$ be a relatively closed subset of $B(x_0,r_0)  $ such that $x_0 \in {E}$, $\beta_E(x_0,r_0) {\leq 1/32}$, $E$   {separates}   $B(x_0,r_0)$, and  $\HH^1(E)  \le {(2+1/8)}r_0  $.   Then, there exists a subset ${\Xi} \subset (r_0/ 16 ,  r_0)$ with 
    \begin{align}\label{the R esti}
        \mathcal{L}^1\big((r_0/ 16 ,  r_0) \setminus {\Xi}\big) \leq  r_0/4
    \end{align}
    such that for all $\rho \in {\Xi}$ we have 
    \begin{equation}\label{eq_good_radius0}
        \begin{gathered}
            \text{$E \cap \partial B(x_0,\rho)$ consists of exactly two points at distance larger than $\rho$.}
        \end{gathered}
    \end{equation}
\end{lemma}

\begin{proof}
    Let $\ell$ be a line passing through $x_0$ such that $E$ is contained in the   strip \EEE
    \begin{equation*}
        S:= \big\{x \in B(x_0,r_0) \colon \,  \mathrm{dist}(x,\ell) \leq \beta_E(x_0,r_0) r_0\big\},
    \end{equation*}
    see \eqref{eq_beta2}.    For $\rho \in (\beta_E(x_0,r_0) r_0,r_0)$, let $N(\rho)$  be the cardinality of  $E \cap \partial B(x_0,\rho)$.   Given a radius $\rho \in (\beta_E(x_0,r_0) r_0,r_0)$, we observe that each arc of $\partial B(x_0,\rho) \cap S$
    joins the two connected components of $B(x_0,r_0) \setminus S$.   Since $E$ separates, it must meet both arcs composing $\partial B(x_0,\rho) \cap S$. This implies $N(\rho) \geq 2$.  Moreover, for $\rho \in (2\beta_E(x_0,r_0) r_0,r_0)$, the two arcs making up $\partial B(x_0,\rho) \cap S$  have  a distance larger or equal to $2 \sqrt{\rho^2 - (\beta_E(x_0,r_0) r_0)^2} > \rho$ from each  other,  where the last inequality follows from   $\rho  > \EEE 2\beta_E(x_0,r_0)r_0$.

    We now set $\Xi := \lbrace \rho \in (2\beta_E(x_0,r_0) r_0,r_0)\colon \, N(\rho)=2\rbrace$, and observe that the lemma will follow once we show \eqref{the R esti} holds and redefine $\Xi$ to be $\Xi \cap (r_0/ 16 ,  r_0)$.       According to the coarea formula, we have
    \begin{equation}\label{arc1}
        \int_{2\beta_E(x_0,r_0)r_0}^{r_0} N(t) \dd{t} \le \HH^1\big(E \cap B(x_0,r_0)\big) \le (2+1/8  )r_0.
    \end{equation}
    Clearly, we also have
    \begin{equation}\label{arc2}
        \int_{2\beta_E(x_0,r_0)r_0}^{r_0} N(t) \dd{t} \geq 2(1-2\beta_E(x_0,r_0))r_0   +  \mathcal{L}^1\big((2\beta_E(x_0,r_0)r_0,r_0) \setminus \Xi\big).
    \end{equation}
    Combining \eqref{arc1} and \eqref{arc2}, and using $\beta_E(x_0,r_0) \le 1/32$,  we have
    $$\mathcal{L}^1\big((2\beta_E(x_0,r_0)r_0,r_0) \setminus \Xi\big) \leq (4 \beta_E(x_0,r_0) +1/8  )r_0 \leq \frac{1}{4}r_0,$$
    thereby concluding the result \eqref{the R esti}.
\end{proof}

In a ball satisfying \eqref{eq_good_radius0}, the flatness can be controlled by the length  excess. 

\begin{lemma}[Flatness control by length excess]\label{lem_flatness_control}
    Let $x_0 \in \R^2$, $r_0 > 0$, and let $E$ be a relatively closed subset of ${B(x_0,r_0)}$   such that $x_0 \in E$ and $E$ {separates} $B(x_0,r_0)$.
    We assume that there exists a constant $C_0 \geq 1$ such that for all $x \in E $ and for all $r > 0$ with $B(x,r) \subset B(x_0,r_0)$ it holds
    \begin{equation}\label{eq_AF_flatness-2nd}
        \HH^1(E \cap B(x,r)) \geq C_0^{-1} r.
    \end{equation}
    We also suppose that $\overline{E} \cap \partial B(x_0,r_0)$ consists of {exactly two points $p$ and $q$ that are at least distance $r_0$ apart.}
    Then, we have
    \begin{equation*}
        \beta_E(x_0,r_0/2)^2 \leq {64} C_0 r_0^{-1} \big( \HH^1(E) - \abs{p-q}\big).
    \end{equation*}
\end{lemma}
\begin{proof}
    By scaling and shifting it is not restrictive to assume that $x_0 = 0$ and $r_0 = 1$.    Let $p$ and $q$ be the two points of $\overline{E} \cap \partial B(0,1)$.  
    To simplify the notations,   we denote the length excess by 
    \begin{equation*}
        \lambda = \HH^1(E) - \abs{p - q}.
    \end{equation*}
    We can assume that  $\lambda \leq C_0^{-1}/{4}$ as otherwise the statement is clearly true since we always have $\beta_E(0,1/2) \leq 1$.

    The set $\overline{E}$ separates the two components of $\partial B(0,1) \setminus \lbrace p,q\rbrace$ in $\overline{B}(0,1)$. Thus, according to \cite[Theorem~14.3]{Ne},  there exists a {(closed)} connected subset $\Gamma \subset \overline{E}$ that connects $p$ and $q$.
    For $x \in \Gamma$, we estimate the distance $\mathrm{dist}(x,\ell)$ from the line $\ell$ passing through $p$ and $q$ in terms of $l_x := \abs{x - p} + \abs{x - q}$.
    For a given $h>0$, one can show that among $x$ with fixed  distance ${\rm dist}(x,\ell) = h$, $l_x$ is minimized when $x$'s orthogonal projection coincides with the midpoint $(p+q)/2$. 
    Considering a point $y \in \R^2$ such that ${\rm dist}(y,\ell) = {\rm dist}(x,\ell)$ and whose orthogonal projection of $y$ onto $\ell$ is the midpoint, it thus satisfies $l_y \leq l_x$ and one can compute by the Pythagorean theorem $\mathrm{dist}(y,\ell)^2 = (l_y^2 - \abs{p-q}^2)/4$, whence
    \begin{equation*}
        \mathrm{dist}(x,\ell)^2 \leq \tfrac{1}{4} (l_x^2 - \abs{p-q}^2).
    \end{equation*}
    Since $x,p,q \in \overline{B(0,1)}$, we have the bound $l_x + \abs{p - q} \leq 6$ {and subsequently we  see that}
    \begin{equation*}
        \mathrm{dist}(x,\ell)^2 \leq     \tfrac{1}{4} (l_x - \abs{p-q}) (l_x + \abs{p-q}) \le  \tfrac{3}{2} (l_x - \abs{p-q}).
    \end{equation*}
    As $\Gamma$ is connected and joins {$p$ to $x$} and {$q$ to $x$}, we have $l_x \leq \HH^1(\Gamma)$. This implies $l_x \leq \HH^1(E)$.
    We conclude that  it holds
    \begin{equation}\label{eq_dist_GL}
        \mathrm{dist}(x,\ell)^2 \leq \tfrac{3}{2} \lambda \quad \text{for  all $x \in \Gamma$.}
    \end{equation}
    Next, we see that a point $y \in E \cap \overline{B(0,1/2)}$ cannot be too far away from {the connected curve $\Gamma$ within $E$}. Indeed, the density assumption (\ref{eq_AF_flatness-2nd}) implies that for each $t \in (0,1/2)$ with $B(y,t) \subset B(0,1) \setminus \Gamma$, we have
    \begin{equation}\label{eq: combi}
        \HH^1\big(E \setminus \Gamma \big) \geq C_0^{-1} t.
    \end{equation}
     As \EEE $\mathcal{H}^1(\Gamma) \ge |p-q|$, we estimate
    \begin{equation*}
        \HH^1\big(E \setminus \Gamma\big) =  \HH^1(E)  - \HH^1(\Gamma) \leq \lambda.
    \end{equation*}
    This, combined with \eqref{eq: combi}, gives 
    \begin{equation*}
        \min(\mathrm{dist}(y,\Gamma),1/2) \leq C_0 \lambda.
    \end{equation*}
    Due to the assumption $\lambda \leq C_0^{-1}/{4}$, this simplifies to
    \begin{equation}\label{eq_dist_FG}
        \mathrm{dist}(y,\Gamma) \leq C_0 \lambda \leq \tfrac{1}{2} \sqrt{C_0} \sqrt{\lambda}.
    \end{equation}
    {As $\sqrt{3/2} + 1/2 \leq 2$,} we combine (\ref{eq_dist_GL}) and (\ref{eq_dist_FG}) to  bound
    \begin{equation}\label{eq_distxpq}
        \mathrm{dist}(y,\ell) \leq 2 \sqrt{C_0} \sqrt{\lambda} \quad \quad \text{for all $y \in E \cap \overline{B(0,1/2)}$.}
    \end{equation}
    This does not directly control $\beta_E(0,1/2)$ because the line $\ell$ may not pass through $x_0 = 0$. But we can can estimate $\beta_E(0,1/2)$ via the line  $\ell'$ which is parallel  to $\ell$ and contains $x_0=0$.     The distance of $\ell$ and $\ell'$ is less than the right-hand side of (\ref{eq_distxpq}). Therefore, we conclude   $\mathrm{dist}(y,\ell') \leq 4 \sqrt{C_0} \sqrt{\lambda}$ for all $y \in E \cap \overline{B(0,1/2)}$, and thus $    \beta_E(0,1/2) \leq 8 \sqrt{C_0} \sqrt{\lambda}$. 
\end{proof}

Next, we show how to control the flatness in terms of the minimality defect of so-called \emph{separation competitors}. This property will be  also relevant in the proofs of the next subsections. \EEE

\begin{definition}[Separation competitor]\label{def: sep comp}
    Given a relatively closed set ${E}$ in $\Omega$ and an open ball {$B \subset \subset \Omega$}, we say that  a relatively closed set $L \subset \Omega$ is a \emph{separation competitor of $E$ in $B$} if 
    $$E \setminus B = {L} \setminus B, $$
    and if, whenever two points $x,y \in \Omega \setminus (B\cup {E})$ are separated by ${E}$, then they are also separated by ${L}$.   The \emph{minimality defect} of ${E}$ in $B$ is the smallest $a \in [0,+\infty]$ such that for all separation competitors  ${L}$ of ${E}$ in $B$, we have
    \begin{equation*}
        \HH^1({E} \cap B) \leq \HH^1({L} \cap B) + a r,
    \end{equation*}
    where $r$ is the radius of $B$.
\end{definition}

The same notion {of competitor} (with a different name) has been used already in the literature, see for example  \cite[Definition 1.10]{Lemenant}. \EEE We now  control the flatness via the minimality defect.

\begin{lemma}[Flatness and minimality defect]\label{lem_flatness}
    Set  $\eps_{\rm def} = 2^{-7} $.   Let $x_0 \in \R^2$, $r_0 > 0$ and let $E$ be a relatively closed subset of $B(x_0,r_0)$ such that $x_0 \in E$, $\beta_E(x_0,r_0) \leq \eps_{\rm def}$, and $E$   separates   $B(x_0,r_0)$.
    We assume that there exists a constant $\lambda >0$ such that for all separation competitors $L$ of $E$ in a ball $B(x_0,{\rho})$, where $\rho \in (0,r_0)$, it holds that
    \begin{equation}\label{eq: Lassum}
        \HH^1(E) \leq \HH^1(L) + \lambda r_0.
    \end{equation}
    Moreover, we suppose  that there exists a constant $C_0 \geq 1$ such that for all $x \in {E}$ and for all $r > 0$ with $B(x,r) \subset B(x_0,r_0)$ it holds
    \begin{equation}\label{eq_AF_flatness}
        \HH^1(E \cap B(x,r)) \geq C_0^{-1} r.
    \end{equation}
    Then, we have
    \begin{equation*}
        \beta_E(x_0,r_0/4) \leq { 32 \EEE }  {\sqrt{C_0}} \sqrt{\lambda}.
    \end{equation*}
\end{lemma}

Note that the result is a simplified variant of the higher dimensional result  \cite[Lemma 4.1]{CL}. The statement here with a bound in terms of $\sqrt{\lambda}$ is exclusively valid in the plane. In fact, it  implies that minimal sets (sets with minimality defect zero in Definition \ref{def: sep comp}{, i.e., $\lambda = 0$}) with small flatness are segments, whereas in higher dimensions  smooth parts of minimal sets are not necessary flat. {For example, a catenoid satisfies $\beta(x,r) \to 0$ at every point but {it never holds that} $\beta(x,r) = 0$.}

\begin{proof}[Proof of Lemma \ref{lem_flatness}]
    We first note that it is not restrictive to assume that $\lambda \in   [0,{1/16}]$  since otherwise the statement readily follows for   $C_0 \ge 1$ since the flatness is always bounded by $1$.    We start by estimating the length excess of $E$ in  balls $B(x_0,\rho)$ for $\rho \in (r_0/2,r_0)$.
    Let $\ell$ {be} a line through $x_0$ such that
    \begin{equation*}
        E \subset \big\{x \in B(x_0,r_0) \colon \, \mathrm{dist}(x,\ell ) \leq \eps_{\rm def} r_0\big\}.
    \end{equation*}
    Then, one can build a separation  competitor $L$ of $E$ in $B(x_0,r_0)$ by setting 
    \begin{equation*}
        L := \left(E \setminus B(x_0, \rho)\right) \cup Z \cup \left(\ell \cap B(x_0,\rho)\right),
    \end{equation*}
    with the wall set
    \begin{equation*}
        Z := \big\{x \in \partial B(x_0,\rho) \colon \,  \mathrm{dist}(x,\ell) \leq \eps_{\rm def} r_0\big\}
    \end{equation*}
    (see Figure \ref{fig:simpleWallSet} for a similar construction). It is elementary to check that {$$\mathcal{H}^1(Z) \le 4 \arcsin\left(\frac{\eps_{\rm def} r_0}{\rho}\right) r_0 \leq 4 \arcsin(2 \eps_{\rm def}) r_0 \leq 16 \eps_{\rm def} r_0.$$}
    Thus, by  \eqref{eq: Lassum} it  follows that
    \begin{equation*}
        \HH^1(E) \leq \mathcal{H}^1(E \setminus B(x_0,\rho)) +    2 \rho +    {16}  \eps_{\rm def} r_0 + \lambda r_0.
    \end{equation*}
    By sending $\rho \to r_0$, using $\lambda \in   [0,{1/16}]$  and  $\eps_{\rm def} = 2^{-7} $, we get  $ \HH^1(E) \le {(2 + 1/8)}r_0$.  Then,  Lemma \ref{lem_twopoints} yields the existence of $\rho \in (r_0/2,r_0)$ such that $E \cap \partial B(x_0,\rho)$ consists of exactly two points $p$ and $q$  with distance larger than $\rho$.  

    Now, we can try a more clever separation competitor of $E$ in $ B(x_0,\rho)$, namely
    \begin{equation*}
        L = \left(E \setminus B(x_0,\rho)\right) \cup [p;q]
    \end{equation*}
    which by \eqref{eq: Lassum} yields
    \begin{equation*}
        \mathcal{H}^1(E) \leq    \mathcal{H}^1(E \setminus B(x_0,\rho)) + \abs{p - q} + \lambda r_0,
    \end{equation*}
    whence
    \begin{equation*}
        \HH^{1}(E \cap B(x_0,\rho)) - \abs{p - q} \leq \lambda r_0.
    \end{equation*}
    Then, Lemma \ref{lem_flatness_control} (applied with $\rho$ in place of $r_0$) and {the fact that $\rho \geq r_0/2$} show that $\beta_E(x_0,\rho/2) \leq {8 \sqrt{2 C_0}} \sqrt{\lambda}$. This along with Remark \ref{rmk_beta} concludes the proof. 
\end{proof}

Before beginning the proof  of  our first decay estimate, we state a lemma which is used to move between competitors of $(u,K)$ and  those \EEE associated with the minimal separating extension $(u,E(x_0,r_0))$ in a ball $B(x_0,r_0).$ 
{It  implies  that the pair $(u,E)$ inherits an `almost-minimality property', up to an additional error term $\eta(x_0,r_0)$.}

\begin{lemma}[Competitors]\label{lem:competitor}
    {Let $(u,K)$ be a Griffith almost-minimizer in $\Omega$.
    Let $x_0 \in K$ and $r_0 > 0$ be such that $B(x_0,r_0) \subset \Omega$ and $\beta(x_0,r_0) + \eta(x_0,r_0)\leq 1/20$.}
    Let  $E$ be a  minimal separating extension of $K$ in $B(x_0,r_0)$. If  $(v,G)$ is a pair {in $B(x_0,r_0)$} with $G$ a relatively closed subset of $B(x_0,r_0)$ and $v\in W^{1,2}_{\rm loc}(B(x_0,r_0)\setminus G; \R^2)$
    such that
    \begin{equation}\label{5X}
        v = u \quad \text{in} \quad B(x_0,r_0) \setminus \left(E \cup B\left(x_0,9r_0/10\right)\right)
    \end{equation}
    and
    \begin{equation}\nonumber
        G \setminus B\left(x_0, 9r_0/10\right) = E \setminus B\left(x_0,9r_0/10\right),
    \end{equation}
    then
    \begin{equation}\nonumber
        \int_{B(x_0,r_0) \setminus E} \mathbb{C}e(u)\colon e(u) \dd{x} + \HH^1(E) \leq \int_{B(x_0,r_0) \setminus G} \mathbb{C}e(v)\colon e(v) \dd{x}+ \HH^1(G) + \eta(x_0,r_0) r_0 + h(r_0) r_0.
    \end{equation}
\end{lemma}

\begin{proof}
    By the coarea formula, the measure of the set of radii $r_* \in (0,r_0)$ such that $E \cap \partial B(x_0,r_* )  \ne K \cap \partial B(x_0,r_*)$, {i.e., $(E \setminus K) \cap \partial B(x_0,r_*) \ne \emptyset$}, is bounded by $  \eta(x_0,r_0)   r_0 \leq  r_0/20 $ (see also the proof of Lemma \ref{lem_twopoints}). So, we choose $r_* \in (9r_0/10,r_0)$ such that
    \begin{equation}\label{eqn:closedCompatibility}
        K \cap \partial B(x_0,r_*) = E \cap \partial B(x_0,r_*) = G \cap \partial B(x_0,r_*).
    \end{equation}
    It follows that $(v, \left[G \cap B(x_0,r_*)\right] \cup\left[K \setminus B(x_0,r_*)\right])$ is an admissible competitor for $(u,K)$ where it can be directly checked that the crack is relatively closed in $\Omega$ by \eqref{eqn:closedCompatibility}.  Here, \EEE we have implicitly extended $v$ to $u$ outside of $B(x_0,r_0)$.
    Using almost-minimality of $(u,K)$ with respect to the competitor, we have
    \begin{equation}\nonumber
        \int_{B(x_0,r_0) \setminus K} \mathbb{C}e(u)\colon e(u) \dd{x} + \HH^1(K \cap B(x_0,r_0)) \leq  \int_{B(x_0,r_0) \setminus G} \mathbb{C}e(v)\colon e(v) \dd{x} + \HH^1(G) + h(r_0) r_0
    \end{equation}
    from which the conclusion follows as $\mathcal{H}^1(E\setminus K) = \eta(x_0,r_0)r_0.$
\end{proof}

We now come to  the proof of the main result of this subsection, namely  the control of the flatness. This is the first time that we will use the extension in Proposition \ref{th: extension}. Note that for this proof we could also use another extension lemma avoiding the notion of bad mass, see \cite[Lemma 4.5]{BIL}. However, the latter does not provide the sharp control of $\beta$ in terms of $\sqrt{\omega}$, compare  Proposition~\ref{prop_flatness_decay} and  \cite[Proposition~3.2]{BIL}.

\begin{proof}[Proof of Proposition \ref{prop_flatness_decay}]
    We let $E(x_0,r_0)$ be a minimal separating extension of $K$ in $B(x_0,r_0)$, see \eqref{eq: eta definition2}, and as usual we  write $E$ in place of $E(x_0,r_0)$ for simplicity. Moreover, within this proof, for brevity we write $\beta_0 := \beta(x_0,r_0)$ and $\kappa = 10^{-6}$. Given the {constants $\bar{\eps}_{\rm flat}$ in (\ref{eq_varepsilon00_flat}) and} $\eps_{\rm def}$ from Lemma \ref{lem_flatness}, we set ${\eps}_{\rm flat} := \min\lbrace \kappa /4, \kappa \eps_{\rm def},  {\bar{\eps}_{\rm flat}/2}\rbrace$.

    \emph{Step 1: Reduction to controlling minimality defect.}    The main part of the proof consists in controlling the minimality defect of $E$: we show that for all separation competitors $L$ of $E$ in $B(x_0, \kappa r_0)$, see Definition~\ref{def: sep comp},  we have 
    \begin{equation}\label{eq_F_minimality_defect}
        \HH^{1}\big(E\cap B(x_0,\kappa r_0)\big) \leq \HH^{1}\big(L \cap B(x_0,\kappa r_0)\big)+ C\Big(\omega(x_0,r_0) + \beta_0 m(x_0,r_0) + \eta(x_0,r_0) + h(r_0)\Big)\kappa r_0,
    \end{equation} 
    for  $C \ge 1$ sufficiently large {which depends on $\mathbb{C}$}.     Then, we will control the flatness of $E$ via Lemma \ref{lem_flatness}. Indeed, suppose for the moment that \eqref{eq_F_minimality_defect} holds, i.e., we have \eqref{eq: Lassum} for $\lambda = C(\omega(x_0,r_0) + \beta_0 m(x_0,r_0) + \eta(x_0,r_0) + h(r_0))$.   Note that $\beta_0 \le \eps_{\rm flat}   \le \kappa /4$. This along with Lemma \ref{lem_separation}(1)  implies that   $E$ separates $B(x_0,\kappa r_0)$.  {As $\beta_K(x_0,r_0)  =\beta_E(x_0,r_0)$,} we have $\beta_E(x_0,\kappa r_0) {\le \kappa^{-1} \beta_E(x_0, r_0)} \le  \kappa^{-1} \eps_{\rm flat}  \le   \eps_{\rm def}$.    Moreover,   \eqref{eq_AF_flatness} holds  by    Lemma~\ref{lem_AFF}.      Then, Lemma~\ref{lem_flatness} yields that
    \begin{equation*}
        \beta_E(x_0,\kappa r_0/4)^2 \leq  C  \big(\omega(x_0,r_0) + \beta_0 m(x_0,r_0) + \eta(x_0,r_0) + h(r_0)\big)
    \end{equation*}
    for a larger constant $C$ depending only on $C_{\rm  Ahlf}$ and thus on $\mathbb{C}$.  
    Since $K \subset E$, we then get $\beta_K(x_0,\kappa r_0/4) \leq \beta_E(x_0,\kappa r_0/4)$, and the statement follows by the scaling properties of   $\beta$, see Remark \ref{rmk_beta}. 
    {Precisely, for $0 < b \leq 10^{-7}$, we have $b \leq \kappa/4$ and thus $\beta_K(x_0,b r_0) \leq  \kappa/(4b)   \beta_K(x_0,\kappa r_0/4)$.}

    It now remains to prove \eqref{eq_F_minimality_defect}.
    {For this, it is not restrictive to assume that
        \begin{align}\label{eta fla}
            \eta(x_0,r_0) \le \eps_{\rm flat}
        \end{align}  
    as otherwise   \eqref{eq_F_minimality_defect} trivially holds for $C$ large enough depending on $\eps_{\rm flat}$ and the constant in Lemma~\ref{lem_AFF}.}

    {As we will rely on an application of the extension proposition}, we recall the geometric function {$\delta$} defined in \eqref{eq_delta}.  In the following, we use a slightly different geometric function inducing big balls $B(x,r_x)$ in the center of $B(x_0,r_0)$, see \eqref{eq: Wdef}. More precisely, we use the function $\delta_{\rm big} \colon E \cap \overline{B(x_0,3r_0/4)}  \to [0,r_0/4]$ defined by
    \begin{equation}\label{eq_delta1}
        \delta_{\rm big}(x) := \max\big\{r_0/4 - \abs{x - x_0}, \delta(x)\big\}.
    \end{equation}
    Note by Lemma \ref{lem_delta} (applicable since $\beta_0 + \eta(x_0,r_0) \le 2\eps_{\rm flat} \le   {\bar{\eps}_{\rm flat}}$ by    the choice of $\eps_{\rm flat}$ and \eqref{eta fla})  that $\delta_{\rm big}$ is still a $1$-Lipschitz geometric function with parameters $(3r_0/4, 16\tau)$ in the sense of Definition~\ref{def: geo func},  but now we have $\delta_{\rm big}(x_0) = r_0/4$.   In view of  \eqref{eq: Wdef}, this will ensure that the domain of extension in \eqref{exsets} contains a ball with center $x_0$ and radius $\frac{1}{4} 10^{-5}r_0 \ge \kappa r_0$.

    \emph{Step 2: Separation competitor and wall set.}     By the definition of the flatness and the fact that $E$ {separates} $B(x_0,r_0)$ there exists a line $\ell_0$ passing through $x_0$  with unit normal $\nu(x_0,r_0)$   such that $E \cap B(x_0,r_0) \subset \lbrace x \colon \mathrm{dist}(x,\ell_0) \leq \beta_0 r_0  \rbrace$ and such that the sets  $D_{\beta_0r_0}^+(x_0,r_0)$ and $D_{\beta_0r_0}^-(x_0,r_0)$ defined in \eqref{eq :Bx0}
    belong to different connected component of $B(x_0,r_0) \setminus E$, denoted by {$\Omega_+^E$ and $\Omega_-^E$}.

     Let \EEE $L \subset B(x_0,r_0)$ be a separation competitor of $E$ in $B(x_0, \kappa r_0)$. We note that $L$ coincides with $E$ {in $B(x_0,r_0) \setminus B(x_0, \kappa r_0)$ and in particular in $B(x_0,r_0) \setminus B(x_0,r_0/2)$.  Therefore, the sets $D_{r_0/2}^+(x_0,r_0)$ and $D_{r_0/2}^-(x_0,r_0)$ (with respect to the same normal $\nu(x_0,r_0)$ as above) are also disjoint from $L$. By definition of a separation competitor, they lie in different connected component of $B(x_0,r_0) \setminus L$, denoted by {$\Omega_+^L$ and $\Omega_-^L$,} respectively.}  Since $L \setminus B(x_0,\kappa r_0) = E \setminus B(x_0,\kappa r_0)$, we also get  
    \begin{equation}\label{eq_X2}
        \Omega_\pm^L \subset \Omega_\pm^E \cup B(x_0, \kappa r_0), \quad \quad  \Omega_\pm^L \setminus B(x_0,\kappa r_0) = \Omega_\pm^E \setminus B(x_0,\kappa r_0).
    \end{equation}  
    Before we can construct a competitor function related to the separation competitor $L$, we need to introduce an additional wall set as motivated at the end of Subsection \ref{sec: stop time}.  Since $\beta_0 + \eta(x_0,r_0) \le 2{\eps}_{\rm flat}\le {\bar{\eps}_{\rm flat}}$,  see \eqref{eta fla}, \EEE we can use Lemma \ref{lem_goodradius} to select a radius $\rho \in [r_0/2, 3r_0/4]$ such that
    \begin{equation}\label{eq: yuhu1}
        \sum_{i \in I(\rho)} r^\sigma_i  \leq C \big(\beta_0 m(x_0,r_0) + \eta(x_0,r_0)\big) r_0,
    \end{equation}
    where $I(\rho) = \big\{i \in I \colon \,  10 B^\sigma_i \cap \partial B(x_0,\rho) \ne \emptyset \big\}$, and   $(B^\sigma_i)_{i \in I}$ is the family of balls constructed in Subsection~\ref{sec: stop time}.  Then, we define 
    \begin{equation}\label{eqn:Gdef1}
        G := L \cup \bigcup_{i \in I(\rho)} \partial (10 B^\sigma_i).
    \end{equation}
    {Let us justify that $G$ is relatively closed in $B(x_0,r_0)$.
    By Corollary \ref{cor: wall set}, we see that that difference between $\bigcup_{i \in I(\rho)} \partial (10 B^\sigma_i)$ and its  closure lies  in $K \cap B(x_0,\rho)$ and since $K \cap B(x_0,r_0) \subset E$ and $L \setminus B(x_0,\kappa r_0) = E \setminus B(x_0,\kappa r_0)$, we have $K \cap B(x_0,\rho) \subset L$. We conclude that replacing $\bigcup_{i \in I(\rho)} \partial (10 B^{\sigma}_i)$ by its closure would not change the definition of $G$. This proves our claim.}
    We also get     
    \begin{equation}\label{eq: GBL}
        G \setminus B\left(x_0,9r_0/10\right) = {L} \setminus B\left(x_0,9 r_0/10\right).
    \end{equation}
    since for all $i \in I(\rho)$ we have $\overline{10 B^\sigma_i} \subset B(x_0, 9 r_0/10)$ by  \eqref{eq: in the next section}  and the fact that  $\rho \le 3r_0/4$.

    \emph{Step 3: Construction of competitor.}     We apply Proposition \ref{th: extension} with respect to $(u,E)$, the radius $\rho \in [r_0/2, 3r_0/4]$ selected above, and  the geometric function $\delta_{\rm big}$ defined in (\ref{eq_delta1}) with parameters $(\rho,16\tau)$. (Note that Proposition \ref{th: extension} is applicable as $\beta_0 \le 10^{-8}/8$,   $16 \tau \le 10^{-8}$, and $ 8\beta_0 \leq  8\eps_{\rm flat}\leq 16  \tau$  by the definition of $\eps_{\rm flat}$ and \eqref{eq_varepsilon00_flat}.) \EEE

    {In the extended {domains} $\Omega^{\rm ext}_{\pm} = \Omega^E_{\pm} \cup W$}, we obtain functions $v_\pm \in LD_{\mathrm{loc}}(\Omega^{\rm ext}_\pm)$, and relatively closed subsets {$S_\pm \subset \Omega^{\rm ext}_{\pm}$}  such that
    \begin{equation}\label{wthf-beta}
        W \subset S_\pm \subset W_{10}, \quad \quad v_\pm = u \ \text{in} \ \Omega^{\rm ext}_\pm \setminus S_\pm, 
    \end{equation}
    and
    \begin{equation}\label{eq: yuhu2}
        \int_{(\Omega^{\rm ext}_\pm \cap B(x_0,\rho)) \setminus W_{10}^{\rm bdy}} \abs{e(v_\pm)}^2 \dd{x} \leq C \int_{B(x_0,\rho) \cap {\Omega_\pm^E}} \abs{e(u)}^2 \dd{x}.
    \end{equation}
    We claim   that 
    \begin{equation}\label{eq_XV2-beta}
        {\Omega_\pm^L} \subset \Omega^{\rm ext}_\pm = {\Omega_\pm^E} \cup W, \quad \quad \quad  W_{10}^{\rm bdy} \subset \bigcup_{i \in I(\rho)} 10 B^\sigma_i.
    \end{equation}
    The first inclusion follows from \eqref{eq_X2} and the fact that $W$ contains {$B(x_0,\frac{1}{4}10^{-5}r_0) = B(x_0,10^{-5} \delta_{\rm big}(x_0))$}, see \eqref{eq: Wdef} and \eqref{eq_delta1}. To see the second inclusion, we note that for each center  $x \in \mathcal{W}^{\rm bdy}$ (defined with respect to $\delta_{\rm big}$), see \eqref{eq: 10bry}, we find $x \notin B(x_0,r_0/4)$. This shows $\delta_{\rm big}(x) = \delta(x)$ by \eqref{eq_delta1}, and thus the set $W_{10}^{{\rm bdy}}$ would remain the same if it was defined with respect to the smaller geometric function $\delta$.  Then, the inclusion follows by Corollary \ref{cor: wall set}. Now,   \eqref{eq: yuhu2}  and \eqref{eq_XV2-beta} particularly imply
    \begin{equation}\label{eq: yuhu3}
        \int_{({\Omega_\pm^L} \cap B(x_0,\rho)) \setminus \bigcup_{i \in I(\rho)}  {10 B^\sigma_i}} \abs{e(v_\pm)}^2 \dd{x} \leq C \int_{B(x_0,\rho)} \abs{e(u)}^2 \dd{x}.
    \end{equation}
    We note that $G \cup\bigcup_{i \in I(\rho)} {10 {B^\sigma_i}}$ is relatively closed by the definition of $G$ {in \eqref{eqn:Gdef1}}, by the fact that $G$ is relatively closed,   and the fact that the difference  between $\bigcup_{i \in I(\rho)} 10 {B^{\sigma_i}}$ and its closure lies in $K \cap \partial B(x_0,\rho) \subset L$. \EEE     Then,  ${\Omega_\pm^L} \setminus (G \cup\bigcup_{i \in I(\rho)} {10{{B^\sigma_i}}})$ is open   and the energy of $v_\pm$ is controlled on this set by \eqref{eq: yuhu3}.  Moreover, in view of  \eqref{eq: 10bry} and \eqref{wthf-beta},  we have $S_\pm \subset W_{10} \subset B(x_0,\rho) \cup W^{\rm bdy}_{10}$. Then  \eqref{wthf-beta} and \eqref{eq_XV2-beta} imply \EEE  
    \begin{align}\label{thesameoutside-beta}
        v_\pm = u \quad \text{in} \quad {\Omega_\pm^L} \setminus \big(B(x_0,\rho) \cup \bigcup\nolimits_{i \in I(\rho)} 10 {{B^\sigma_i}}\big).
    \end{align}
    We define $v \in LD(B(x_0,r_0) \setminus G)$ by
    \begin{align*}
        v =
        \begin{cases}
            v_+ &\text{in} \ {\Omega_+^L} \setminus (G \cup\bigcup_{i \in I(\rho)}  {10 {{B^\sigma_i}}})\\
            v_- &\text{in} \ {\Omega_-^L} \setminus (G \cup\bigcup_{i \in I(\rho)} {10 {{B^\sigma_i}}})\\
            u   &\text{in} \ B(x_0,r_0) \setminus \left(G \cup {\Omega_+^L} \cup {\Omega_-^L} \cup \bigcup_{i \in I(\rho)} 10  {{B^\sigma_i}}\right)\\
            0   &\text{in} \ \bigcup_{i \in I(\rho)} 10 B^\sigma_i.
        \end{cases}
    \end{align*}
    This is a well-defined function in $LD(B(x_0,r_0) \setminus G)$ because the piecewise domains in the construction are disjoint open sets which cover $B(x_0,r_0) \setminus G$. 
    Note that $v = u$ outside $B(x_0,9r_0/10)$ due to \eqref{thesameoutside-beta},  $B(x_0,\rho) \subset B(x_0,9r_0/10)$,  and the fact that   $10 B^\sigma_i \subset B(x_0,9r_0/10)$ for all  $i \in I(\rho)$ {(by (\ref{eq: in the next section}) and the fact that $\rho \leq 3/4$).}
    Moreover, by \eqref{eq: GBL} and the fact that $L$ is a separation competitors of $E$ in $B(x_0, \kappa r_0)$, we get $G \setminus B\left(x_0,9r_0/10\right) = E \setminus B\left(x_0,9 r_0/10\right)$.     Thus, the pair $(v,G)$ is   a competitor of {$(u,E)$} in $B(x_0,9r_0/10)$ in {the sense of Lemma \ref{lem:competitor}. Using the lemma,} we compare their energies using   \eqref{eq: yuhu1} and \eqref{eq: yuhu3}--\eqref{thesameoutside-beta}  to find
    \begin{align*}
        \int_{B(x_0,r_0)} \hspace{-0.4cm} & \mathbb{C} e(u)\colon e(u)    \dd{x}  + \HH^{1}({E})  \le          \int_{B(x_0,r_0)} \hspace{-0.4cm}  \mathbb{C} e(v)\colon e(v)    \dd{x} + \HH^{1}(G \cap B(x_0,r_0)) + \eta(x_0,r_0)r_0  + h(r_0)r_0\\
                                          &         \leq C \int_{B(x_0,r_0)} \hspace{-0.4cm} |e(u)|^2    \dd{x} + \HH^{1}(L \cap B(x_0,r_0))   + C \big(\beta_0 m(x_0,r_0) + \eta(x_0,r_0) + h(r_0)\big) r_0,
    \end{align*}    
    and thus, as  $L \setminus B(x_0,\kappa r_0) = E \setminus B(x_0,\kappa r_0)$, we have
    $$
    \HH^{1}\big(E \cap B(x_0, \kappa r_0)\big) \leq \HH^{1}\big(L \cap B(x_0, \kappa r_0)\big) +   C\big(\omega(x_0,r_0) + {\beta_0} m(x_0,r_0) + \eta(x_0,r_0) + h(r_0)\big) r_0.
    $$
    This shows \eqref{eq_F_minimality_defect} and concludes the proof. 
\end{proof}

\subsection{The bad mass}\label{section_badmass}

In this subsection we study the notion of \emph{bad mass} defined in  \eqref{eq_defi_bad_mass} and prove Proposition \ref{prop_badmass_decay}.  We recall that $m$ can   be regarded as a quantification {of} how much $K$ differs from  a Reifenberg-flat set, and we have already seen that it is relevant for measuring the size of wall sets in the application of the extension result.  This notion was introduced by {\sc Lemenant} \cite{Lemenant}  to study the regularity of   the singular set for Mumford--Shah minimizers  near minimal cones, and was also used in the Griffith case \cite{CL}. Note that, in contrast to \cite{CL}, the notion of bad mass needs to be adapted since the definition of good balls \eqref{eq_goodball1}--\eqref{eq_goodball2} involves a minimal separating extension $E(x_0,r_0) \supset K \cap B(x_0,r_0)$. As  $E(x_0,r_0)$ can change with the ball $B(x_0,r_0)$, delicate estimates are necessary to derive decay properties on $m$. As some of these subtleties have been overlooked in \cite{Lemenant}, we take the occasion to provide all details in the present work.

{In the following, we will work under the assumption that, given $x_0 \in K$ and  $r_0 > 0$, we have  $B(x_0,r_0) \subset \Omega$, {$h(r_0) \leq \varepsilon_{\rm Ahlf}$,} and
    \begin{equation}\label{eq_varepsilon00}
        \beta(x_0,r_0) + \eta(x_0,r_0) \leq \bar{\eps}_{\rm mass} :=   \tau/(400 M).
    \end{equation}
    In particular, as $\bar{\eps}_{\rm mass} = \bar{\eps}_{\rm flat}$, the covering $B_i^\sigma $ introduced in \eqref{def_Bi} satisfies \eqref{eq: in the next section}.
}

We start by noting that the definition of the bad mass is not really intrinsic because there could be several minimal separating extensions of $K$. In Remark \ref{rem: almost the same} below,  we see  that different choices of $E(x_0,r_0)$ would only change the bad mass by an error term $\eta$. Since we always need to deal with $\eta$-error terms, the ambiguous definition is actually irrelevant for us. {As another option to} avoid ambiguity, one could also define the bad mass as the infimum of (\ref{eq_defi_bad_mass}) for all possible minimal separating extensions.

We proceed   with a scaling and shifting property for $m$. The scaling of the bad mass is unfortunately not obvious, since for two balls $B(x_0,r_0)$ and $B(x_1,r_1)$ with $B(x_1,r_1) \subset B(x_0,r_0)$, in general we do not have $R(x_1,r_1) \subset R(x_0,r_0)$ for the corresponding sets given in \eqref{eq: the R}. However, we can {prove} the following.

\begin{lemma}[Scaling properties of $m$]\label{rmk_badmass}
    Let $\gamma >0$. Let $x \in K \cap B(x_0,9r_0/10)$ and $r > 0$ be such that $B(x, 9r/10)  \subset B(x_0,9r_0/10)$ and {$r \ge \gamma r_0$}. Then, there exists  $\eps_{\rm mass} \in (0, \bar{\eps}_{\rm mass}] $  \EEE depending on $\gamma$ such that, if  $\beta(x_0,r_0) + \eta(x_0,r_0) \le \eps_{\rm mass}$ and $h(r_0) \le \eps_{\rm Ahlf}$,   we have 
    \begin{equation}\label{1ndsta}
        \HH^1\big(K \cap (R(x,r) \setminus R(x_0,r_0))\big) \leq  Cr_0 \eta(x_0,r_0),
    \end{equation}
    and  
    \begin{equation}\label{2ndsta}
        m(x,r) \leq C \frac{r_0}{r} \big(m(x_0,r_0) + \eta(x_0,r_0)\big),
    \end{equation}
    for some $C \geq 1$ depending only on $\mathbb{C}$. 
\end{lemma}
Before we come to the proof,  let us start with a general observation of a minimal separating extension. From the definition of $\sigma$ in \eqref{defdx} we get  that for all $t \in ({\sigma}(x),r_0/10)$  we have
\begin{equation*}
    \beta_K(x,{t}) \leq \tau \quad \quad \text{ and } \quad \quad      t^{-1} \HH^1\big((E \cap B(x,t)) \setminus K\big) \leq \tau.
\end{equation*}
But if $\sigma(x) > 0$, one can find a sequence of points $(t_k)_k \in (0,\sigma(x))$ such that $t_k \to \sigma(x)$ and either (\ref{eq_goodball1}) fails for all $k$, or (\ref{eq_goodball2}) fails for all $k$.
In the first case, we say that {flatness} (\ref{eq_goodball1}) stops at $x$, and we have
\begin{equation}\label{eq: exact equa}
    \beta_K(x,\sigma(x)) = \tau
\end{equation}
by the scaling properties of $\beta$ in Remark \ref{rmk_beta}. Otherwise, we say that {separation} (\ref{eq_goodball2}) stops at $x$, and we have
\begin{equation}\label{eq: stopppi}
    {\sigma(x)}^{-1} \HH^1\big((E \cap B(x,{\sigma(x)})) \setminus K\big) = \tau
\end{equation}
by regularity properties of measures.

\begin{proof}[Proof of Lemma \ref{rmk_badmass}]
    Let  $x \in {B(x_0,9r_0/10)}$ {and $r \ge \gamma r_0$} be such that $B(x,9 r/10) \subset B(x_0,9r_0/10)$. We choose $\eps_{\rm mass}$ sufficiently small compared to {$\gamma$ and} $\bar{\eps}_{\rm mass}$ defined in \eqref{eq_varepsilon00}, such that  $\beta(x_0,r_0) + \eta(x_0,r_0) \le \eps_{\rm mass}$ implies  $\beta(x,r) + \eta(x,r) \le \bar{\eps}_{\rm mass}$ by the scaling properties of $\beta$ and $\eta$ in Remark \ref{rmk_beta} and Lemma \ref{lem_eta_scaling}.   Let $E$ be a minimal separating extension of $K$ in $B(x_0,r_0)$ and $F$ be  a minimal separating extension of $K$ in $B(x,r)$. We denote the corresponding stopping time functions by $\sigma_E$ and $\sigma_F$, respectively.    For $z \in  K \EEE \cap {B(x,9r/10)}$ such that $\sigma_F(z) > 0$, either {flatness} (\ref{eq_goodball1}) stops at $x$ or {separation} (\ref{eq_goodball2}) stops at $x$. We let $R_1(x,r)$ be the union of the balls $B(z,M\sigma_F(z))$ {for which flatness stops} and $R_2(x,r)$ be the union of the balls $B(z,M\sigma_F(z))$ {for which separation stops}. Note that $R(x,r) = R_1(x,r)\cup R_2(x,r)$, where $R(x,r)$ is defined in \eqref{eq: the R} with respect to $B(x,r)$ and the minimal extension $F$.  Our goal is to check that 
    \begin{align}\label{eq: toooocheck}
        {\rm (i)} \ \ R_1(x,r) \subset R(x_0,r_0), \quad \quad \quad {\rm (ii)} \ \   \HH^1\big(K \cap R_2(x,r)\big)  \leq  Cr_0 \eta(x_0,r_0).
    \end{align}
    Then, both statements  \eqref{1ndsta}--\eqref{2ndsta} indeed   follow from  the definition of $m$, see   \eqref{eq_defi_bad_mass}.

    {We focus first on (i). For  a center $z$  of a ball in $R_1(x,r)$, the flatness stops at $z$, so there are {radii $t \in (0,\sigma_F(x))$ arbitrarily} close to $\sigma_F(x)$ such that $\beta_K(x,t) > \tau$. (Note that the flatness in \eqref{eq_goodball1} refers to $K$ and not to the extensions.) Such a ball $B(x,t)$ cannot be a good ball for $\sigma_E$, {and so} $\sigma_E(x) \geq t$. This shows that $\sigma_F(x) \leq \sigma_E(z)$ and therefore $R_1(x,r) \subset R(x_0,r_0)$.}

    We now address (ii).   Repeating the argument in  \eqref{def_Bi}, we apply the Vitali's covering lemma to obtain a disjoint family of balls $(B(z_i,M\sigma_F(z_i)))_{i \in I}$ {for which separation stops} such that $R_2(x,r) \subset \bigcup_{i \in I} 4 B(z_i,M \sigma_F(z_i))$.
    We observe by Ahlfors-regularity of $K$,  see \eqref{eqn:AhlforsReg},   that
    \begin{equation}\label{eq: ccccc1}
        \HH^1(K \cap R_2(x,r)) \le  \sum_{i \in I}   \HH^1\big(K \cap 4 B(z_i,M \sigma_F(z_i)) \big)    \leq C_{\rm Ahlf}  \sum_{i \in I} 4M\sigma_F(z_i).
    \end{equation}
    On the other hand, as {separation} (\ref{eq_goodball2}) stops at $z_i$, for each $ i \in I$, \eqref{eq: stopppi} implies  
    \begin{equation*}
        \HH^1\big(\big(F \cap B(z_i,\sigma_F(z_i))\big)  \setminus K\big) = \tau \sigma_F(z_i).
    \end{equation*}
    Recall  that the balls $(B(z_i,M\sigma_F(z_i)))_{i \in I}$ are disjoint and contained in $B(x,r)$ due to \eqref{eq: in the next section}  applied for $\sigma_F$ on $B(x,r)$.  Thus,  we can estimate 
    \begin{equation}\label{eq: ccccc2}
        \sum_{i \in I} \sigma_F(z_i) \leq C \HH^1\big( (F \setminus K)  \cap B(x,r) \big) \leq C r\eta(x,r),
    \end{equation}
    {where $C =C(\tau) \geq 1$ is a constant that depends on $\tau$}.
    Combining  \eqref{eq: ccccc1} and \eqref{eq: ccccc2}, and the scaling property of $\eta$ (see  Lemma \ref{lem_eta_scaling} and pass to a smaller $\eps_{\rm mass}$ if necessary),  we get
    $$    \HH^1\big(K \cap R_2(x,r)\big) \le  Cr_0 \eta(x_0,r_0).    $$
    This shows  \eqref{eq: toooocheck}(ii) and concludes the proof. 
\end{proof}

\begin{remark}\label{rem: almost the same}
    {\normalfont
        The previous proof, applied for $x=x_0$ and $r = r_0$ for a different minimal extension $F$ and a corresponding stopping time $\sigma_F$, shows 
    $ \HH^1\big(K \cap (R_F(x_0,r_0) \triangle R(x_0,r_0))\big) \leq  Cr_0 \eta(x_0,r_0)$, where $R_F(x_0,r_0)$ denotes the set in \eqref{eq: the R} with respect to $F$, and $\triangle$ is the symmetric difference of sets.  This shows that different choices of the minimal extension lead to bad masses which differ at most by an error of order $\eta(x_0,r_0)$. }
\end{remark}

The above properties imply the part of  Lemma  \ref{lemma: scaling properties} regarding the bad mass.

\begin{proof}[Proof of Lemma  \ref{lemma: scaling properties}]
    We need to check there exists $C_{\rm shift}$ such that for $x_0 \in K$, $r_0 > 0$ with $B(x_0,r_0) \subset \Omega$,  and $x \in K \cap  B(x_0,r_0/ 4) \EEE $,   it holds that 
    $$m(x, r_0/2) \le C_{\rm shift} \big( m(x_0,r_0) + \eta(x_0,r_0)\big).$$
    By assumption we have $\eta(x_0,r_0) + \beta(x_0,r_0) \le  \eps_{\rm mass}$ with $\eps_{\rm mass}$ from Lemma \ref{rmk_badmass} for $\gamma=1/2$. Then, the estimate is a consequence of Lemma \ref{rmk_badmass}, for  $C_{\rm shift} \ge  2C$.
\end{proof}

{We now show that, when the flatness (\ref{eq_goodball1}) stops,} we can replace $E$ by a separation competitor, see Definition \ref{def: sep comp},  which has less measure in a quantified way. This allows to estimate the bad mass in a ball $B(x_0,r_0)$.  

\begin{lemma}[Favorable separation competitors and control on bad mass]\label{lem_competitor}
    Let $(u,K)$ be a Griffith almost-minimizer. There exist constants $\theta \in (0,1)$ and  $C \ge 1$ both {depending} on $\mathbb{C}$ such that the following holds.    Let $x_0 \in K$ and  $r_0 > 0$ be such that $B(x_0,r_0) \subset \Omega$, $h(r_0) \leq \varepsilon_{\rm Ahlf}$, and  $\beta(x_0,r_0) + \eta(x_0,r_0) \leq {\eps}_{\rm mass}$, where $\eps_{\rm mass}$ denotes the constant of Lemma \ref{rmk_badmass} applied for {$\gamma = 1/4$}.   We let $E$ be a minimal separating extension of $K$ in $B(x_0,r_0)$, and we let $(B^\sigma_i)_{i\in I}$ be the family of balls built in Subsection~\ref{sec: stop time} with centers $(x^\sigma_i)_{i \in I}$ and  radii  $ r_i^\sigma= M\sigma(x^\sigma_i)$.    We let ${I_1}$ denote the set of indices $i$ such that $4 B^\sigma_i \cap B(x_0,r_0/4) \ne \emptyset$ and {flatness (\ref{eq_goodball1}) stops} at $x^\sigma_i$.   

    Then, for all $i \in I_1$, there exists a separation competitor $L^\sigma_i \subset B(x_0,r_0)$ of $E$ in $D^\sigma_i := B(x^\sigma_i,4 \sigma(x^\sigma_i))$ with
    \begin{equation*}
        \HH^1(L_i^\sigma \cap D^\sigma_i) \le   \HH^1(E \cap D^\sigma_i)  -  \theta r^\sigma_i,
    \end{equation*}
    and we have
    \begin{equation*}
        m(x_0,r_0/4) \leq C \left(r_0^{-1} \sum_{i \in {I_1}} r^\sigma_i + \eta(x_0,r_0)\right).
    \end{equation*}
\end{lemma}
\begin{proof}
    For $i \in {I_1}$, we have $4 \sigma(x^\sigma_i) \in (\sigma(x^\sigma_i),r_0/10]$ by the observation after \eqref{eq_ri_estimate}.   Thus, by Lemma \ref{eq_betaF} we get that
    \begin{equation}\label{eq: imorta}
        \beta_E\big(x^\sigma_i, t\big) \leq 4 \tau \quad  \text{for all $t \in [4\sigma  (x^\sigma_i), r_0/10]$}, \EEE
    \end{equation}
    and that $E$ {separates} $D^\sigma_i$ by Lemma \ref{lem_separation}(2).  On the other hand, {flatness (\ref{eq_goodball1}) stops} at $x_i$, so
    \begin{equation}\label{eq_tau_fail}
        \tau  = \beta_K(x^\sigma_i,\sigma(x^\sigma_i)) \leq \beta_E(x^\sigma_i,\sigma(x^\sigma_i)).
    \end{equation}
    (The first identity follows from \eqref{eq: exact equa}.)     We proceed by contradiction and assume that,  for all separation competitors $L$ of $E$ in $D^\sigma_i = B(x^\sigma_i,4 \sigma(x^\sigma_i))$, we have
    \begin{equation*}
        \HH^1({L} \cap D^\sigma_i) > \HH^1(E \cap D^\sigma_i) -\lambda \cdot 4\sigma(x^\sigma_i) \quad \quad \text{for $\lambda :=    C_{\rm Ahlf}^{-1} \EEE \left(\frac{\tau}{ 64}  \right)^2$. }
    \end{equation*}
    Then, since $\beta_E\big(x^\sigma_i, 4\sigma  (x^\sigma_i)\big) \leq 4 \tau$  by \eqref{eq: imorta} and $\tau \le  \eps_{\rm def}/4$,   Lemma \ref{lem_flatness} is applicable and we get  
    \begin{equation}\label{eq: CCCCC}
        \beta_E(x^\sigma_i,\sigma(x^\sigma_i)) \leq   32   \sqrt{C_{\rm Ahlf}} \EEE \sqrt{\lambda} = \tau/2. 
    \end{equation}
    Here, we note that  \eqref{eq_AF_flatness} holds for $C_0 =  C_{\rm Ahlf}$ by  Lemma \ref{lem_AFF}. Now, \eqref{eq: CCCCC} contradicts (\ref{eq_tau_fail}).   We thus deduce that there exists a separation competitor $L^\sigma_i$ of $E$ in $D^\sigma_i$ such that
    \begin{equation*}
        \HH^1\big (L^\sigma_i \cap D^\sigma_i \big)  \le        \HH^1\big(E \cap D^\sigma_i\big)  - \theta  r_i^\sigma
    \end{equation*}
    for $\theta = 4\lambda/M $, where we use that $ r_i^\sigma= M\sigma(x^\sigma_i)$.

    Now, we pass to the second part of the lemma and we estimate $\HH^1(K \cap R(x_0,r_0/4))$. Since $R(x_0,r_0) \subset \bigcup_i 4   B^\sigma_i  $, see \eqref{def_Bi}, and $R(x_0,r_0/4) \subset B(x_0,r_0/4)$ by definition,    we have
    \begin{align}\label{eq: combos1}
        K \cap R(x_0,r_0/4)  \subset \Big(  K \cap \bigcup_{ B^\sigma_i \in \mathcal{B}}   4 B^\sigma_i \Big) \cup \big(K \cap (R(x_0,r_0/4) \setminus R(x_0,r_0))\big),
    \end{align}
    where $\mathcal{B} = \lbrace  B^\sigma_i \colon 4 B^\sigma_i \cap B(x_0,r_0  /4) \EEE \ne \emptyset \rbrace$. 
    We partition the balls $\mathcal{B} = \mathcal{B}^1 \cup \mathcal{B}^2 $, where for $\mathcal{B}^1$  {flatness (\ref{eq_goodball1}) stops}  and  for $\mathcal{B}^2$  {separation (\ref{eq_goodball2}) stops}. By {\eqref{eqn:AhlforsReg}}, we have
    \begin{align}\label{eq: combos2}
        \sum_{B^\sigma_i \in \mathcal{B}^1}\mathcal{H}^1\big(K \cap 4 B^\sigma_i\big) = \sum_{i \in {I_1}}\mathcal{H}^1\big(K \cap 4 B^\sigma_i\big) \le  \sum_{i \in {I_1}} C r_i^\sigma.  
    \end{align}
    For the balls in $\mathcal{B}^2$ instead, {by \eqref{eq: stopppi}} we have 
    \begin{equation*}
        \HH^1\big( \big(E  \cap B(x^\sigma_i,\sigma(x^\sigma_i))\big) \setminus K\big) =  \tau \sigma(x^\sigma_i) =  \frac{\tau}{M} r^\sigma_i.
    \end{equation*}
    Since the balls $(B^\sigma_i)_i$ are disjoint  and contained in $B(x_0,r_0)$ by \eqref{eq: in the next section}, \EEE we conclude {by \eqref{eqn:AhlforsReg} and  the definition of $\eta$ that}
    \begin{align}\label{eq: combos3}
        \sum_{B^\sigma_i \in \mathcal{B}^2}\mathcal{H}^1\big(K \cap 4 B^\sigma_i\big)  & \le C_{\rm Ahlf} \sum_{B^\sigma_i \in \mathcal{B}^2} 4r_i^\sigma \le   C   \sum_{B^\sigma_i \in \mathcal{B}^2}   \HH^1\big( (E  \cap B(x_i,\sigma(x^\sigma_i))) \setminus K\big)  \notag \\&    \leq C \HH^1\big((E \cap B(x_0,r_0)) \setminus K\big)  \leq  C\eta(x_0,r_0) r_0,
    \end{align}
    {where $C = C(M,\tau) \geq 1$ is a constant} only depending on $\mathbb{C}$.       By combining \eqref{eq: combos1}--\eqref{eq: combos3} and recalling that by  Lemma \ref{rmk_badmass}  (for $\gamma=1/4$)  
    $$\HH^1\big(K \cap (R(x_0,r_0/4) \setminus R(x_0,r_0))\big) \leq C r_0  \eta(x_0,r_0)$$   
    we conclude the proof.  
\end{proof}

We now conclude with the main result of this subsection,   the control on the bad mass.

\begin{proof}[Proof of Proposition \ref{prop_badmass_decay}]
    Given $\gamma>0$ as in the statement, we choose the constant  $\eps_{\rm mass}$ as in Lemma~\ref{rmk_badmass}  and recall \EEE $\eps_{\rm mass} \le \bar{\eps}_{\rm mass}$, where $\bar{\eps}_{\rm mass}$ is given in \eqref{eq_varepsilon00}. We let $\theta \in (0,1)$ be the constant from Lemma~\ref{lem_competitor}. We suppose that   $\beta(x_0,r_0) + \eta(x_0,r_0)\leq  {\eps}_{\rm mass}$, {$h(r_0) \leq \varepsilon_{\rm Ahlf}$,} and we let $E$  be a minimal separating extension of $K$ in $B(x_0,r_0)$. Moreover, let    $(B^\sigma_i)_{i \in I}$ be the family of balls constructed in Subsection \ref{sec: stop time}.      To simplify the notation, we let $\beta_0 := \beta(x_0,r_0)$.

    \emph{Step 1: Construction of separation competitor $L$.}   By Lemma \ref{lem_competitor} we get
    \begin{equation}\label{eq_mJ}
        m(x_0,r_0/4) 
        < C \left(r_0^{-1} \sum_{i \in I_1}  r^\sigma_i \EEE + \eta(x_0,r_0)\right),
    \end{equation}
    where  ${I_1}$ denote the set of indices $i \in I$ such that $4 B^\sigma_i \cap B(x_0,r_0/4) \ne \emptyset$ and {flatness (\ref{eq_goodball1}) stops} at $x^\sigma_i$.  Recall  that $\overline{B^\sigma_i} \subset B(x_0,r_0/2)$ for $i \in {I_1}$ since $r^\sigma_i < r_0/200$ by \eqref{eq: in the next section}.  We can choose a finite subset of indices ${\tilde{I}_1 \subset I_1}$ such that \eqref{eq_mJ} still holds {(as the inequality is strict)} for ${\tilde{I}_1}$ in place of ${I_1}$. For convenience, we denote this index set still by ${I_1}$.   For all $i \in {I_1}$, Lemma \ref{lem_competitor} yields a separation competitor $L^\sigma_i \subset B(x_0,r_0)$ of $E$ in the ball $D^\sigma_i = B(x^\sigma_i, 4 \sigma(x^\sigma_i))  \subset \subset \EEE B^\sigma_i$ such that
    \begin{equation}\label{eq_Li}
        \HH^1\big(L^\sigma_i \cap D^\sigma_i\big) \leq \HH^1\big(E \cap D^\sigma_i\big) -  \theta r^\sigma_i.
    \end{equation}
    We define  the closed set 
    \begin{equation*}
        L := \left(E \setminus \bigcup_{i \in {I_1}} D^\sigma_i\right) \cup \bigcup_{i \in {I_1}} \left(L^\sigma_i \cap D^\sigma_i\right).
    \end{equation*}
    The balls $(\overline{D^\sigma_i})_{i \in {I_1}}$ are pairwise disjoint and contained in $B(x_0,r_0/2)$. Therefore,   $L$ coincides with $E$ outside $B(x_0,r_0/2)$. Moreover, since $(L^\sigma_i)_i$ are separation competitors in $D^\sigma_i$, we find that $L$ is a separation competitor of $E$ in  $B(x_0,r_0/2)$. \EEE   

    By the definition of the flatness and the fact that $E$ {separates} $B(x_0,r_0)$ there exists a line $\ell_0$ passing through $x_0$  with unit normal  $\nu(x_0,r_0)$    such that $E \cap B(x_0,r_0) \subset \lbrace x \colon \mathrm{dist}(x,\ell_0) \leq \beta_0 r_0 \rbrace$ and such that the sets $D_{\beta_0r_0}^+(x_0,r_0)$ and $D_{\beta_0r_0}^-(x_0,r_0)$ defined in \eqref{eq :Bx0} 
    belong to different connected components of $B(x_0,r_0) \setminus E$, denoted by {$\Omega_+^E$ and $\Omega_-^E$}. Since $L$ is a separation competitor of $E$ in $B(x_0,  r_0)  $ coinciding with $E$ outside $B(x_0,r_0/2)$,  the sets    $D_{r_0/2}^+(x_0,r_0)$ and $D_{r_0/2}^-(x_0,r_0)$ (with respect to the same normal $\nu(x_0,r_0)$) 
    also lie in different connected   components \EEE of $B(x_0,r_0) \setminus L$, denoted by {$\Omega_+^L$ and $\Omega_-^L$} respectively.
    We clearly have   
    \begin{equation}\label{eq_X}
        {\Omega_\pm^L \subset \Omega_\pm^E} \cup \bigcup_{i \in {I_1}} D^\sigma_i, \quad \quad \quad  {\Omega_\pm^L} \setminus B(x_0,r_0/2) = {\Omega_\pm^E} \setminus B(x_0,r_0/2).
    \end{equation}
    To conclude this step, we use (\ref{eq_mJ}) and  (\ref{eq_Li})  to estimate
    \begin{align}
        \HH^1\big(L \cap B(x_0,r_0)\big)    &\leq \HH^1\big(E \cap B(x_0,r_0)\big) - \theta \sum_{i \in {I_1}} r^\sigma_i\notag\\
                                            &\leq \HH^1\big(E \cap B(x_0,r_0)\big) - \frac{\theta}{C} m(x_0,r_0/4) r_0  +   \eta(x_0,r_0) r_0.\label{eq_L}
    \end{align}

    \emph{Step 2: Wall set.} Before we can construct a competitor function related to the separation competitor $L$, we need to introduce an additional wall set as motivated at the end of Subsection \ref{sec: stop time}. We use Lemma~\ref{lem_goodradius} (applicable since $\beta(x_0,r_0) + \eta(x_0,r_0) \leq  \eps_{\rm mass} \le   \bar{\eps}_{\rm mass}  { = \bar{\eps}_{\rm flat}}$) to select a radius $\rho \in [r_0/2, 3 r_0/4]$ such that 
    \begin{equation}\label{eq_choice_rho}
        \sum_{i \in I(\rho)} r^\sigma_i  \leq C \big(\beta_0m(x_0,r_0) + \eta(x_0,r_0)\big)  r_0,
    \end{equation}
    where $I(\rho) = \big\{i \in I \colon \,  10 B^\sigma_i \cap \partial B(x_0,\rho) \ne \emptyset \big\}$. Then, we introduce
    \begin{equation*}
        G := L \cup \bigcup_{i \in I(\rho)} \partial (10 B^\sigma_i).
    \end{equation*}
    We observe that
    \begin{equation}\label{eq: auch das}
        G \setminus B\left(x_0, 9r_0/10\right) = E \setminus B\left(x_0, 9r_0/10\right)
    \end{equation}
    because $L$ coincides with $E$ outside $B(x_0,r_0/2)$ and for all $i \in I(\rho)$ we have $\overline{10 B^\sigma_i} \subset B(x_0, 9 r_0/10)$ by \eqref{eq: in the next section}. This also implies  $K \cap \partial B(x_0,\rho) \subset E \cap \partial B(x_0,\rho) \subset L$. This along with   \eqref{2ndpro} in Corollary \ref{cor: wall set} shows that $G$   is relatively closed in $\Omega$.    To conclude this step, we use (\ref{eq_L}) and (\ref{eq_choice_rho}) to estimate
    \begin{align}\label{eq_energy_G}
        \HH^{1}\big(G \cap B(x_0,r_0)\big) &\leq \HH^{1}\big(L \cap B(x_0,r_0)\big) + \sum_{i \in I(\rho)} \HH^{1}\big(\partial (10 B^\sigma_i)\big) \\
                                           &\leq \HH^{1}( E )  - \frac{\theta}{C} m(x_0,r_0/4) r_0 + C \beta_0 m(x_0,r_0) r_0            + C \eta(x_0,r_0) r_0.\notag
    \end{align}

    \emph{Step 3: Construction of competitor.}      Now, we build a function $v \in LD(B(x_0,r_0) \setminus G)$ such that $(v,G)$ is a competitor of $(u,K)$ in $B(x_0,r_0)$ and the energy of $v$ is controlled by 
    \begin{equation}\label{eq_energy_v}
        \int_{B(x_0,r_0) \setminus G} \abs{e(v)}^2 \dd{x} \leq C \int_{B(x_0,r_0) \setminus E} \abs{e(u)}^2 \dd{x}.
    \end{equation}
    We apply Proposition \ref{th: extension} to  $(u,E)$, the radius $\rho \in [r_0/2, 3r_0/4]$ selected above, and  the geometric function $\delta$ defined in (\ref{eq_delta}) with parameters $(\rho,16\tau)$, cf.\ Lemma \ref{lem_delta}. (Note that Proposition \ref{th: extension} is applicable {with $16 \tau$ in place of $\tau$} as  $\beta_0 \le 10^{-8}/8$,  $16 \tau \le10^{-8}$, and $8\beta_0 \le 8\eps_{\rm mass} \le 16\tau$  by \eqref{eq_varepsilon00} and $\eps_{\rm mass} \le \bar{\eps}_{\rm mass}$.)

    {In the extended domains $\Omega^{\rm ext}_{\pm} := \Omega^E_{\pm} \cup W$}, we obtain functions $v_\pm \in LD_{\mathrm{loc}}(\Omega^{\rm ext}_\pm)$, and relatively closed subsets {$S_\pm \subset \Omega^{\rm ext}_{\pm}$} such that
    \begin{equation}\label{wthf}
        W \subset S_\pm \subset W_{10}, \quad \quad v_\pm = u \ \text{in} \ \Omega^{\rm ext}_\pm \setminus S_\pm, 
    \end{equation}
    and
    \begin{equation}\label{eq_XV2XXX}
        \int_{(\Omega^{\rm ext}_\pm \cap B(x_0,\rho)) \setminus W_{10}^{\rm bdy}} \abs{e(v_\pm)}^2 \dd{x} \leq C \int_{B(x_0,\rho) \cap \Omega^E_\pm} \abs{e(u)}^2 \dd{x}.
    \end{equation}
    We note that 
    \begin{equation}\label{eq_XV2}
        {\Omega_\pm^L} \subset \Omega^{\rm ext}_\pm = {\Omega_\pm^E} \cup W, \quad \quad \quad  W_{10}^{\rm bdy} \subset \bigcup_{i \in I(\rho)} 10 B^\sigma_i.
    \end{equation}
    The second inclusion was proven in Corollary \ref{cor: wall set}.
    {To show the first inclusion, in view of (\ref{eq_X}) and the definition of $W$ in \eqref{eq: Wdef}, it suffices to show that  $ D^\sigma_i = B(x^\sigma_i, 4 \sigma(x^\sigma_i)) \subset B(x^\sigma_i, 10^{-5} \delta(x^\sigma_i))$. With Lemma~\ref{lem_delta}(1) we indeed find   $10^{-5} \delta(x^\sigma_i) = 10^{-5} r^\sigma_i =  10  \sigma(x^\sigma_i)$, where $r^\sigma_i$ was defined before \eqref{def_Bi} (recall that $M =10^6$).}

    We note that $G \cup\bigcup_{i \in I(\rho)} {10 {B^\sigma_i}}$ is relatively closed by the definition of $G$ and the fact that $G$ is relatively closed. Moreover, in view of  \eqref{eq: 10bry} and \eqref{wthf},  we have $S_\pm \subset W_{10} \subset B(x_0,\rho) \cup W^{\rm bdy}_{10}$.   Then   \eqref{wthf} and \eqref{eq_XV2} \EEE imply
    \begin{align}\label{thesameoutside}
        v_\pm = u \quad \text{in} \quad {\Omega_\pm^L} \setminus \big(B(x_0,\rho) \cup \bigcup\nolimits_{i \in I(\rho)} 10 {{B^\sigma_i}}\big).
    \end{align}
    We now define $v \in LD(B(x_0,r_0) \setminus G)$ by
    \begin{align*}
        v =
        \begin{cases}
            v_+ &\text{in} \ {\Omega_+^L} \setminus (G \cup\bigcup_{i \in I(\rho)}  {10 {{B^\sigma_i}}})\\
            v_- &\text{in} \ {\Omega_-^L} \setminus (G \cup\bigcup_{i \in I(\rho)} {10 {{B^\sigma_i}}})\\
            u   &\text{in} \ B(x_0,r_0) \setminus \left(G \cup {\Omega_+^L \cup \Omega_-^L} \cup \bigcup_{i \in I(\rho)} {10 {{B^\sigma_i}}}\right)\\
            0   &\text{in} \ \bigcup_{i \in I(\rho)} 10 B^\sigma_i.
        \end{cases}
    \end{align*}
    This is a well-defined function in $LD(B(x_0,r_0) \setminus G)$ because the piecewise domains in the construction are disjoint open sets which cover $B(x_0,r_0) \setminus G$. Estimate \eqref{eq_energy_v}  follows  from  \eqref{eq_XV2XXX}--\eqref{thesameoutside}. \EEE

    Note that $v = u$ outside $B(x_0,9r_0/10)$ due to \eqref{thesameoutside},  $B(x_0,\rho) \subset B(x_0,9r_0/10)$,  and the fact that   $10 B^\sigma_i \subset B(x_0,9r_0/10)$ for all  $i \in I(\rho)$  (by  \EEE \eqref{eq: in the next section} {and $\rho \leq 3r_0/4$}).  This along with \eqref{eq: auch das} shows that  the pair $(v,G)$ is   a competitor of $(u,K)$ in $B(x_0,9r_0/10)$ in {the sense of Lemma \ref{lem:competitor}. Using the lemma,} we can compare their energies by using \eqref{eq_energy_G} and \eqref{eq_energy_v} to obtain
    \begin{align*}
        \int_{B(x_0,r_0)} \hspace{-0.4cm} \mathbb{C}  & e(u)\colon e(u)     \dd{x}  + \HH^{1}(E)   \le          \int_{B(x_0,r_0)} \hspace{-0.4cm}  \mathbb{C} e(v)\colon e(v)    \dd{x} + \HH^{1}(G \cap B(x_0,r_0)) + \eta(x_0,r_0)r_0 +   h(r_0)r_0\\
                                                      &         \leq C \int_{B(x_0,r_0)} \hspace{-0.4cm} |e(u)|^2    \dd{x} + \HH^{1}(E)    -  \frac{\theta}{C} m(x_0,r_0/4) r_0 + C \big(\beta_0 m(x_0,r_0) + \eta(x_0,r_0) + h(r_0)\big) r_0.
    \end{align*}
    {With this, we} find
    $$ m(x_0,r_0/4) \le C\omega(x_0,r_0) +   C \big(\beta(x_0,r_0) m(x_0,r_0) + \eta(x_0,r_0) + h(r_0)\big). $$
    This shows part \eqref{eq: mass scale-wo} of the  statement. For general $b \in [\frac{1}{4}\gamma,\frac{1}{4}]$, we use Lemma \ref{rmk_badmass} for $r_0/4$ in place of $r_0$ and Lemma \ref{lem_eta_scaling} {to} {obtain \eqref{eq: mass scale}}. 
\end{proof}

\subsection{Decay of the elastic energy}\label{sec: energydecay}

This subsection is devoted   entirely \EEE to the proof of Proposition \ref{prop_energy_decay}.

\begin{proof}[Proof of Proposition \ref{prop_energy_decay}]
    We can directly assume $0 < b < 1/2$, otherwise (\ref{eq_energy_decay0}) is trivial by Remark \ref{rmk_beta}. It suffices to prove {the estimate $$\omega_n(x_0,b r_0)< C_* b \omega_n(x_0,r_0) $$ for a chosen $C_*>0$ depending only on $\mathbb{C}$ and a chosen} $\eps_{\rm el}$ depending only on $b$ {and $\mathbb{C}$} {with}   $\beta(x_0,r_0) \le \eps_{\rm el}$ under the additional assumption  that
    \begin{align}\label{good assu}
        {m(x_0,r_0)\beta(x_0,r_0) + \eta(x_0,r_0) + h(r_0) \leq  \eps_{\rm el} \omega(x_0,r_0).}
    \end{align}
    Indeed, if \eqref{good assu} does not hold, \eqref{eq_energy_decay0} follows by Remark \ref{rmk_beta} and by  choosing $C^b_{\rm el}$ large enough depending on $b$. The proof is divided into several steps and is based on a contradiction argument. {Unraveling the structure of the contradiction argument, we will repeatedly reduce the statement to simpler `Claims' which build back up to the initial claim, in the sense that (Claim $i+1$) will imply (Claim $i$).}

    \emph{Step 1: Competitor sequence for contradiction and main strategy of the proof.}    Our goal is to argue by contradiction. We assume that for some constant $C_* \geq 1$ to be specified below, there exists a sequence of Griffith almost-minimizers $({u_n},K_n)$ with gauge {$h_n$} in $B(x_n,r_n)$ such that $x_n \in K_n$, {$r_n > 0$}, {$h_n(r_n) \leq \varepsilon_{\rm Ahlf}$},
    \begin{equation}\label{ineq01}
        \beta_n(x_n,r_n)   \to 0, \quad \quad \quad         \omega_n(x_n,r_n)^{-1} \big(\beta_n(x_n,r_n) m_n(x_n,r_n)    + \eta(x_n,r_n)    + {h_n(r_n)}\big) \to 0,
    \end{equation}
    and 
    \begin{equation}
        \omega_n(x_n,b r_n)> C_* b \omega_n(x_n,r_n),\label{ineq1}
    \end{equation}
    where {$\beta_n(x_n,r_n),$ etc.,} are the shorthands for the quantities  defined with respect to $(u_n,K_n)$.  Our goal is to find a contradiction for a constant $C_* \geq 1$ large enough {which depends on $\mathbb{C}$} but is \emph{independent of $b$.}  Note that for all $n$, $K_n$ is Ahlfors-regular in $B(x_n,r_n)$, with a constant that does not depend on $n$.

    We rescale the sequence $({u_n},K_n)$ to the unit ball $B := B(0,1)$  via ${\tilde u_n} := {r_n^{-1/2}} {u_n}( r_n \cdot  +x_n)$ and $\tilde{K}_n := r_n^{-1}(  K_n - x_n)$. 
    {The pair $(\tilde{u}_n,\tilde{K}_n)$ is an almost-minimizer in $B$ with gauge ${\tilde{h}_n(\cdot) := h_n( r_n \cdot)}$.}  {For simplicity, we denote  the rescaled crack still by $K_n$}.     Note that this does not affect the flatness $\beta_n$, the normalized {elastic} energy $\omega_n$, the normalized {length} of holes $\eta_n$, and the bad mass $m_n$ since the quantities  are scale invariant, see Remarks \ref{rem: normalization} and \ref{norm2}.   In particular, we have
    \begin{equation}\label{ineq01-new}
        0 \in K_n, \quad \quad \beta_n + \eta_n \to 0, \quad \quad \quad         \omega_n^{-1} \big(\beta_n m_n   +\eta_n   + \tilde{h}_n(1)\big) \to 0,
    \end{equation}
    where for brevity we write $\beta_n := \beta_n(0,1)$, $\eta_n := \eta_n(0,1)$,  $\omega_n:= \omega_n(0,1)$, and  $m_n := m_n(0,1)$ {with the quantities now computed with respect to $(\tilde u_n,\tilde K_n)$.} Note that $\eta_n \to 0$ indeed follows from \eqref{ineq01} since $\omega_n$ is bounded from above by \eqref{eqn:AhlforsReg2}.
    {The inequality \eqref{ineq1} {becomes} $\omega_n(0,b) > C_* b \omega_n$, and recalling the definition of $\omega$ in \eqref{eq: main omega}, this means
        \begin{equation}
            \frac{1}{b}\int_{B(0,b) \setminus K_n} |e({\tilde{u}_n})|^2 \, \dd x > C_* b \int_{B \setminus K_n} |e({\tilde{u}_n})|^2 \, \dd x.\label{ineq1-new}
    \end{equation}}
    For all $0 < \rho < 1$ and for all {competitors}  $({v},G)$ \EEE of $({\tilde{u}_n},K_n)$ in $B(0,\rho)$, we have by  \eqref{amin}
    \begin{multline}\label{eq: competitor as before}
        \int_{B(0,\rho) \setminus K_n} \mathbb{C} e({\tilde{u}_n})\colon e({\tilde{u}_n})    \dd{x} + \HH^1(K_n \cap B(0,\rho)) \\\leq \int_{B(0,\rho) \setminus G} \mathbb{C}   e({v})\colon e({v}) \EEE    \dd{x} + \HH^1(  G \EEE \cap B(0,\rho)) + \tilde{h}_n(1). 
    \end{multline}
    Let $\ell_n$ be {a line approximating} $K_n$ in $B$   in the sense of   \eqref{eq_beta}. {For simplicity, we will assume that $\ell_n \equiv \mathbb{R}e_1$ {for all $n$}. As the {bulk energy $\int \mathbb{C} e(u) : e(u) \dd{x}$} is not rotation invariant, one must {in principle} take a convergent subsequence of lines, but this does not create any technical difficulty, so we omit this detail.}

    We normalize the {elastic} energy by introducing the normalized functions 
    \begin{equation*}
        {u_n^{\rm norm}}(x):= \frac{1}{\sqrt{\omega_n}} {\tilde{u}_n}(x),
    \end{equation*}
    where we note by  \eqref{eq: main omega}  and   \eqref{ineq1-new} \EEE  that  $\omega_n =  \int_{B \setminus K_n} |e({\tilde{u}_n})|^2 >0$. {By an abuse of notation, we refer to the normalized function ${u_n^{\rm norm}}$ by $u_n$. {Keeping the definitions of the constants given by $\beta_n$, etc., as in \eqref{ineq01-new}, throughout} the rest of the proof we will only be interested in the \textit{normalized functions} $u_n$ and \textit{rescaled cracks} $K_n$. In summary, by} \eqref{ineq1-new} and \eqref{eq: competitor as before}, {for all $n$,} $u_n$ belongs  to $LD(B \setminus K_n)$ and satisfies  
    \begin{equation}
        \int_{B \setminus K_n}|e(u_n)|^2 = 1,\label{energyBound}
    \end{equation}
    as well as     the estimate
    \begin{equation}\label{eq_contradiction}
        \int_{B(0,b)\setminus K_n}|e(u_n)|^2 > C_* b^2,
    \end{equation}
    and for all $0 < \rho < 1$ and for all {competitors} $({v},G)$ of $(u_n,K_n)$ in $B(0,\rho)$,   it holds that
    \begin{multline}\label{eq_Kn_minimality}
        \int_{B(0,\rho) \setminus K_n} \mathbb{C} e(u_n)\colon e(u_n)    \dd{x} + \omega_n^{-1} \mathcal{H}^1(K_n \cap B(0,\rho)) \\\leq  \int_{B(0,\rho) \setminus G} \mathbb{C} e({v})\colon e({v})    \dd{x} + \omega_n^{-1} \mathcal{H}^1(G \cap B(0,\rho)) + \omega_n^{-1} \tilde{h}_n(1). 
    \end{multline}

    The strategy of the proof is now as {follows}. {Define $\ell_0 := \R e_1$.}
    We show that  $u_n$ (up to substracting rigid motions)  converges to a limit $u \in H^1(B \setminus \ell_0;\R^2)$, which  is {locally (away from $\partial B(0,9/10)$)} a weak solution of $\mathrm{div}({\mathbb{C}}e(u)) = 0$ in $B(0,9/10) \setminus \ell_0$ {with a homogeneous Neumann condition} on $\ell_0$. By elliptic regularity we will {deduce}
    that for $0 <b < 1/2$, 
    \begin{equation}\label{eq_elliptic_regularity}
        \int_{B(0,b) \setminus \ell_0} \abs{e(u)}^2 \, \dd x \leq C b^2 \int_{B \setminus \ell_0} \abs{e(u)}^2  \, \dd x \le C b^2, \tag{Claim 1.a}
    \end{equation}
    for some constant $C \geq 1$ {which depends on $\mathbb{C}$ but not on  $b$}.

    The remaining challenge consists  in \EEE  proving the strong convergence 
    \begin{equation}\label{eq_energy_decay_goal}
        \lim_{n \to +\infty} \int_{B(0,b) \setminus K_n} \abs{e(u_n)}^2  \, \dd x  = \int_{B(0,b) \setminus \ell_0} \abs{e(u)}^2  \, \dd x. \tag{Claim 1.b}
    \end{equation} 
    Once this is achieved, \eqref{eq_elliptic_regularity} and  (\ref{eq_contradiction}) along with \eqref{eq_energy_decay_goal}  contradict each other for {$C_*$ chosen sufficiently large depending only on $C$ in \eqref{eq_elliptic_regularity} and, in particular, independent of $b$. Consequently proving \eqref{eq_elliptic_regularity} and \eqref{eq_energy_decay_goal} will conclude the proof.}

    In Step 2, we obtain the limit $u$ by a compactness argument. Step 3 is then devoted to {deriving a} minimality property and the proof of \eqref{eq_elliptic_regularity}. {The rest of the proof addresses \eqref{eq_energy_decay_goal}. Step 4 reduces the statement to a claim on measures, whose proof is further split into Steps  5--8.}

    \noindent\emph{Step 2. Convergence locally in $B \setminus \ell_0$.} We  show that the sequence $u_n$ converges locally in $B \setminus \ell_0$ after substracting suitable rigid motions. To this aim, we fix $0<t<1/10$ and   consider the sets 
    \begin{align}\label{DDDDDDD}
        D_{t}^\pm  =  \big\{ x \in {B(0,1)} \colon     \pm x \cdot e_2 > t \big\},  
    \end{align}
    see \eqref{eq :Bx0} and note that  $\nu(0,1) = e_2$.  Recall \EEE that $\beta_n \to 0$ by \eqref{ineq01-new}. Thus, {for $n$ {large} enough such that $\beta_n \leq t$}, we can apply the Poincar\'e-Korn inequality in  \eqref{eq: korns} {to} find 
    \begin{equation}\label{eq: the abschatz}
       { \|u_n-a^\pm_n\|_{H^1(D_{t}^\pm)}\leq \|u_n-a^\pm_n\|_{H^1(D_{\beta_n}^\pm)} \leq C \| e(u_n)\|_{L^2(D_{\beta_n}^\pm)}}
    \end{equation}
    for {the} rigid motions $a^\pm_n(x) = A^\pm_nx + b^\pm_n$ defined in \eqref{eq: the rigid motions}, where $C \geq 1$ is a universal  constant. 
    

    Thanks to   \eqref{energyBound} and a compactness   argument, for each $0 < t <1/10$, we find a subsequence such that $(u_n-a^\pm_n)_n$ converges {weakly in $H^1(D_{t}^\pm;\R^2)$ and strongly in $L^2(D_t^{\pm};\R^2)$}. Then, by a  diagonal \EEE argument for  $t \to 0$,  for a  subsequence (not relabeled), we obtain a function $u \in H^1(B \setminus \ell_0;\R^2)$ such that $(u_n-a^\pm_n)_n$ converges to $u$ weakly in $H^1(D_{t}^\pm;\R^2)$ and strongly in $L^2(D_{t}^\pm;\R^2)$, for any $0<{t}<1/10$. In particular, $u_n-a^\pm_n$ converges locally in $B \setminus \ell_0$.  By monotone  convergence and lower semicontinuity, we obtain 
    \begin{equation}\label{eq_lower_semicontinuity}
        \int_B \abs{e(u)}^2 \dd{x} = \lim_{t \to 0 }    \int_{D_{t}^+ \cup D_{t}^-} \abs{e(u)}^2 \dd{x} \leq   \liminf_{n \to \infty} \int_{B \setminus K_n}{} \abs{e(u_n)}^2 \dd{x} = 1,
    \end{equation}
    where the last identity follows from \eqref{energyBound}.

    \noindent\emph{Step 3. Minimality property of $u$ and elliptic regularity.}
    We prove that the limit $u$ satisfies an elliptic equation with a Neumann condition {on each side of} $\ell_0$ and thereafter obtain \eqref{eq_elliptic_regularity}.  To this end, we use a jump transfer argument which requires a minimal  separating extension $E_n$ related to $K_n$ on $B$. 
    {Using an adaptation of Lemma~\ref{lem:competitor} to the rescaled minimality condition \eqref{eq_Kn_minimality}, we know that   for each admissible pair $({v},G)$ in $B$ such  that
        \begin{equation}\label{eq_Fn_minimality0}
            G \setminus B(0, 9/10) = E_n \setminus B(0,9/10) \quad \text{ and } \quad  {v} = u_n \quad \text{in} \quad B \setminus \big(E_n \cup B(0,9/10)\big),
        \end{equation}
        we have
        \begin{multline}\label{eq_Fn_minimality}
            \int_{B \setminus E_n} \hspace{-0.4cm} \mathbb{C} e(u_n)\colon e(u_n)    \dd{x} + \omega_n^{-1}  \mathcal{H}^1(E_n)  \leq  \int_{B \setminus G}  \hspace{-0.4cm} \mathbb{C} e({v})\colon e({v})    \dd{x} + \omega_n^{-1} \mathcal{H}^1(G) + \omega_n^{-1} (\eta_n + \tilde{h}_n(1)) 
        \end{multline}
    for sufficiently large $n$ such that $\beta_n + \eta_n \leq 1/20$.}

    Now, we consider a test function $\varphi \in H^1(B \setminus \ell_0;\R^2)$ such that $\varphi = 0$ on $B \setminus B(0,9/10)$.
    We denote by $\Omega_\pm^n$ the connected components of $B \setminus E_n$ that contain the sets $D^\pm_{1/2}$, see \eqref{DDDDDDD},   and we define
    \begin{equation*}
        \varphi_n(x) :=
        \begin{cases}
            \varphi(x_1,|x_2|)  & \ \text{in} \ {\Omega_+^n}\\
            \varphi(x_1,-|x_2|) & \ \text{in} \ {\Omega_-^n}\\
            0                      & \ \text{otherwise},
        \end{cases}
    \end{equation*}
    on $B \setminus E_n$,     where $x=(x_1,x_2)$. The definition by reflection ensures that $\varphi_n \in H^1(B\setminus E_n;\R^2)$. One can check that $\varphi_n = 0$ on $B \setminus B(0,9/10)$ and that $\varphi_n = \varphi$ on $B \cap \lbrace |x_2| >  \beta_n \rbrace$.  As   $\beta_n \to 0$ by \eqref{ineq01-new}, we get  $\varphi_n \to \varphi$ strongly in $L^2(B;\R^2)$ and $e(\varphi_n) \to e(\varphi)$ strongly in $L^2(B;\R^{2\times 2})$. 
    We apply the minimality property (\ref{eq_Fn_minimality}) to compare $(u_n,E_n)$ and $(u_n + \varphi_n, E_n)$ and we obtain 
    \begin{equation*}
        \int_{B \setminus E_n} \mathbb{C} e(u_n) \colon   e(u_n) \, {\dd{x}}\leq \int_{B \setminus E_n} \mathbb{C} e(u_n+\varphi_n) \colon  e(u_n+\varphi_n)  \, {\dd{x}}   + \omega_n^{-1} (\eta_n + \tilde{h}_n(1)). 
    \end{equation*}
    This implies, by expanding the squares, that
    \begin{equation*}
        0 \leq 2 \int_{B \setminus E_n} \mathbb{C} e(u_n):e(\varphi_n)\,{\dd{x}} + C\int_{B \setminus E_n} \abs{e(\varphi_n)}^2\,{\dd{x}} + {\omega_n^{-1}} (\eta_n + \tilde{h}_n(1)). 
    \end{equation*}
    Using that ${\omega_n^{-1}} (\eta_n + \tilde{h}_n(1))  \to 0$, see \eqref{ineq01-new},   and the fact  that $\varphi_n \to \varphi$, $e(\varphi_n) \to e(\varphi)$ strongly in $L^2$, as well as  $e(u_n) \rightharpoonup e(u)$, we pass to the limit as $n \to +\infty$ and deduce
    \begin{equation*}
        0 \leq 2 \int_{B \setminus \ell_0} \mathbb{C} e(u):e(\varphi)\, {\dd{x}} + C\int_{B \setminus \ell_0} \abs{e(\varphi)}^2\, {\dd{x}}.
    \end{equation*}
    In this last equation, we can replace $\varphi$ by $\varepsilon \varphi$ for an arbitrary $\varepsilon \in \R$ so we deduce in fact
    \begin{equation*}
        \int_{B \setminus \ell_0} \mathbb{C} e(u):e(\varphi)\, {\dd{x}} = 0.
    \end{equation*}
    We conclude that $u$ is {locally (away from $\partial B(0,9/10)$)} a weak solution of $\mathrm{div}({\mathbb{C}}e(u)) = 0$ in $B(0,9/10) \setminus \ell_0$ {with a} homogeneous Neumann condition on {each side of} $\ell_0$.
    We deduce by elliptic regularity {(see Remark \ref{rmk_alpha})}    that, for given $0 <b < 1/2$,   it holds that
    \begin{equation*}
        \int_{B(0,b) \setminus {\ell_0}} \abs{e(u)}^2 \, \dd x \leq C b^2 \int_{B \setminus \ell_0} \abs{e(u)}^2  \, \dd x \le C b^2
    \end{equation*}
    for a constant $C \geq 1$ {which depends on $\mathbb{C}$}, where the last step follows from  (\ref{eq_lower_semicontinuity}). This shows \eqref{eq_elliptic_regularity}.

    \noindent\emph{Step 4: Proof of \eqref{eq_energy_decay_goal}.} In order to show \eqref{eq_energy_decay_goal},  we take any subsequence (not relabeled) such that the left-hand side of (\ref{eq_energy_decay_goal}) converges, and we consider the sequence of measures
    \begin{equation*}
        \mu_n:= |e(u_n)|^2 \mathcal{L}^2.
    \end{equation*}
    By \eqref{energyBound} we find $\mu_n(B) \leq 1$, and thus we can extract a further subsequence such that $\mu_n \overset{\ast}{\rightharpoonup} \mu$ in $B$, for some measure $\mu$.   The rest of the proof is devoted to showing
    \begin{equation}\label{eq: to proove}
        \mu = \abs{e(u)}^2 \mathcal{L}^2 \quad \text{in $B(0,1/2)$},\tag{Claim 2}
    \end{equation}
    which {by properties of measures directly} implies (\ref{eq_energy_decay_goal}). The proof of \eqref{eq: to proove} is divided into {two steps}. We first prove that  $e(u_n)   \to e(u)$ in $L^2_{\mathrm{loc}}(B \setminus \ell_0; \R^{2 \times 2})$ {to find} 
    \begin{equation}\label{eq: the first thing}
        \mu|_{B \setminus \ell_0}   =|e(u)|^2 \mathcal{L}^2|_B.  \tag{Claim 2.a}
    \end{equation}
    Afterwards, we conclude \eqref{eq: to proove} by {showing}
    \begin{equation}\label{eq: the second thing}
        \mu\left(\ell_0\cap B\left(0,1/2\right)\right) = 0. \tag{Claim 2.b}
    \end{equation}
    {The rest of the proof is dedicated to proving \eqref{eq: the first thing} and \eqref{eq: the second thing}.}

    \emph{Step 5: Proof of \eqref{eq: the first thing}.}     We claim that    $e(u_n) \to e(u)$ in $L^2_{\mathrm{loc}}(B \setminus \ell_0; \R^{2\times 2})$ which immediately gives  \eqref{eq: the first thing}  To this end, it suffices to check that for   any ball $\overline{B(x,r)} \subset B \setminus \ell_0$ we get
    \begin{equation}\label{eq: reverse}
        \limsup_{n \to +\infty} \int_{B(x,r)} \mathbb{C}e(u_n)\colon e(u_n) \dd{x} \leq \int_{B(x,r)} \mathbb{C}e(u)\colon e(u) \dd{x}.
    \end{equation}
    Then, the strong convergence  follows  from the weak convergence {of $e(u_n)$} and {the fact that weak convergence plus the convergence of norms (a consequence of \eqref{eq: reverse}) implies strong convergence in a Hilbert space.}

    We consider a small $\delta > 0$ such that $\overline{B(x,r+\delta)} \subset B \setminus \ell_0$ and a cut-off function $\varphi \in C^\infty_c(B(x,r+\delta))$ such that $0 \leq \varphi \leq 1$ and $\varphi = 1$ on $B(x,r)$. For $n$ large enough, in view of \eqref{ineq01-new},  we have    $\overline{B(x,r+\delta)} \subset B \setminus K_n$.
    We compare $(u_n,K_n)$ with the competitor $(v_n,K_n)$, where
    \begin{equation*}
        v_n := \varphi (u + a_n^\pm) + (1 - \varphi) u_n,
    \end{equation*}
    where $a^\pm_n$ are the rigid motions found in \eqref{eq: the abschatz}.  {Writing for shorthand $|\xi|_\mathbb{C} = \sqrt{\mathbb{C} \xi \colon  \xi}$ for $\xi \in \R^{2 \times 2}$,} we note that 
    $$
    |e(v_n)|_\mathbb{C} \le      \varphi \abs{e(u)}_\mathbb{C} + (1 - \varphi) \abs{e(u_n)}_\mathbb{C} + C\abs{\nabla \varphi} \abs{u_n - a_n^\pm -u}
    $$
    for a constant $C$ depending on $\mathbb{C}$.
    Thus, by the elementary identity $(a+b)^2\leq (1+\varepsilon)a^2+\left(1+\varepsilon^{-1}\right)b^2$ for all $a,b \ge 0$ and  $\varepsilon > 0$, we get  
    \begin{align*}
        \int_{B} |e(v_n)|^2_\mathbb{C} \dd{x} &\leq (1+\varepsilon) \int_{B} (\varphi \abs{e(u)}_\mathbb{C} + (1-\varphi) \abs{e(u_n)}_\mathbb{C})^2 \dd{x} 
        + C\left(1+\varepsilon^{-1}\right) \int_B |\nabla \varphi|^2 \abs{u_n - a_n^\pm - u}^2 \dd{x}  \\ 
                                              &\le    (1+\varepsilon) \int_{B} \big(\varphi |e(u)|_\mathbb{C}^2 +(1-\varphi)|e(u_n)|_\mathbb{C}^2 \big) \dd{x}
                                              + C\left(1+\varepsilon^{-1}\right) \int_B |\nabla \varphi|^2 \abs{u_n - a_n^\pm - u}^2 \dd{x},
    \end{align*} 
    where the second step follows from   the convexity of $t \mapsto t^2$.   According to the almost-minimality property (\ref{eq_Kn_minimality}) of $(u_n,K_n)$ and {since the crack set of the competitor is kept unchanged}, we have
    \begin{equation*}
        \int_B \abs{e(u_n)}_\mathbb{C}^2 \, \dd x \leq \int_B \abs{e(v_n)}_\mathbb{C}^2 \, \dd x + \omega_n^{-1} h_n(1)
    \end{equation*}
    and using the two previous estimates, this gives
    \begin{multline*} 
        \int_B \varphi \abs{e(u_n)}_\mathbb{C}^2 \, \dd x \leq  (1 + \varepsilon) \int_{B} \varphi \abs{e(u)}_\mathbb{C}^2 \dd{x} + \varepsilon \int_{B} (1 - \varphi) \abs{e(u_n)}_\mathbb{C}^2 \dd{x} \\ + C\left(1+\varepsilon^{-1}\right) \int_B |\nabla \varphi|^2 \abs{u_n- a_n^\pm - u}^2 \dd{x} + \omega_n^{-1} h_n(1).
    \end{multline*} 
    Then, by definition of $\varphi$, we find
    \begin{align*}
        \int_{B(x,r)} \abs{e(u_n)}_\mathbb{C}^2  \, \dd x & \leq (1+\varepsilon) \int_{B(x,r+\delta)} \abs{e(u)}_\mathbb{C}^2 \dd{x} + \varepsilon \int_B \abs{e(u_n)}_\mathbb{C}^2 \dd{x} \\ & \ \ \ \ + C\left(1+\varepsilon^{-1}\right) \int_{B(x,r+\delta)}   |\nabla \varphi|^2 \abs{u_n- a_n^\pm - u}^2 \dd{x} + \omega_n^{-1} h_n(1).
    \end{align*} \EEE
    We recall that $\int_B \abs{e(u_n)}^2 \dd{x} = 1$, see \eqref{energyBound}, and  that $u_n - a_n^\pm \to u$ strongly in $L^2(B(x,r+\delta);\R^2)$ by  Step 2. Thus,  passing to the limit $n \to +\infty$ and using \eqref{ineq01-new} gives
    \begin{equation*}
        \limsup_{n \to +\infty} \int_{B(x,r)} {\mathbb{C}e(u_n)\colon e(u_n)} \dd{x} \leq (1 + \varepsilon)   \int_{B(x,r+\delta)} {\mathbb{C}e(u)\colon e(u)} \dd{x} + {C}\varepsilon.
    \end{equation*}
    Since $\varepsilon > 0$ and $\delta > 0$ are arbitrary, we obtain \eqref{eq: reverse}. {As remarked previously, this shows that} ${e(u_n)}  \to e(u)$ in $L^2_{\mathrm{loc}}(B \setminus \ell_0; \R^{2 \times 2})$, and therefore {concludes} \eqref{eq: the first thing}.

    \emph{Step 6: Proof of \eqref{eq: the second thing}.}   Let  $\varepsilon > 0$ be small and let {$\Xi_{\rm bad}$} be the finite set of radii $\rho \in [1/2, 3/4]$ such that $\mu(\partial B(0,\rho)) \geq \varepsilon$. Choose $s_\eps>0$ sufficiently small such that the set     
    $${\Xi_{\rm good}} := \big\{t \in \left[\tfrac{1}{2},\tfrac{3}{4}\right] \colon \,  \mathrm{dist}(t,{\Xi_{\rm bad}}) \ge  s_\eps\big\}$$      
    satisfies $\mathcal{L}^1({\Xi_{\rm good}}) \ge 3/16$. {In particular, the complement of $\Xi_{\rm good}$ in $[1/2,3/4]$ has measure  at most  $ 1/16$.} Then, by   Lemma \ref{lem_goodradius} (applicable for $n$ large enough since $\beta_n  + \eta_n \to 0$), we {may} select a radius $\rho_n \in [1/2, 3/4]$ with $\rho_n \in {\Xi_{\rm good}}$  such that
    \begin{equation}\label{eq: rhorhorho}
        \sum_{i \in I(\rho_n)} r^\sigma_i \leq C (\beta_n m_n  + \eta_n),
    \end{equation}
    for some constant $C \geq 1$ depending on $\mathbb{C}$.    We can extract a subsequence so that $(\rho_n)_n$ converges to some $\rho_\infty \in [1/2, 3/4]$. As $\rho_n \in {\Xi_{\rm good}}$ for all $n \in \N$, we  get $\mathrm{dist}(\rho_\infty,{\Xi_{\rm bad}}) \ge  s_\eps$, and in particular
    \begin{equation}\label{eq_partial_rho}
        \mu(\partial B(0,\rho_\infty)) \leq \varepsilon.
    \end{equation}
    To obtain \eqref{eq: the second thing}, we will prove
    \begin{equation}\label{the main inequality}
        \mu\left(\ell_0 \cap B(0,\rho_\infty)\right) \leq C \mu\left(\ell_0 \cap \partial B(0,\rho_\infty)\right).\tag{Claim 3}
    \end{equation}
    This along with \eqref{eq_partial_rho}, {$\rho_\infty \geq 1/2$,} and the fact that $\eps>0$ is arbitrary {gives \eqref{eq: the second thing}.}

    To establish \eqref{the main inequality},  for $a \in (0,1/4)$ small which will eventually be sent to $0$,  we define the  sets \EEE
    \begin{align}\label{UaVa}
        U_a &:= \big\{y \in \R^2 \colon \mathrm{dist}\big(y, \ell_0 \cap \overline{B(0,\rho_\infty - 2 a)}\big)  < \tfrac{1}{2}10^{-5}a\big\}, \notag \\
        V_a &:= \big\{y \in \overline{B(0,\rho_\infty +a)} \colon \, \mathrm{dist}(y,\ell_0) \leq a\big\}
    \end{align}
    and,   for $n$ {large} enough, \EEE we will show the estimate
    \begin{equation}\label{eq: imp one}
        \int_{U_a} \abs{e(u_n)}^2 \, \dd x \leq C \int_{V_a \setminus U_a} \abs{e(u_n)}^2 \, \dd x +  C \omega_n^{-1} \big( \beta_n  m_n + \eta_n + \tilde{h}_n(1)\big).  \tag{Claim 4}    
    \end{equation}   
    In view of \eqref{eq: imp one}, sending $n \to \infty$ and using \eqref{ineq01-new}  as well as $\mu_n \overset{\ast}{\rightharpoonup} \mu$ in $B$ we conclude 
    \begin{equation*}
        \mu\big(\ell_0 \cap B(0,\rho_\infty - 2 a)\big) \leq C \mu\big(V_a \setminus U_a\big).
    \end{equation*}
    One can check the as $a \to 0$ the set on the right-hand side shrinks to a set with two points only, namely $\ell_0 \cap \partial B(0,\rho_\infty)$. Therefore, sending $a \to 0$ we indeed obtain $\mu(\ell_0 \cap B(0,\rho_\infty)) \leq C \mu(\ell_0 \cap \partial B(0,\rho_\infty))$, i.e., \eqref{the main inequality}. {Thus, we must prove {\eqref{eq: imp one}} to conclude the theorem.}

    \emph{Step 7: {Proof of \eqref{eq: imp one}}.}
    The proof of the claim relies on the construction of an appropriate competitor.  We again denote by $E_n$ a minimal separating extensions of $K_n$  in $B$, and we introduce
    \begin{equation}\label{eq: Gn}
        G_n :=  E_n \cup \bigcup_{i \in I(\rho_n)} \partial (10 B^{\sigma_n}_i),
    \end{equation}
    where $\sigma_n$ denotes the stopping time function with respect to $E_n$,  the balls $(B^{\sigma_n}_i)_i$ are defined correspondingly in \eqref{def_Bi}, and $\rho_n$ is given in \eqref{eq: rhorhorho}.        {By Corollary} \ref{cor: wall set}  we find \EEE  that $G_n$   is relatively closed in {$B$}.    It is   also  clear that
    \begin{equation}\label{EequalG}
        G_n \setminus B\left(0, 9/10\right) = E_n \setminus B\left(0,9/10\right)
    \end{equation}
    since   for all $i \in I(\rho_n)$ we have $\overline{10 B^{\sigma_n}_i} \subset B(x_0, 9 r_0/10)$ by \eqref{eq: in the next section}. (Note that \eqref{eq_varepsilon00_flat} holds  for $n$ large enough as $\beta_n + \eta_n \to 0$.)    Our goal is to {use} Proposition \ref{th: extension} to construct a sequence $v_n$ with $v_n = u_n$ on $B \setminus (E_n \cup B(0,9/10))$ such that  $(v_n,G_n)$ is a competitor of $(u_n,E_n)$ in the sense of \eqref{eq_Fn_minimality0}--\eqref{eq_Fn_minimality}. Then, we can {compare the energies using (\ref{eq: rhorhorho}) to find}
    \begin{equation}\label{eq_energy_comparison}
        \int_{B \setminus K_n } \mathbb{C} e(u_n) \colon  e(u_n)  \, {\dd{x}}\leq \int_{B \setminus G_n} \mathbb{C} e(v_n) \colon  e(v_n) \, {\dd{x}}   + C \omega_n^{-1} \big( \beta_n m_n   + \eta_n+ \tilde{h}_n(1)  \big).
    \end{equation}
    {To recover \eqref{eq: imp one} from the above equation, we will construct $v_n$ such  that} on the right-hand side of \eqref{eq_energy_comparison} the elastic energy of $v_n$ is controlled by the elastic energy of $u_n$ away from $\ell_0$. For this, the geometric function used in the extension result needs to be  well chosen.

    {Let}  $0 < a < 1/4$ be  small as {in \eqref{UaVa} above}. {We define the function $f_n \colon [0,+\infty) \to [0,a]$ by}
    \begin{equation*}
        f_n(t) =
        \begin{cases}
            a    & \ \text{for} \ t \leq \rho_n - a\\
            \rho_n - t   & \ \text{for} \ \rho_n - a \leq t \leq \rho_n\\
            0       & \ \text{for} \ t \geq \rho_n.
        \end{cases}
    \end{equation*}
    Then, we define {the geometric function}  $\delta_n \colon E_n \cap \overline{B(0,3/4)} \to [0,1/4]$ by 
    \begin{equation*}
        \delta_n(x) := \max \big\{ f_n(\abs{x}), \delta_{E_n}(x) \big\} \quad \text{for $x \in E_n$},
    \end{equation*}
    where $\delta_{E_n}$ is defined in (\ref{eq_delta}) with respect to $E_n$. The function $\delta_n$ is a geometric function with parameter $(3/4,16\tau)$, where we use $\beta_n+\eta_n \to 0$ (see \eqref{ineq01-new}) to apply  Lemma \ref{lem_delta}(3), and then we obtain \eqref{eq: geotau} for $E_n$ in place of $E$.

    According to Lemma \ref{lem_delta}(2) and \eqref{eq_ri_estimate}, we have $\delta_{E_n}(x) \leq \max_i (10 r^{\sigma_n}_i) = \max_i (10M \sigma_n(x^{\sigma_n}_i))  \leq C  (\beta_n + \eta_n)$ for each $x \in E_n$ for a universal constant $C = C(M,\tau) \geq 1$. As $\beta_n + \eta_n {\to 0}$ by \eqref{ineq01-new}, we {have} $\delta_{E_n} \le a$ for $n$ large enough. This together with the definition of $f_n$ shows that 
    \begin{align}\label{eq: hilft}
        \text{ $ \delta_n(x) \le a$ for all $x \in E_n$, with equality if  $x \in E_n \cap \overline{B(0,\rho_n - a)}$. }
    \end{align}
     We denote \EEE by $\Omega^n_{+}$ and  $\Omega^n_{-}$ the two connected components of $B \setminus E_n$ containing the respective connected components  of $\lbrace x \in B \colon \mathrm{dist}(x, \ell_n)  > 1/2\rbrace$.    We apply Proposition \ref{th: extension} with respect to $(u_n,E_n)$,  the radius $\rho_n$ defined {above in (\ref{eq: rhorhorho})}, and  the geometric function $\delta_n$ with parameters $(\rho_n,16\tau)$. (Note   that Proposition \ref{th: extension} is applicable as $\beta_n \le 10^{-8}/8$,  $16 \tau \le10^{-8}$, and {$8\beta_n \leq 16  \tau$} for $n$ large enough.)

    {In the extended {domains} $\Omega^{n,\rm ext}_{\pm} := \Omega^n_{\pm} \cup W^n$}, we obtain functions $v^n_{\pm} \in LD_{\mathrm{loc}}(\Omega^{n,{\rm ext}}_{\pm})$, and relatively closed subsets $S^n_{\pm} \subset \Omega^{n,{\rm ext}}_{\pm}$  such that
    \begin{equation}\label{eq: for LLL}
        W^n \subset S^n_\pm \subset W^n_{10}, \quad \quad \quad S^n_\pm \setminus B(0,\rho_n) \subset W^{n,{\rm bdy}}_{10}, \quad \quad \quad v^n_\pm = u_n \ \text{in} \ \Omega^{n,{\rm ext}}_\pm \setminus S^n_\pm,
    \end{equation}
    and
    \begin{equation*}
        \int_{(\Omega^{n, {\rm ext}}_\pm \cap B(0,\rho_n)) \setminus W_{10}^{n,{\rm bdy}}} \abs{e(v^n_\pm)}^2 \dd{x} \leq C \int_{B(0,\rho_n) \cap \Omega^n_\pm} \abs{e(u_n)}^2 \dd{x},
    \end{equation*}
    where the sets $W^n$, $W^{n}_{10}$, and  $W_{10}^{n,{\rm bdy}}$ are defined in \eqref{eq: Wdef}--\eqref{eq: 10bry} with respect to $E_n$,  $\rho_n$, and $\delta_n$, and $\Omega^{n, {\rm ext}}_\pm  :=  \Omega^n_\pm  \cup W^n$.
    We will also use the refined estimate from Remark \ref{rem: refinement},  namely \EEE  
    \begin{align}\label{eq tubi}
        \int_{ S^n_\pm  \setminus W_{10}^{n,{\rm bdy}}   }  |e(v^n_\pm)|^2 \, {\rm d}x \le C \int_{ E^n_{\rm tube}  }  |e(u_n)  |^2 \, {\rm d}x,
    \end{align}
    where  $E^n_{\rm tube}$ is defined in \eqref{eq: tube} with respect to $E_n$,  $\rho_n$, and $\delta_n$.

    We note that 
    \begin{equation}\label{again to check}
        W_{10}^{n,{\rm bdy}} \subset \bigcup_{i \in I(\rho_n)} 10 B^{\sigma_n}_i.
    \end{equation}
    To see this, we first observe that  for all $0 \leq t \leq \rho_n$, we have $t + f_n(t) \leq \rho_n$. Thus, for $x \in E_n \cap \overline{B(0,\rho_n)}$, {it holds $\abs{x} + f_n(\abs{x}) \leq \rho_n$ and thus} $B(x,50 \cdot 10^{-5} f_n(\abs{x})) \subset B(0,\rho_n)$ by  the \EEE {triangle} inequality. Now, whenever $x \in E_n \cap \overline{B(0,\rho_n)}$ satisfies   $B(x,50 r_x) \cap \partial B(0,\rho_n) \neq \emptyset$, where $r_x = 10^{-5}\delta_n(x)$, the previous observation shows that we necessarily  have $\delta_n(x) = \delta_{E_n}(x)$. This implies that  the set  $W_{10}^{n,{\rm bdy}}$ defined in  \eqref{eq: 10bry} with respect to $\delta_n$ remains the same if it is defined with respect to the smaller geometric function $\delta_{E_n}$. Then, the inclusion follows from Corollary \ref{cor: wall set}.

    According to  \eqref{extiii} in Proposition \ref{th: extension}, we have     
    \begin{equation}\label{eq_W_inclusion}
        \overline{B(0,\rho_n)} \setminus \left(  E_n \EEE \cup \Omega^n_+ \cup \Omega^n_-\right) \subset W^n.
    \end{equation}
    We also note that $G_n \cup\bigcup_{i \in I(\rho_n)} {10 {{B^{\sigma_n}_i}}}$ is relatively closed by the definition of $G_n$ in \eqref{eq: Gn} and the fact that ${G_n}$ is relatively closed.  We now define $v_n \in LD(B\setminus G_n)$ by 
    \begin{align}\label{piewice def}
        v_n =
        \begin{cases}
            v^n_+ &\text{in} \ \Omega^n_+ \setminus (G_n \cup\bigcup_{i \in I(\rho_n)}  {10 {{B^{\sigma_n}_i}}})\\
            v^n_- &\text{in} \ \Omega^n_- \setminus (G_n \cup\bigcup_{i \in I(\rho_n)} {10 {{B^{\sigma_n}_i}}})\\
            u_n   &\text{in} \ B \setminus \left(\Omega^n_+ \cup \Omega^n_- \cup \overline{B(0,\rho_n)} \cup G_n \cup  \bigcup_{i \in I(\rho_n)} 10  {{B^{\sigma_n}_i}}\right)\\
            0   &\text{in } \ B(0,\rho_n) \setminus \left({\Omega_+^n \cup \Omega_-^n} \cup G_n \cup \bigcup_{i \in I(\rho_n)} 10 {{B^{\sigma_n}_i}}\right)\\            
            0   &\text{in} \ \bigcup_{i \in I(\rho_n)} 10 B^{\sigma_n}_i.
        \end{cases}
    \end{align}
    Note that the function is well-defined  since  the piecewise domains in the construction are disjoint open sets which cover $B \setminus G_n$. To see this, we note that the set $\partial B(0,\rho_n) \setminus  ({\Omega_+^n \cup \Omega_-^n} \cup G_n \cup \bigcup_{i \in I(\rho_n)} 10 {{B^{\sigma_n}_i}})$ which is not covered in the definition above is actually empty since   $  \partial B(0,\rho_n) \setminus \left(\Omega^n_+ \cup \Omega^n_- \cup G_n\right) \subset W_{10}^{n,{\rm bdy}} \subset \bigcup_{i \in I(\rho_n)} 10 B^{\sigma_n}_i$   by   \eqref{eq: Gn}, \EEE \eqref{eq: for LLL},  \eqref{again to check}, and (\ref{eq_W_inclusion}). For the same reason,   the connected components of 
    \begin{align*} 
        B \setminus \left(\Omega^n_+ \cup \Omega^n_- \cup G_n \cup \bigcup_{i \in I(\rho_n)} 10 {B^{\sigma_n}_i}\right)
    \end{align*} 
    are either contained in $B(0,\rho_n)$ or $B(0,1) \setminus \overline{B(0,\rho_n)}$, respectively, and we set $v_n = 0$  or $v_n = u_n$ on such components, {respectively}. We refer to Figure~\ref{fig:simpleWallSet} for an illustration. This guarantees that we do not induce any jump between $u_n$ and $0$, potentially coming from the fourth line in the definition of \eqref{piewice def}, and thus   $v_n \in LD(B\setminus G_n)$. 

    For $a$ small enough, recalling \eqref{eq: in the next section},  we can also check that $W_{10}^n$ and $\bigcup_{i \in I(\rho_n)} 10 {B^{\sigma_n}_i}$ are contained in $B(0,9/10)$. Then, \eqref{eq: for LLL} and \eqref{piewice def} imply that $v_n = u_n$ on $B \setminus (E_n \cup B(0,9/10))$. This along with \eqref{EequalG} shows that  the pair $(v_n,G_n)$ is a competitor of $(u_n,E_n)$ in $B(0,9/10)$ in the sense of \eqref{eq_Fn_minimality0}.   Recalling the definition of $G_n$ in \eqref{eq: Gn} and the estimate  \eqref{eq: rhorhorho}, we {finally have}  \eqref{eq_energy_comparison} {for our constructed $v_n$}.

    Noting that $v^n_\pm = u_n$ in $\Omega^n_\pm \setminus S^n_\pm$, an inspection of the piecewise definition of $v_n$ in \eqref{piewice def} shows that for all $x \in B \setminus (G_n \cup S^n_+ \cup S_n^-)$, we have either $v_n = 0$ or $v_n = u_n$ in a neighborhood of $x$. With this,  (\ref{eq_energy_comparison}) simplifies to
    \begin{equation*}
        \int_{S^n_+ \cup S^n_-} \mathbb{C} e(u_n) \colon  e(u_n) \dd{x}  \leq   \int_{S^n_+ \cup S^n_-}  \mathbb{C} e(v_n) \colon  e(v_n) \dd{x}   + C \omega_n^{-1} \big( \beta_n m_n + \eta_n + \tilde{h}_n(1)\big).
    \end{equation*}
    
 We further estimate the right-hand side.      By  \eqref{eq: for LLL} and \eqref{again to check}, we observe that $$(S^n_+ \cup S^n_-) \setminus B(0,\rho_n) \subset W^{n,{\rm bdy}}_{10} \subset \bigcup_{i \in I(\rho_n)} 10 B^{\sigma_n}_i$$
    and, as $W \subset S^n_{\pm} \subset \Omega^n_{\pm} \cup W$, we also see that
     $$(S^n_+ \cup S^n_-) \cap \Omega^n_+ \subset \Omega^n_+ \cap S^n_+, \quad \quad { (S^n_+ \cup S^n_-) \cap \Omega^n_- \subset \Omega^n_- \cap S^n_-.}$$
    Hence,  each   point of $S^n_+ \cup S^n_+$ is either contained in $\bigcup_{i \in I(\rho_n)} 10 B^{\sigma_n}_i$ or $B(0,\rho_n) \setminus (\Omega_n^+ \cup \Omega_n^- \cup \bigcup_{i \in I(\rho_n)} 10 B^{\sigma_n}_i)$, where $e(v_n) = 0$, or in $S^n_{\pm} \cap \Omega^n_{\pm} \setminus \bigcup_{i \in I(\rho_n)} 10 B^{\sigma_n}_i$, where $e(v_n) =e (v^n_{\pm})$. We conclude that
    \EEE
    \begin{equation*}
        { \int_{S^n_+ \cup S^n_-} \mathbb{C} e(v_n) \colon  e(v_n) \dd{x}  \leq \sum_{h = \pm}  \int_{S^n_h \setminus W^{n,{\rm bdy}}_{10}}  \mathbb{C} e(v^n_h) \colon  e(v^n_h) \dd{x}.}
    \end{equation*}
    Tying together the above inequalities, by the fact that $W^n \subset S^n_{\pm}$ by \eqref{eq: for LLL}, the fine estimate \eqref{eq tubi}, and the coercivity of $\mathbb{C}$, we conclude that
    \begin{equation}\label{eqn:almost_done}
        \int_{W^n} \abs{e(u_n)}^2 \, \dd x \leq C \int_{E_{\rm tube}^n} \abs{e(u_n)}^2 \, \dd x +  C \omega_n^{-1} \big( \beta_n m_n  + \eta_n + \tilde{h}_n(1)\big).
    \end{equation}

    To {obtain \eqref{eq: imp one},} it remains to replace the domains of integration $ W^n    $ and $ E^n_{\rm tube} $ {in \eqref{eqn:almost_done}}. Recalling the definition of $U_a$ and $V_a$ in \eqref{UaVa}, we claim that, for $n$ sufficiently large, we have
    \begin{equation}\label{eq_rmk_vh1}
        {\rm (i)}  \ \     U_a \subset W^n, \quad \quad \quad  {\rm (ii)}   \ \  E^n_{\rm tube} \subset V_a \setminus U_a. \tag{Claim 5}
    \end{equation}
    Assuming \eqref{eq_rmk_vh1} holds, \eqref{eqn:almost_done} directly implies \eqref{eq: imp one}. Now, we turn to proving \eqref{eq_rmk_vh1}.

    \emph{Step 8: Proof of  \eqref{eq_rmk_vh1}}. We start with (\ref{eq_rmk_vh1})(i). {For} $y \in U_a$, there exists $x_0 \in \ell_0 \cap \overline{B(0,{\rho_\infty} - 2 a)}$ such that $\abs{x_0 - y} < \frac{1}{2}10^{-5}a$.
    Since $\beta_n + \eta_n \to 0$ by \eqref{ineq01-new}, Lemma \ref{lem_bilateral_flatness-old} {shows that the} bilateral flatness {from} \eqref{eq: bilat flat} {also goes to zero}. {For $n$ large enough, we have $\abs{\rho_n - \rho_{\infty}} < \frac{1}{2} 10^{-5} a$ and by the previous observation on the bilateral flatness, there exists $x \in E_n$ such that $\abs{x_0 - x} <\frac{1}{2} 10^{-5} a $.}
    Therefore, $x \in E_n \cap \overline{B(0,\rho_n - a)}$ and thus $\delta_n(x) = a$ by   \eqref{eq: hilft}.      Then,  by  the \EEE triangle inequality it follows that $y \in B(x, 10^{-5}\delta_n(x))$, and then $y \in W^n$ by \eqref{eq: Wdef}, {as desired}.

    We now show  (\ref{eq_rmk_vh1})(ii).
    For $y \in    E^n_{\rm tube}$, by   {simply rewriting the} definition in \eqref{eq: tube},   {we have that} there exists 
    \begin{align}\label{xfrom}
        \text{$x \in E_n \cap \overline{B(0,\rho_n)}$ such that $y \in B(x,50 r_x )$  and $\mathrm{dist}(y,E_n) \geq r_x$, with $r_x = 10^{-5}\delta_n(x)$.}
    \end{align} 
    As $\abs{y - x} \leq 50 \cdot 10^{-5} \delta_n(x) \leq 50 \cdot 10^{-5} a$, see \eqref{eq: hilft},  and $\mathrm{dist}(x, \ell_n)  \leq \beta_n$, we have $\mathrm{dist}(y,\ell_0)  \leq a$ for $n$ large enough. 
    {Moreover, we also have $\abs{\rho_{\infty} - \rho_n} \leq 50 \cdot 10^{-5} a$ for $n$ big enough, whence $y \in \overline{B(0,\rho_{\infty}+a)}$.}
    Consequently, $y\in V_a$, see \eqref{UaVa}.

    We {prove} that  $y \notin U_a$ by contradiction. {Assume} by contradiction that there exists $x_0 \in \ell_0 \cap \overline{B(0,\rho_\infty - 2 a)}$ such that $\abs{y - x_0} < \frac{1}{2}10^{-5} a$. As $\mathrm{dist}(x_0,E_n)$ is controlled by the bilateral flatness {of $E_n$ in $B(0,1)$} which converges to $0$ by  $\beta_n + \eta_n\to 0$ and  Lemma \ref{lem_bilateral_flatness-old}, we have
    \begin{equation}\label{eq:  contraddiii}
        \mathrm{dist}(y,E_n) < 10^{-5}a
    \end{equation}
    for $n$ sufficiently large.
    Recalling that $\abs{y - x} \leq 50 \cdot 10^{-5} a$ {for the} $x$ {coming} from \eqref{xfrom}, we also have $x \in \overline{B(0,\rho_n - a)}$ for {$n$ sufficiently large} by the {triangle} inequality. Then, \eqref{eq: hilft} implies   $\delta_n(x) = a$. {Applying \eqref{xfrom} for} $y \in E_{\rm tube}^n$ {then gives} $\mathrm{dist}(y,E_n) \geq {r_x =} 10^{-5}a $, which is a contradiction to \eqref{eq:  contraddiii}. {This concludes the proof of (\ref{eq_rmk_vh1}) and thereby the entire proof.}
\end{proof}


\vspace{0.5em}

\begin{center}
{ \sc Acknowledgements}
\end{center}

\vspace{0.5em}

This work was supported by the DFG project FR 4083/3-1 and by the Deutsche Forschungsgemeinschaft (DFG, German Research Foundation) under Germany's Excellence Strategies EXC 2044 -390685587, Mathematics M\"unster: Dynamics--Geometry--Structure and EXC 2047/1 - 390685813. K.S. also thanks Sergio Conti for insightful discussions on related problems.

\appendix

\section{Minimal separating extensions} \label{sec:MinimalExtension}

In this section, we justify that the minimization over separating extensions to obtain (\ref{eq: eta definition2}) in the definition of $\eta$ is well-posed. Secondly, we show that this minimal separating extension is Ahlfors-regular as stated in Lemma \ref{lem_AFF}.

We  consider a relatively closed subset $K \subset B(x_0,r_0)$ such that $x_0 \in K$, $\beta_K(x_0,r_0) \leq 1/2$ and which is Ahlfors-regular {in $B(x_0,r_0)$, i.e., \eqref{eqn:AhlforsReg} holds for all balls $B(x,r) \subset B(x_0,r_0)$ with $x \in K$.}
As in \eqref{eq: main assu}, we define a \emph{separating extension} of $K$ in $B(x_0,r_0)$ as a relatively closed and coral set $E \subset B(x_0,r_0)$ such that $K \cap B(x_0,r_0)  \subset E$, $\beta_E(x_0,r_0) = \beta_K(x_0,r_0)$, and $E$ separates $B(x_0,r_0)$ in the sense of Definition \ref{defi_separation}.
Recall that $E$ is coral if for all $x \in B(x_0,r_0)$ and $r > 0$  it holds \EEE $\HH^1(E \cap B(x,r)) > 0$. We note the assumption that a separating extension is coral is for technical convenience, as a brief measure-theoretic argument can be used to show that a separating extension, which is not coral, has a coral representative.

\begin{proposition}
    Let $K$ be as stated at the beginning of the section. Define the collection of all admissible separating extensions of $K$ in $B(x_0,r_0)$ by $\mathcal{C}_{\rm sep}$. Then, there is $E \in \mathcal{C}_{\rm sep}$ such that $$E = \underset{{F \in \mathcal{C}_{\rm sep}}}{\operatorname{argmin}} \ \mathcal{H}^1(F \setminus K).$$ We call such a minimizer $E$ a \emph{minimal separating extension} of $K$.
\end{proposition}
\begin{proof}
    Without loss of generality, we prove the statement when $B(x_0,r_0) = B(0,1).$ We further fix $\varepsilon_0 := \beta_K(0,1)$.

    \textit{Step 1: Reduction to Lipschitz curves.} We show that any separating extension contains a separating Lipschitz curve.
    Let $E \subset B(0,1)$ be a separating extension, and suppose  that \EEE $\ell$ is an approximating line in the sense of \eqref{eq_beta} with $E$ contained in a strip of thickness $\eps_0$ about $\ell$, i.e., 
    \begin{equation}\nonumber
        E \subset S_{\ell} := \set{x \in \R^2 \colon \, \mathrm{dist}(x,\ell) \leq \varepsilon_0}.
    \end{equation}
    Let $A_{\ell}$ be the union of two arcs given by
    \begin{equation}\nonumber
        A_{\ell} := S_{\ell} \cap \partial B(0,1).
    \end{equation}
    We now show that there exists a Lipschitz map $f \colon [0,1] \to \overline{B(0,1)}$ with Lipschitz norm controlled by $ C (\HH^{1}(E) + 1)$ such that
    \begin{equation}\label{eqn:LipschitzCurve}
        \text{$\Gamma := f([0,1])$ separates $B(0,1)\quad $ and} \quad A_{\ell} \subset \Gamma \subset E \cup A_{\ell} \subset S_{\ell}. 
    \end{equation}
    We fix two points $p$, $q$ in distinct connected component of $B(0,1) \setminus S_\ell$.
    We observe that the set
    \begin{equation}\nonumber
        \Gamma_0 := E \cup A_{\ell} \cup (\ell \setminus B(0,1))
    \end{equation}
    separates $p$ and $q$ in $\R^2$.
    According to \cite[Theorem 14.3]{Ne}, there is a connected component $\Gamma_1$ of $\Gamma_0$ which separates $p$, $q$ in $\R^2$. 
    {Since $\Gamma_1 \setminus \overline{B(0,1)} = \ell \setminus \overline{B(0,1)}$, we see in particular that the set $\Gamma := \Gamma_1 \cap \overline{B(0,1)}$ is connected.}
    Then by \cite[Theorem 1.8]{DS}, there exists a map $f : [0,1] \to \R^2$ with Lipschitz norm controlled by $ C \HH^1(\Gamma)$ and $\Gamma = f([0,1])$.
    It is clear by construction {that $\Gamma$ separates the points $p,q$ in $B(0,1)$} and that
    \begin{equation}\nonumber
        A_{\ell} \subset \Gamma \subset E \cup A_{\ell} \subset S_{\ell},
    \end{equation}
    concluding the proof of \eqref{eqn:LipschitzCurve}.

    \textit{Step 2: Existence of a minimizer.}
    We consider a minimizing sequence $E_i$ in $\mathcal{C}_{\rm sep}$ with $$\mathcal{H}^1(E_i\setminus K) \to \inf_{E' \in \mathcal{C}_{\rm sep}} \HH^1(E'\setminus K).$$
    By Step 1, for all $i$, there exists a connected, compact set $\Gamma_i$ and a Lipschitz map $f_i \colon [0,1] \to \R^2$ such that (\ref{eqn:LipschitzCurve}) holds with the appropriate replacements (e.g., $\ell$ becomes $\ell_i$). As the sequence $f_i$ is uniformly bounded and uniformly Lipschitz, we may apply the Arzel\`a-Ascoli theorem to extract a subsequence converging uniformly to a function $f \colon [0,1] \to \overline{B(0,1)}$. Similarly, we may assume that the lines $\ell_i$ converge to $\ell$ in an appropriate topology.
    We introduce the set $\Gamma := f([0,1])$. It is clear that $\Gamma_i$ converges to $\Gamma$ in Hausdorff distance and therefore 
    \begin{equation*}
        A_\ell \subset \Gamma \cup K \subset S_\ell.
    \end{equation*}
    We define
    $$E := (\Gamma \cup K) \cap B(0,1),$$ 
    and note that the above subset relation ensures $\beta_E(0,1) = \beta_K(0,1)$.
    To see that $E $ is a separating extension of $K$, it remains to verify the separation property. We consider any   path $\gamma$ in $B(0,1)$ that travels from $p$ to $q$, two points belonging to distinct components of $B(0,1)\setminus S_\ell$. For sufficiently large $i$, $p$ and $q$ belong to distinct components of $B(0,1)\setminus S_{\ell_i}$, and therefore the path $\gamma$ crosses $\Gamma_i$ at some point $f_i(t_i)$, $t_i \in [0,1]$. Using the convergence of $f_i \to f$ and the uniform Lipschitz bound, we may extract a subsequence (not relabeled) such that the points $f_i(t_i)\in \gamma$ converge to a point $f(t) \in \gamma \cap \Gamma$ with $t \in [0,1]$. In particular, as $\gamma\subset B(0,1)$,  we get \EEE $\gamma \cap E \neq \emptyset$, showing that $E$ must in fact separate.

    To show that $E$  is \EEE a minimal separating extension, we outline how to verify that
    \begin{equation}\label{eqn:golabLSC}
        \HH^1(E\setminus K) \leq \liminf_{ i \to \infty} \HH^1(E_i\setminus K).
    \end{equation}
    For this, as $\Gamma$ is connected, we will use a slight modification of Golab's theorem applied to finitely many connected components (we note one could instead show that $\Gamma_i\setminus K$ is uniformly concentrated as in \cite[Theorem 10.14]{Morel-Solimini:2012}). We break each $\Gamma_i  \cap (B(0,1) \setminus K) \EEE $ into its countable connected components $\bigcup_{m}  \Gamma_i^m$.
    {Since $\Gamma_i = f_i([0,1])$ meets $\partial B(0,1)$,  for  each $m$ and each point $y \in \Gamma_i^m$,  there is  a continuous path in $\Gamma_i$ connecting $y$ to $\partial B(0,1)$. Starting from $y$ until the first time it meets $K \cup \partial B(0,1)$, the path must stay within $\Gamma_i^m$. We deduce that if $y \in \Gamma_i^m$ is such that ${\rm dist}(y,K \cup \partial B(0,1))>1/n$ for some integer $n \geq 1$, then we have $\mathcal{H}^1(\Gamma^m_i)> 1/n.$}

    By the naive upper bound $\mathcal{H}^1(E_i)\leq \mathcal{H}^1(K \cup (\mathbb{R}e_1 \cap B(0,1))) \leq \mathcal{H}^1(K)+2$, there are at most $m_n \leq (\mathcal{H}^1(K)+2)n$ components reaching a distance greater than $1/n$ from $K \cup \partial B(0,1)$ (now reordered so these components are first in line).
    These components must converge in Hausdorff distance to a portion of the curve $\Gamma$ in $B(0,1) \setminus K$ given by $\Gamma^{n}$, so we may apply Golab's theorem (see \cite[Theorem 10.19]{Morel-Solimini:2012}) to find
    \begin{equation}\label{eqn:preLSC}
        \HH^1\big((\Gamma^n \cap B(0,1))\setminus K\big) \leq \liminf_{ i \to \infty} \sum_{m\leq m_n}\HH^1(\Gamma_i^m ) \leq \liminf_{ i \to \infty} \HH^1\big((\Gamma_i\cap B(0,1))\setminus K \big).
    \end{equation}
    {For any $y \in (\Gamma\cap B(0,1))\setminus K$, it holds that ${\rm dist}(y , K \cup \partial B(0,1))>1/n$ for some $n\in \mathbb{N}$. We may choose an index $m(y)$ with $m(y)\leq m_n$ such that ${\rm dist }(y,\Gamma_i^{m(y)})\to 0$ as $i\to \infty$. It necessarily follows that $y \in \Gamma^n,$ and consequently, $(\Gamma^n \cap B(0,1))\setminus K$ increases to the set $(\Gamma \cap B(0,1))\setminus K = E\setminus K$ as $n \to \infty$. Passing $n\to \infty$ on the left-hand side of (\ref{eqn:preLSC}), we conclude~(\ref{eqn:golabLSC}).}
\end{proof}

We now show that a minimal separating extension $E$ inherits the Ahlfors-regularity of $K.$

\begin{proof}[Proof of Lemma \ref{lem_AFF}]
    Without loss of generality, we suppose  that \EEE $B(x_0,r_0) = B(0,1).$

    \textit{Step 1: Proof of item (1).}   
    {We let $x \in E \setminus K$ and $r > 0$ be such that $B(x,r) \cap K = \emptyset$ and $B(x,r) \subset B(0,1)$.}
    For $0 < \rho < r$, we let $N(\rho) : = \HH^0\big(E \cap \partial B(x,\rho)\big)$ be the cardinality of $E \cap \partial B(x,\rho).$
    We show that for $0<\rho<r$, it holds that $N(\rho) \geq 2$. With this, (\ref{eq_AFF0})  follows from the coarea formula as
    \begin{equation}\nonumber
        \HH^1\big(E \cap B(x,r)\big) \ge \int_0^r \HH^0\big(E \cap \partial B(x,\rho)\big) \dd{\rho} \geq {2}r.
    \end{equation}
    We proceed by contradiction. Assuming that $N(\rho) \leq 1$ for some $0 < \rho < r$, we show that $E_{\rho}:= E \setminus B(x,\rho)$  is also a separating extension of $K$ in $B(0,1)$, which will contradict the minimality of $E$ as $E$ is coral and thus $\HH^1(E\cap B(x,\rho))>0$.

    {It is  clear that $K \subset E_{\rho}$ since $K \subset E$ and $B(x,\rho)$ is disjoint from $K$.
    We also have $\beta_{E_\rho}(0,1) = \beta_E(0,1)$ by the inclusions $K \subset E_{\rho} \subset E$ and the fact that $\beta_K(0,1) = \beta_E(0,1)$.}
    We now verify the separation property  of $E_\rho$. \EEE Suppose $\ell$ is an approximating line in the sense of \eqref{eq_beta} with $E$ contained in a strip of thickness $\eps_0 : = \beta_E(0,1)$ about $\ell$, i.e., 
    \begin{equation}\nonumber
        E \subset S_{\ell} := \set{x \in \R^2 \colon \, \mathrm{dist}(x,\ell) \leq \varepsilon_0}.
    \end{equation}
    {As $\overline{B(x,\rho)} \subset B(0,1)$, we can find two points $p,q \in B(0,1) \setminus \overline{B(x,\rho)}$ that belong to distinct connected components of $B(0,1) \setminus S_{\ell}$.
        Assuming by contradiction that $E_\rho = E \setminus B(x,\rho)$ did not separate, there is a curve $\Gamma :[0,1] \to B(0,1) \setminus E_{\rho}$ connecting $p$ and $q$. If $\Gamma$ never meets $B(x,\rho)$, it connects $p$ and $q$ in the complement of $E$, which is a contradiction to the fact that $E$ separates.
    Otherwise, we consider the first time $t_1$ and the last time $t_2$ such that $\Gamma$ meets $\partial B(x,\rho)$. Since the intersection $E \cap \partial B(x,\rho)$ consists of at most one point, we can find an arc in $\partial B(x,\rho)$ connecting $\Gamma(t_1)$, $\Gamma(t_2)$ without meeting $E$. Replacing the portion $\Gamma|_{[t_1,t_2]}$ by this arc, we find again that $p$ and $q$ are connected in $B(0,1) \setminus E$.  This gives a   contradiction.}

    \textit{Step 2: Proof of item (2).}  Let us first show \EEE that for all $x \in E$ and for all $r > 0$ such that $B(x,r) \subset B(0,1)$, we have
    \begin{equation}\nonumber
        \HH^1(E \cap B(x,r)) \geq C^{-1} r.
    \end{equation}
    Precisely, if $x \in K$, this follows from the Ahlfors-regularity of $K$. If $x \in E \setminus K$ and $B(x,r/2) \cap K = \emptyset$, this follows from item (1) of the lemma. Otherwise there is $y\in B(x,r/2) \cap K$ with $B(y,r/2)$ contained in $B(x,r)$, and we may once again use the Ahlfors-regularity of $K$.

    We finally prove that for all $x \in E$ and for all $0 < r \leq 1$, we have
    \begin{equation}\nonumber
        \HH^1(E \cap B(x,r)) \leq C r.
    \end{equation}
    As the upper bound is already satisfied for $K$, this just follows from noticing that if in a ball $B(x,r)$, $\HH^1((E \cap B(x,r))\setminus K)\geq 2\pi r$, a better minimal separating extension is given by
    $$ K \cup \EEE \left(E\setminus (B(x,r) \cap S_\ell)\right) \cup \partial(B(x,r) \cap S_\ell),$$
    where $S_\ell$ is defined as in Step 1 of the proof.
\end{proof}

\section{Proof of Lemma \ref{lem_reifenberg}}\label{sec:lemmReif}

{Here we  give \EEE the proof of Lemma \ref{lem_reifenberg}, showing that a set $K$ with small bilateral flatness at all scales is actually separating in a smaller ball. }

\begin{proof}
     We start with a preliminary observation: \EEE    Consider an arbitrary pair of points $x, y \in K \cap B(0, r_0/4)$. For any $r$ such that $\abs{x - y} < r \leq r_0/2$, we have $\beta^{\rm bil}(x,r)\le \eps_0$ and from the definition of $\beta^{\rm bil}(x,r)$ and the fact that $y \in K \cap B(x,r)$, one can {directly} show that for all $x'$ on the segment $[x;y]$ we have
    \begin{equation*}
        \mathrm{dist}(x',K) \leq 2 \varepsilon_0 r.
    \end{equation*}
    Since $r$ can be chosen arbitrary close to $\abs{x - y}$, we actually conclude that 
    \begin{equation}\label{eq: AAAG}
        \mathrm{dist}(x',K) \leq 2 \varepsilon_0 \abs{x - y}\quad \quad \text{for all $x,y \in B(0,r/4)$ and for all $x' \in [x;y]$}.
    \end{equation}
    Consequently, for $z\in K$ minimizing the above distance  for \EEE $x' := (x+y)/2$,  as $\eps_0\le 1/8$,  we get
    \begin{equation}\label{eqn:midpointClose}
        |z - x| \leq \abs{z - \frac{x+y}{2}} + \frac{1}{2} |x - y| \leq \frac{3}{4} \abs{x-y},
    \end{equation}
    and likewise for $|y-z|$. Note that $z$ cannot be too far from the origin because
    \begin{equation*}
        \abs{z} \leq \abs{(x+y)/2} + 2 \varepsilon_0 \abs{x - y} \leq \max(\abs{x},\abs{y}) + 2 \varepsilon_0 \abs{x - y}.
    \end{equation*}

     We now start with the actual proof. \EEE We let $\ell$ be a line passing through $x_0$ for which the infimum of $\beta^{\rm bil}(x_0,r_0/2)$ in \eqref{eq: bilat flat} is attained.
    We let $(e_1,e_2)$ be the canonic basis of $\R^2$ and we assume without loss of generality that $\ell = \mathbb{R}e_1$ and $x_0 = 0$. 
    Suppose by contradiction that $K$ does not separate $B(0,r_0/10)$. By definition of $\ell$  and \EEE  by $\beta^{\rm bil}(x, r_0/2) \le \varepsilon_0$  we have
    \begin{align}
        K \cap B(0,r_0/10) &\subset \set{y \colon \, \mathrm{dist}(y,\ell) \leq \varepsilon_0 r_0/2} \subset \set{y \colon \, \mathrm{dist}(y,\ell) \leq r_0/20}\label{eq_beta100_separation}
    \end{align}
    and since the separation in $B(0,r_0/10)$ does not depend on the chosen line satisfying (\ref{eq_beta100_separation}), we can find a continuous path $\Gamma$ in $B(0,r_0/10)\setminus K$ that joins the two connected components of 
    \begin{equation*}
        \set{y \in B(0,r_0/10) \colon \,  \mathrm{dist}(y,\ell) \geq r_0/20}.
    \end{equation*}
     As \EEE the points on the segment $\ell \cap B(0,r_0/2)$ are at most distance $\varepsilon_0 r_0/2$ away from $K$,  we can fix two points $z_{\rm start}, z_{\rm end} \in K$ such that
    \begin{equation*}
        |z_{\rm start}  + (r_0/8) e_1| \leq \varepsilon_0 r_0/2 \quad \text{and} \quad |z_{\rm end}- (r_0/8) e_1| \leq \varepsilon_0 r_0/2,
    \end{equation*}
    and in particular $z_{\rm start}, z_{\rm end} \in B(0,r_0/6) \setminus B(0,r_0/10)$ and also $\abs{z_{\rm start} - z_{\rm end}} \leq r_0/3$.
    Applying (\ref{eqn:midpointClose}) with $x= z_{\rm start}$ and $y = z_{\rm end}$, we find a point $z_1\in K$ such that
    \begin{equation*}
        |z_i  - z_{i+1}| \leq \frac{3}{4} \abs{z_{\rm start}-z_{\rm end}} \quad \text{for $i=0,1$},
    \end{equation*}
    where we have labeled $z_0 = z_{\rm start}$ and $z_2 = z_{\rm end}$. Moreover, we have
    \begin{equation*}
        \abs{z_1} \leq \frac{r_0}{6} + 2 \varepsilon_0 \abs{z_{\rm start} - z_{\rm end}},
    \end{equation*}
    so $z_1$ stays within $B(0,r_0/4)$. This observation allows  to    repeat this \EEE procedure to the smaller segments $[z_{\rm start},z_1]$ and $[z_1,z_{\rm end}]$.
    Iterating this midpoint construction between adjacent $z_i$'s and relabeling the sequence, we construct $(z_i)_{i = 0}^L\subset K$ with $z_0 = z_{\rm start}$ and $z_L = z_{\rm end}$ such that 
    \begin{equation*}
        |z_i - z_{i+1}|\leq \left(\frac{3}{4}\right)^{L-1} \abs{z_{\rm start}-z_{\rm end}}  \quad \text{for $i=0,\ldots, L-1$}
    \end{equation*}
    and for $i = 0,\ldots,L$,
    \begin{equation*}
        \abs{z_i} \leq \frac{r_0}{6} + 2 \varepsilon_0 \sum_{k=0}^{ L-2} \left(\frac{3}{4}\right)^k \abs{z_{\rm start} - z_{\rm end}}.
    \end{equation*}
    Note that  $z_i \in B(0,r_0/4)$ for all $i$ and all $L$ since $\eps_0 <1/36 $ and 
    \begin{equation*}
        \frac{r_0}{6} + 2 \varepsilon_0 \sum_{k=0}^{+\infty} \left(\frac{3}{4}\right)^k \abs{z_{\rm start} - z_{\rm end}} \leq  \frac{r_0}{6} + 2 \varepsilon_0 \cdot 4 \cdot 2 \cdot \tfrac{1}{6} r_0 \le   \frac{r_0}{6} + 3 \varepsilon_0 r_0.
    \end{equation*}
    Letting $\Gamma_{\rm poly}$ be the polygonal curve connecting adjacent $z_i$'s, it is contained in $B(0,r_0/4)$ and satisfies by \eqref{eq: AAAG}
    \begin{equation*}
        \sup_{x \in \Gamma_{\rm poly}} {\rm dist}(x , K) \leq 2 \eps_0 \left(\frac{3}{4}\right)^{L-1}  \abs{z_{\rm start}-z_{\rm end}},
    \end{equation*}
    and in particular, as $\varepsilon_0 \leq 1/100$,
    \begin{equation*}
        \sup_{x \in \Gamma_{\rm poly}} {\rm dist}(x , \ell) \leq r_0/20.
    \end{equation*}
    Since $\Gamma_{\rm poly}$ runs from the left side of $B(0,r_0/10)$ to the right side of $B(0,r_0/10)$ while being contained in the strip $\lbrace y \colon \,  \mathrm{dist}(y,\ell) \leq r_0/20\rbrace$, it must cross $\Gamma$.
    It follows that
    \begin{equation*}
        \inf_{x \in \Gamma} {\rm dist}(x , K) \leq 2 \eps_0 \left(\frac{3}{4}\right)^{L-1} \abs{z_{\rm start}-z_{\rm end}}.
    \end{equation*}
    As $L$ can be chosen arbitrarily large, this shows that $\inf_{x \in \Gamma} {\rm dist}(x , K) = 0$. Since both $\Gamma$ and $K$ are closed, this implies $K\cap \Gamma \neq \emptyset$, a contradiction to the assumption.
\end{proof}


\typeout{References}




\end{document}